\documentclass[10pt,reqno]{amsart}
\usepackage{graphicx,amsmath,amssymb,amsthm,xcolor,comment,mathrsfs}
\usepackage[margin=1in]{geometry}
\usepackage{hyperref}
\usepackage{tikz}
\usetikzlibrary{decorations.markings}
\usepackage[shortlabels]{enumitem}
\usepackage{stackengine} 
\newcommand\oast{\stackMath\mathbin{\stackinset{c}{0ex}{c}{0ex}{\ast}{\bigcirc}}}

\title{The four-dimensional Anderson model: a case study for critical SPDEs}
\author{Yu Deng}
\address{University of Chicago}
\email{yudeng@uchicago.edu}
\author{Hao Shen}
\address{University of Wisconsin-Madison}
\email{pkushenhao@gmail.com}
\date{}

\theoremstyle{definition}
\newtheorem {definition}{Definition}[section]
\newtheorem{prop}[definition]{Proposition}
\newtheorem{thm}[definition]{Theorem}
\newtheorem{lemma}[definition]{Lemma}
\theoremstyle{remark}
\newtheorem{remark}{Remark}[section]
\numberwithin{equation}{section}

\newcommand{\Eb}{\mathbb E}

\newcommand{\Nb}{\mathbb N}

\newcommand{\Pb}{\mathbb P}

\newcommand{\Rb}{\mathbb R}

\newcommand{\Tb}{\mathbb T}

\newcommand{\Zb}{\mathbb Z}
\newcommand{\Ac}{\mathcal A}
\newcommand{\Bc}{\mathcal B}
\newcommand{\Cc}{\mathcal C}

\newcommand{\Ec}{\mathcal E}

\newcommand{\Hc}{\mathcal H}
\newcommand{\Ic}{\mathcal I}
\newcommand {\Jc}{\mathcal J}

\newcommand{\Lc}{\mathcal L}

\newcommand{\Nc}{\mathcal N}

\newcommand{\Pc}{\mathcal P}
\newcommand{\Qc}{\mathcal Q}
\newcommand{\Rc}{\mathcal R}
\newcommand{\Sc}{\mathcal S}
\newcommand{\Tc}{\mathcal T}

\newcommand{\mf}{\mathfrak m}
\newcommand{\ff}{\mathfrak f}
\newcommand{\pf}{\mathfrak p}

\newcommand{\rf}{\mathfrak r}
\newcommand{\yf}{\mathfrak y}

\newcommand{\lf}{\mathfrak l}

\newcommand{\qf}{\mathfrak q}

\newcommand{\Xf}{\mathfrak X}

\newcommand{\nf}{\mathfrak n}

\newcommand{\Rs}{\mathscr R}

\newcommand{\Xs}{\mathscr X}

\newcommand{\vX}{\boldsymbol{X}}
\newcommand{\vY}{\boldsymbol{Y}}
\newcommand{\vP}{\boldsymbol{P}}

\newcommand{\dirac}{\boldsymbol{\delta}}

\newcommand{\eps}{\varepsilon}
\newcommand{\Wick}[1]{\boldsymbol{:\!} #1 \boldsymbol{\!:}}

\tikzset{
    midarrow/.style={
        postaction={ decorate,  decoration={
                markings,
                mark={at position 0.6 with {\arrow{>}}}
            }}},
    dot/.style={circle, fill=black, inner sep=1pt},
    Bdot/.style={rectangle, fill=black, inner sep=1.5pt}
}

\begin{document}

\begin{abstract}
We study the weakly coupled elliptic Anderson model with spatial white noise on the four-dimensional torus,
which provides a basic example of a critical SPDE requiring
renormalization at arbitrarily high orders.  
With coupling $\lambda |\log\eps|^{-\frac12}$ where $\lambda>0$ is sufficiently small,
we prove that the Green's function of the corresponding random Schr\"odinger operator,  suitably centered and rescaled, converges to a centered Gaussian random field with explicit covariance. 

The main difficulty is that, for such critical models, one must expand up to order
$|\log\varepsilon|$, while the perturbative expansion contains factorially many
pairings and a growing number of renormalization terms.  To overcome this, we construct a truncated renormalized parametrix and prove sharp high-order
bounds for its remainder.  A central ingredient is a multiscale analysis based on a new version of Hepp trees, combined with new estimates for summations over permutations.
These estimates reveal a precise balance between logarithmic losses from scale summation and factorial gains from the structure of primitive pairings.  The methods developed here are intended as a first step toward a general  theory for critical SPDEs with weak couplings.
\end{abstract}

\maketitle

\section{Introduction}\label{intro} The purpose of this paper (and follow-up works) is to develop suitable machinery for the study of renormalization and solution theories for \emph{critical} SPDEs in the perturbative regime, i.e. when the (normalized) coupling constant is sufficiently small.

In this paper we start with a \emph{linear} problem, namely the elliptic Anderson model in dimension $4$. We expect our methods to also apply to nonlinear models (such as $\Phi_4^4$ and Yang-Mills etc.), which will be explored in future works.

\subsection{The main result}\label{sec.intro_1}
Let $\xi$ be the space white noise on $\Tb^4=[-\pi,\pi]^4$, also fix a smooth, compactly supported cutoff function $\rho\geq 0$ on $\Tb^4$, which is symmetric in all variables and even in each variable, with $\int_{\Tb^4}\rho=1$. For $\eps\in(0,1)$, define $\xi_\eps:=\rho_\eps*\xi$ where $\rho_\eps=\eps^{-4}\rho(\cdot/\eps)$. Consider the Schr\"{o}dinger operator
\begin{equation}\label{eq.intro_1}\Lc_{\eps}:=1-\Delta-\lambda_\eps\cdot\xi_{\eps}+\Cc_\eps;\qquad \lambda_\eps:= \frac{\lambda}{\sqrt{|\log\eps|}},\end{equation}where $0<\lambda\ll 1$ is a small absolute constant (independent of $\eps$), and $\Cc_\eps$ is a renormalization constant to be fixed later. Define the Green's function 
\begin{equation}\label{eq.anderson}G_\eps(x,y):=(\Lc_\eps^{-1}\dirac_y)(x)\end{equation} where $\dirac_y$ is the Dirac mass at $y$ (we understand $G_\eps=0$ if $\Lc_\eps$ is not invertible). Define also the free Green's function $G(x,y)=G(x-y)=((1-\Delta)^{-1}\dirac_y)(x)$.

\begin{thm}\label{thm.main} Fix $\lambda$ small enough. Then there exists a suitable choice of renormalization constants $\Cc_\eps$ (depending on $\eps$ and $\rho$), such that the sequence of random fields converges in law \[\lambda_\eps^{-1}\cdot (G_\eps(x,y)-G(x,y))\xrightarrow{\mathrm{law}}\Hc(x,y)\qquad \textrm{as }\eps\to 0,\] and the limit $\Hc(x,y)$ is a centered Gaussian random field with covariance function independent of the cutoff function $\rho$, which is
\begin{equation}\label{eq.cov}\Eb(\Hc(x_1,y_1)\Hc(x_2,y_2))=\frac{2\pi^2}{2\pi^2-\lambda^2}\int_{\Tb^4}G(x_1-z)G(y_1-z)G(x_2-z)G(y_2-z)\,\mathrm{d}z.\end{equation}
\end{thm}
\begin{remark} The $\Cc_\eps$ defined in \eqref{eq.setup_11} is a sum of a {\it diverging number} of renormalization constants.
As far as we know, Theorem \ref{thm.main} is the first convergence result for a critical SPDE where \emph{infinitely many} renormalization terms are needed (see discussions in Section \ref{sec.intro_2}). Previously, there has been a term-by-term convergence result, proved recently in \cite{GR26}. Namely, \cite{GR26} shows that the expansion term of order $n$ converges to a Gaussian field for each fixed $n$ \emph{independent} of $\eps$, and that summing up these limits in $n$ leads to the covariance (\ref{eq.cov}) for $\lambda<\sqrt 2 \pi$.
\end{remark}
\begin{remark} In principle, the parabolic Anderson model
\[\partial_tu=\Delta u+\lambda_\eps\xi_\eps\cdot u-\Cc_\eps'\cdot u\] (where $\Cc_\eps'$ is some renormalization constant possibly different from the $\Cc_\eps$ in Theorem \ref{thm.main}) can be studied using similar arguments as in the proof of Theorem \ref{thm.main}, without additional technical difficulties. We plan to pursue this in a future work.
\end{remark}
\subsection{Critical SPDEs}\label{sec.intro_2} The field of stochastic partial differential equations (SPDEs) has seen tremendous success in the last 15 years. With the development of regularity structure theory \cite{H14}, paracontrolled calculus \cite{GIP15}, and renormalization group approaches \cite{Kup16,D22}, the (local-in-time) solution theory for general \emph{subcritical} SPDEs is now well understood.  (c.f. \cite{HairerKPZ,CatellierChouk,BCCH,HairerGeometric,CCHS_2D,CCHS_3D,HairerQuasi,Otto2024priori,GubiQuasi}.)

The next obvious target is then the study of \emph{critical} equations. Many physically important SPDEs belong to this class, for example the Anderson, $\Phi^4$ and Yang-Mills models in dimension $4$, the Kardar–Parisi–Zhang (KPZ) and stochastic heat equations in dimension $2$, etc. Compared to subcritical ones, they have the following distinctive features, which are also the main sources of difficulties:
\begin{enumerate}[{(1)}]
\item They are scaling invariant, i.e. small and large scales are expected to have the same behavior;
\item High order iterations are expected to be no better (e.g. in terms of regularity) than lower order ones;
\item As such, one usually needs to perform an \emph{infinite order} expansion in constructing solutions;
\item Generally, there will also be \emph{infinitely many} intrinsic renormalization terms;
\item Sometimes there is a critical value of parameter, at which the system exhibits a \emph{phase transition} between Gaussian and non-Gaussian limits.
\end{enumerate}

A typical example, in fact the \emph{major} example of critical SPDE where rigorous results have been proved, is the stochastic heat equation (SHE) in dimension $2$. Consider, for example, the It\^o solution to the following \emph{spatially mollified} SHE:
\begin{equation}\label{eq.she}
    \partial_t u_\eps(t,x)
    =
    \frac12 \Delta u_\eps(t,x)
    +
    \beta_\eps u_\eps(t,x)\xi_\eps(t,x),
    \qquad
    u_\eps(0,x)\equiv 1\ (x\in\mathbb{R}^2).
\end{equation}
Here $\beta_\eps:=\beta\cdot|\log\eps|^{-1/2}$, and $\xi_\eps(t):=\rho_\eps *\xi(t)$ is space mollification of $\xi$, where $\xi$ is space-time white noise and $\rho_\eps(x)=\eps^{-2}\rho(x/\eps)$ with $\rho$ a smooth cutoff.

It turns out that, (\ref{eq.she}) exhibits a phase transition from Gaussian to non-Gaussian, at the \emph{critical value} $\beta=\sqrt{2\pi}$. In fact (see \cite[Theorem~5.2]{CSZ2024review}),  if $\beta<\sqrt{2\pi}$, then for every fixed $t>0$, the sequence of random fields converges in law,
\begin{equation}\label{eq.she_limit}
    \frac{1}{\beta_\eps}\bigl(u_\eps(t,x)-1\bigr)
    \xrightarrow{\mathrm{law}}
     v(t,x)\qquad \textrm{as }\eps\to 0,
\end{equation}
and the limit $v(t,x)$ is a centered Gaussian random field with covariance function
\[
    \Eb(v(t,x_1)v(t,x_2))=\frac{2\pi}{2\pi-\beta^2}\int_0^t\int P_{s}(x_1-y)P_s(x_2-y)\,\mathrm{d}y\mathrm{d}s,
\] where $(P_t)_{t\ge 0}$ is the heat semigroup of $\frac12\Delta$. Equivalently, $v(t,x)$ can be written as the solution to the linear (Edwards--Wilkinson) SPDE\[\partial_t v=
    \frac12 \Delta v+ \bigg(\frac{2\pi}{2\pi-\beta^2}\bigg)^{1/2}\cdot\xi, \qquad v(0,x)\equiv 0.\] See \cite[Theorem 2.17]{CSZ17b} for the proof (in fact, \cite{CSZ17b} also discusses the solution as a space-time field, instead of as a  spatial field at a given $t$). 
Similar Gaussian fluctuation results are also obtained for 2D KPZ 
\cite{CD20,CSZ20,G20} via Cole-Hopf along with other techniques.
Some of these results are also extened to the nonlinear case where $u_\eps \xi_\eps$ in the equation is replaced by $\sigma(u_\eps) \xi_\eps$ \cite{Tao24,DG24}. We remark that pointwise statistics of the model is also of great interest; we refer the review articles \cite{CSZ2024review,CSZ2025ICM} for further literature.

    Note that, the result \eqref{eq.she_limit} is analog of Theorem \ref{thm.main} (but for full range of $\beta$) where the scaling limit is Gaussian. However, things are completely different when $\beta=\sqrt{2\pi}$: the $u_\eps$ itself converges to 
    a highly nontrivial (non-Gaussian) limit called the stochastic heat flow (SHF), as shown in the  breakthrough work \cite{CSZ23a,T24} (the former proves the result in terms of directed polymer measures, whereas the latter proves the result for  SHE as in \eqref{eq.she} using axiomatic characterization). In fact, arbitrary finite moments of the SHF have been studied, and they exhibit a double exponential growth, in contrast with Gaussians (\cite{CSZ19b,GQT21,CSZ23b,GN25}).

    We also mention a related line of work \cite{CES21,CannizzaroDuke,CannizzaroBurgers} on stationary solutions to critical SPDEs which exploits explicit Gaussian invariant measures and the associated Wiener chaos analysis of their generators. For anisotropic KPZ and stochastic Burgers equation, with coupling strength $\frac{\lambda}{\sqrt{|\log\eps|}}$ 
    the corresponding stationary dynamics converge, for {\it every} $\lambda>0$, to a linear stochastic heat equation with nontrivially renormalized coefficients.
    We refer to \cite{CT24} for an overview of these results.  More recently, \cite{CMTsuperdiffusive} treated the 2D stochastic Burgers equation at strong coupling (i.e. without $|\log\eps|^{-\frac12}$),  and proved that, under logarithmically superdiffusive space-time scaling, its solution converges to a linear stochastic heat equation with renormalized coefficients.
    Let us also mention \cite{DubedatShen,Piazza2025} on critical (conformal) SPDEs which have a notion of probabilistically-weak stationary solutions also due to well-understood invariant (symmetrizing) measures.

    \subsubsection{A comparison between two models} We make two remarks regarding the differences between the 4 dimensional Anderson model in Theorem \ref{thm.main} and the 2 dimensional SHE in \cite{CSZ17b,CSZ2024review}.
    
    (1) For SHE (\ref{eq.she}), since we do not smooth out the noise in time and are considering the It\^o solution, it follows from the one-sided support of the heat kernel that, if we formally expand the solution to (\ref{eq.she}) as a superposition of $n$-th order chaos, then
    \begin{enumerate}[{(a)}]
    \item Any self-pairing term (i.e. lower order chaos) in the expansion is exactly $0$;
    \item If we take correlation between two $n$-th order chaos, then there is \emph{only one} pairing between the different noises that gives nonzero contribution. 
    \end{enumerate}
    This means that (\ref{eq.she}) does not require any renormalization terms, and the correlations between $n$-th order chaos grow \emph{exponentially} in $n$, without any potential factorial loss. This then allows one to expand to infinite order for each fixed $\eps$, without encountering factorial divergence, at least when $\beta<\sqrt{2\pi}$.

    Of course, the above is the very special feature for (\ref{eq.she}) due to the one-sided support of heat kernels, which is not true for the Anderson model (\ref{eq.anderson}). In fact, for (\ref{eq.anderson}) there will be infinitely many renormalization terms (or strictly speaking $|\log\eps|\to\infty$ terms). Theorem \ref{thm.main} seems to be the first result ever that rigorously deals with such a scenario.
    
    Moreover, if we expand the solution to (\ref{eq.anderson}) up to $n$-th order chaos, then we inevitably encounter $n!$ type divergence when $n\to\infty$ for fixed $\eps$. As such, we need to perform \emph{finite order} expansions at order $n$ with $n\sim|\log\eps|\to\infty$ in the limit, and obtain very precise estimates for these very high order terms which involve very complicated combinatorial structures. The main novelty of this paper is to develop a whole set of machinery (namely Propositions \ref{prop.hepp_3}, \ref{prop.main_3}, \ref{prop.main2_2} and \ref{prop.main_4}) to estimate such combinatorial structures. We believe such machinery, with suitable generalizations, can be applied to more general linear and nonlinear critical SPDEs (such as $\Phi_4^4$ or Yang-Mills), for which the factorial divergence also needs to be addressed. See Section \ref{sec.idea} for more details.

    (2) For SHE (\ref{eq.she}), the simple structure of correlations due to property (b) above allows to explicitly calculate the contributions of each order $n$. In fact, for $\beta<\sqrt{2\pi}$, these contributions decay exponentially like $(\hat{\beta})^n$ with $\hat{\beta}:=\beta/\sqrt{2\pi}$, so in the end only order $O(1)$ terms contribute; however, for $\beta=\sqrt{2\pi}$, all terms up to order $\sim|\log\eps|$ contribute uniformly, and the contribution of these very high order terms ($\sim|\log\eps|$) is crucial to the non-Gaussianity of the limit.

    For the Anderson model (\ref{eq.anderson}), due to the high complexity of combinatorial structures, it is not clear what is the transition threshold between Gaussian and non-Gaussian (or whether such a threshold exists at all), though a natural guess in view of (\ref{eq.cov}) is $\lambda=\sqrt{2}\pi$. It would be preferable to prove Theorem \ref{thm.main} for all $\lambda<\sqrt{2}\pi$ (instead of sufficiently small $\lambda$), but this seems to require much more precise estimates than in this paper, which is not available at this point.
\subsection{Ideas of the proof}\label{sec.idea} Here we briefly discuss the ideas of the proof of Theorem \ref{thm.main}. We expect that the observations made here should provide general guidelines to solving critical SPDEs. The case of nonlinear equations will involve much more sophisticated combinatorics, which we plan to address in future works.
\subsubsection{Parametrix construction} Start with the formal expansion
\begin{equation}\label{eq.idea_1}(1-\Delta-\lambda_\eps\xi_\eps)^{-1}\sim(1-\Delta)^{-1}+\sum_{n=1}^\infty\big((1-\Delta)^{-1}\circ(\lambda_\eps\xi_\eps)\big)^n\circ(1-\Delta)^{-1}.
\end{equation} The goal is to subtract suitable renormalization terms from (\ref{eq.idea_1}) (and match them by choosing $\Cc_\eps$ in (\ref{eq.intro_1})), and control the rest of the expression (in terms of suitable moments) for each $n$.

To illustrate the ideas, for simplicity
let us ignore the renormalizations as well as all self-pairing terms (i.e. lower order chaos) for now in (\ref{eq.idea_1}). If we take the second moment of the resulting no-self-pairing term, then by Wick's formula we get the expression
\begin{equation}\label{eq.idea_2}\lambda_\eps^{2n}\sum_{\pi\in S_n}\int_{(\Tb^4)^{2n}}\prod_{j=0}^{n}G(x_j-x_{j+1})G(y_j-y_{j+1})\cdot\prod_{j=1}^n\eta_\eps(y_j-x_{\pi(j)})\,\prod_{j=1}^n\mathrm{d}x_j\mathrm{d}y_j,
\end{equation} where $\eta=\rho *\rho$ is another smooth cutoff. The key point of (\ref{eq.idea_2}) is the summation over all permutations $\pi\in S_n$; unlike the case \eqref{eq.she}, all such permutations will have nonzero contributions, which leads to an $n!$ loss when $n$ is large. Because of this, there is no hope to control the terms in (\ref{eq.idea_1}) for \emph{arbitrarily large} $n$ relative to $\eps$, and the series (\ref{eq.idea_1}) written as is will not converge. 

The idea, which goes back to \cite{DH21}, is to construct a \emph{parametrix} of $1-\Delta-\lambda_\eps\xi_\eps$, which will be a \emph{truncated} version of the series (\ref{eq.idea_1}). Say we truncate (\ref{eq.idea_1}) at level $n\leq A$, and define the resulting parametrix by $\Pc$, then we have
\begin{equation}\label{eq.idea_3}(1-\Delta-\lambda_\eps\xi_\eps)\circ\Pc-1=\Rc:=-(\lambda_\eps\xi_\eps)\circ \big((1-\Delta)^{-1}\circ(\lambda_\eps\xi_\eps)\big)^A\circ(1-\Delta)^{-1},
\end{equation} and the same for $\Rc'=\Pc\circ (1-\Delta-\lambda_\eps\xi_\eps)-1$. If we can bound $\|\Rc\|_{\mathrm{op}},\|\Rc'\|_{\mathrm{op}}<1/2$ in a suitable operator norm, then we can invert $1+\Rc$ and $1+\Rc'$ by von Neumann series, which then allows us to construct $(1-\Delta-\lambda_\eps\xi_\eps)^{-1}$. See Section \ref{sec.setup_6} for details.

In order to bound $\Rc$ (and similarly $\Rc'$), note that the factor $\lambda_\eps\xi_\eps$ yields a loss of $\eps^{-2}$ in terms of operator norm (as $\|\xi_\eps\|_{L^\infty}\sim\eps^{-2}$ with high probability). As for the remaining factors $\big((1-\Delta)^{-1}\circ(\lambda_\eps\xi_\eps)\big)^A\circ(1-\Delta)^{-1}$, the best thing we can get is something like
\[\|\big((1-\Delta)^{-1}\circ(\lambda_\eps\xi_\eps)\big)^A\circ(1-\Delta)^{-1}\|_{\mathrm{op}}\leq (C\lambda)^A\] even without accounting for the potential factorial loss, due to the criticality of the system. Therefore, a basic necessary condition for the above parametrix to work is that 
\[(C\lambda)^A\cdot\eps^{-2}\ll 1\Rightarrow A\gtrsim|\log\eps|.\]

On the other hand, consider the integral expression (\ref{eq.idea_2}). To see the picture, let us just integrate in a small domain where all $x_j$ and $y_j$ belong to the same ball of radius $\eps$. This leads to the integral which has size (note that the Green's function $|G(z)|\sim |z|^{-2}$ for small $|z|$)
\[\lambda_\eps^{2n}\cdot n!\cdot (\eps^4)^{2n}\cdot \eps^{-4(n+1)}\cdot \eps^{-4n}\sim \lambda^{2n}\cdot n!\cdot |\log\eps|^{-n}\cdot\eps^{-4}.\] As $\lambda$ is a constant independent of $\eps$, if $n\gg\lambda^{-2}|\log\eps|$, then the factorial loss $n!$ will become dominant, and the whole expression will diverge factorially in $n$; conversely, if $n\lesssim|\log\eps|$ (say with constants independent of $\lambda$), then this contribution will have size $(C\lambda)^{2n}\cdot\eps^{-4}$, which is acceptable. This leads to our first observation, namely

\textbf{Observation 1:} in order to control the solutions to critical SPDEs, we need to perform finite order (truncated) expansions up to order $A\sim|\log\eps|$, where $\eps$ is the ultraviolet scale.

See also \cite{GabrielAC} for similar phenomena in the context of parabolic PDE with critical random initial data.
We remark that this observation is also consistent with the situation of SHE \cite{CSZ17b,CSZ23a}, where terms of order $|\log\eps|$ emerge at the threshold value $\beta=\sqrt{2\pi}$.
\subsubsection{Hepp trees} Now, still ignoring the renormalizations, the goal is to control the terms in (\ref{eq.idea_1}) of order $n\lesssim|\log\eps|$, see (\ref{eq.setup_20}) in Proposition \ref{prop.main_1} (we also need Proposition \ref{prop.main_2}, but this is a term-by-term estimate and follows from \cite{GR26}). This leads to expressions of the form (\ref{eq.idea_2}); for simplicity, we will neglect all scales smaller than $\eps$, and replace $\eta_\eps$ factors by Dirac mass, which reduces (\ref{eq.idea_2}) to an integral of the form
\begin{equation}\label{eq.idea_4}
\Ic\sim\lambda_\eps^{2n}\sum_{\pi\in S_n}\int_{(\Tb^4)^{n}}\prod_{j=0}^{n}(\eps^2+|x_j-x_{j+1}|^2)^{-1}\cdot(\eps^2+|x_{\pi(j)}-x_{\pi(j+1)}|^2)^{-1}\,\prod_{j=1}^n\mathrm{d}x_j.
\end{equation} The main task now is to prove the boundedness of $\Ic$ for $n\lesssim|\log\eps|$; see Proposition \ref{prop.primi_1} for the precise statement of this estimate.

An important tool to deal with such integrals is the \emph{Hepp tree}, see \cite{HairerQuastel,ChandraHairer,HairerBPHZ,GR26}, which serves to distinguish the different scales among the distances between the points $x_j$. Basically, it is constructed by choosing two points among $\{x_j\}$ with smallest distance, then view them as a single point, and find the next smallest distance among the resulting points, and so on. This naturally results in the Hepp tree $\Tc$, which is a binary tree with a dyadic scale $N_\nf$ assigned to each branching node $\nf$, such that the original numbers $x_j$ are represented by the leaves of $\Tc$, and
\begin{equation}\label{eq.idea_5}|x_i-x_j|\sim N_\nf\end{equation} where $\nf$ is the lowest common ancestor of the leaves representing $x_i$ and $x_j$. See Definitions \ref{def.hepp}, \ref{def.hepp_3} and Lemma \ref{lem.hepp} for a version of this adapted to our setting.

In the term-by-term case (where $n$ is fixed independent of $\eps$), once this Hepp tree is fixed, \emph{and the permutation $\pi\in S_n$ is also fixed}, the integrand in (\ref{eq.idea_4}) will be comparable to a fixed power combination of the $N_\nf$ due to (\ref{eq.idea_5}); on the other hand, the volume of the set of points $(x_1,\cdots,x_n)$ described by the given Hepp tree is also bounded by a power combination of $N_\nf$, which allows one to estimate (\ref{eq.idea_4}) by first integrating and then summing in the dyadic $(N_\nf)$. We then only need to count the number of logarithmic losses coming from summing over $(N_\nf)$, which is determined by the combinatorial character of the Hepp tree $\Tc$ together with the permutation $\pi$. Roughly speaking, this is the approach taken in subcritical SPDEs (which have no logarithmic losses) and the term-by-term result \cite{GR26}.

However, in the current setting where $n$ can be as large as $|\log\eps|$, there is no room to fix $\pi$ first and lose dangerous $n!$ factors. Instead, we shall rewrite (\ref{eq.idea_4}) as
\begin{equation}\label{eq.idea_6}
\Ic\sim \lambda^{2n}|\log\eps|^{-n}
\frac{1}{n!}\int_{(\Tb^4)^n}
\bigg(\sum_{\pi\in S_n}\prod_{j=0}^n(\eps^2+|x_{\pi(j)}-x_{\pi(j+1)}|^2)^{-1}\bigg)^2\,\prod_{j=1}^n\mathrm{d}x_j.
\end{equation} Then we control this quantity by two ingredients: (i) the \textbf{volume estimate}, which bounds the volume of the set of points $(x_1,\cdots,x_n)$ such that \emph{a permutation of} $(x_j)$ is described by the given Hepp tree $\Tc$; (ii) the \textbf{integrand estimate}, which bounds the summation in $\pi$ in (\ref{eq.idea_6}) for a fixed point $(x_1,\cdots,x_n)$ described by the given Hepp tree.

The first ingredient, i.e. the volume estimate, is provided by Proposition \ref{prop.hepp_3}. Note that the right hand sides of (\ref{eq.vol_2})--(\ref{eq.vol_2.1}) has only exponential factors $C^n$: this is the only loss allowed, and anything worse, e.g. $(\log\log n)^n$, would lead to divergence. To achieve this, we need to make very sharp estimates, exploiting things including the implicit constants in (\ref{eq.idea_5}) which may depend on the size of the Hepp tree $\Tc$, and the size of the automorphism group of $\Tc$ (see the left hand sides of (\ref{eq.vol_2})--\ref{eq.vol_2.1})). See Section \ref{sec.proof_hepp_3} for the proof.
\subsubsection{Summation over permutations}\label{sec.idea_3} It remains to prove the second ingredient, i.e. the integrand estimate, which is stated in Proposition \ref{prop.main_3}. 

Note that we are estimating a summation over $\pi\in S_n$. Start with the easier term-by-term setting: in this case, once $\Tc$ is fixed, all the permutations $\pi$ can be classified according to an ``index" set $W$. This $W$ is a subset of $\Bc\backslash\{\rf\}$ (the set of non-root branching nodes), which is defined as the set of non-root branching nodes of $\Tc$ such that all leaves in the subtree rooted at this node \emph{form a single block} in the permutation $\pi$. The case $W=\Bc\backslash\{\rf\}$ produces the largest contribution in the $\pi$ summation, and every element $\nf\not\in W$ yields one gain factor of form $N_{\nf}/N_{\nf^+}$ where $\nf^+$ is the parent node of $\nf$ in $\Tc$. Upon multiplying by the volume factor in Proposition \ref{prop.hepp_3} and summing over $(N_\nf)$, each such factor (for each element missed by $W$) will save us one logarithmic loss. Schematically that is:
\begin{equation}\label{eq.idea_7}\sum_{(N_\nf)}(\textrm{contribution to $\Ic$ with fixed $W$})\lesssim_n |\log\eps|^{-n}\cdot |\log\eps|^{|W|+1}=|\log\eps|^{-1}\cdot |\log\eps|^{|W^c|}.
\end{equation} This is an extension of the arguments in \cite{GR26}, with the latter only distinguishing the ``contributing" case $W=\Bc\backslash\{\rf\}$ and the ``non-contributing" case $W\subsetneq\Bc\backslash\{\rf\}$. In fact, it is easy to see that $W=\Bc\backslash\{\rf\}$, i.e. \emph{every} subtree forms a single block in $\pi$, if and only if the associated Feynman diagram is contributing in the sense of \cite{GR26} (i.e. formed by successively inserting bubbles to the trivial diagram).

Now, in the case $n\sim|\log\eps|$, the estimate (\ref{eq.idea_7}) is replaced by the precise sharp estimate
\begin{equation}\label{eq.idea_8}\sum_{(N_\nf)}(\textrm{contribution to $\Ic$ with fixed $W$})\leq (C\lambda)^{2n}|\log\eps|^{-n}\cdot (n-|W|)!|\log\eps|^{|W|+1}.
\end{equation} In other words, the logarithmic loss (worse for larger $W$) is \emph{balanced} by the factorial $(n-|W|)!$ (which becomes better for larger $W$). For example, if $W=\Bc\backslash\{\rf\}$, then we have full logarithmic losses, but \emph{the number of such permutations} is bounded exponentially in $n$ (as observed in \cite{GR26}). On the other hand, if $W=\varnothing$, then we have $n!$ such permutations, but \emph{there is no logarithmic loss}, which is still under control as $n\lesssim|\log\eps|$ and $n!\leq n^n\lesssim C^n|\log\eps|^n$. For the precise statement of (\ref{eq.idea_8}), see (\ref{eq.vol_3}) in Proposition \ref{prop.main_3}. We summarize this as 

\textbf{Observation 2:} when $n\sim|\log\eps|$, there is generally a balance between ``logarithmic loss" and ``factorial loss" which corresponds to a classification of permutations $\pi$ (or choice of pairings, or Feynman diagrams).

Now we proceed to prove (\ref{eq.idea_8}). This is done by an inductive argument in $\Tc$; for the precise statement for induction, see Proposition \ref{prop.main2_2}. In the inductive step, we consider a subtree $\Tc_\nf$ rooted at a branching node $\nf$, such that
\begin{enumerate}[{(a)}]
\item The scales $N_\mf$ for all $\mf\in\Tc_\nf$ are all comparable;
\item For the parent $\nf^+$ of $\nf$ we have $N_{\nf^+}\gg N_\nf$.
\end{enumerate} For the precise condition, see (\ref{eq.gap_4}) in Section \ref{sec.induct}. In the language of \cite{ChandraHairer,HairerBPHZ,GR26}, this $\Tc_\nf$ is an ``unsafe" group while any subtree of $\Tc_\nf$ is a ``safe" group. These conditions guarantee that, when considering distances with leaves outside of $\Tc_\nf$, all points of $\Tc_\nf$ can be viewed as the same point. 

As such, summing over a permutation $\pi$ of all leaves of $\Tc$, can be recast as summing over a permutation $\pi_1$ of all leaves of $\Tc_\nf$, a way of breaking $\pi_1$ into blocks of consecutive elements (defined by a ``skipped set" $O$, see Proposition \ref{prop.main_4}), and a permutation $\pi_2$ of the smaller tree $\Tc'$, where $\Tc'$ is formed by merging $\Tc_\nf$ into a new leaf with different copies (each block of consecutive elements in $\pi_1$ counts as a copy). The set $W$ contains $\nf$ if and only if there is only one block in $\pi_1$. For details, see \textbf{Step 2} in the proof of Proposition \ref{prop.main2_2} in Section \ref{sec.induct}.

In this way, the summation over $\pi$ can be factorized into a summation over $(\pi_1,O)$, and a summation over $\pi_2$. The latter follows from the induction hypothesis for $\Tc'$, and the former requires a standalone estimate, which is Proposition \ref{prop.main_4}.

In Proposition \ref{prop.main_4}, the main feature is that all distance scales $N_\mf$ for $\mf\in\Tc_\nf$ are comparable; for simplicity we will assume they are all equal, i.e. there is only one scale (in reality, we still need to be very careful about the relative multiples for these scales, to avoid any super-exponential loss). A toy model for this case is the estimate (say with $O=W=\varnothing$)
\begin{equation}\label{eq.idea_9}
\sum_{\pi\in S_n}\prod_{j=1}^{n-1}|x_{\pi(j)}-x_{\pi(j+1)}|^{-2}\leq C^n (n!)^{1/2}\tau^{-2(n-1)},
\end{equation} provided that the $x_j$ are different points in $\Tb^4$ that are at least distance $\tau$ apart. Note that if we naively bound $|x_{\pi(j)}-x_{\pi(j+1)}|\geq\tau$ in (\ref{eq.idea_9}) and sum over $\pi$, we would get $n!$ instead of $(n!)^{1/2}$ which is unacceptable (upon squaring in (\ref{eq.idea_6})).

On the other hand, the sharp estimate in (\ref{eq.idea_9}) follows by iterating the local inequality:
\begin{equation}\label{eq.idea_10}\sup_x\sum_{y\in A}\mathbf{1}_{|x-y|\geq\tau}|x-y|^{-2}\lesssim |A|^{1/2}\tau^{-2},\end{equation} where $A$ is a finite set of points that are at least $\tau$ apart. This is because we are in dimension $4$, and the most compact way of packing $|A|$ boxes of size $\tau$ is to form a large box of size $|A|^{1/4}\tau$; see (\ref{eq.lem4_2}) in the proof of Lemma \ref{lem.aux3}.

The proof of general case of Proposition \ref{prop.main_4} is a sophisticated version of (\ref{eq.idea_9}), which requires a number of improvements, such as a bilinear/two-variable version of (\ref{eq.idea_10}) (see (\ref{eq.aux8}) in Lemma \ref{lem.aux3}); also the presence of different copies of the same leaf leads to various multinomial coefficient factors, which need to be controlled using Lemmas \ref{lem.aux1}--\ref{lem.aux2}, see \textbf{Step 3} in the proof of Proposition \ref{prop.main_4} in Section \ref{sec.final}.
\subsubsection{Renormalization terms} Finally, we turn to the renormalization terms. Algebraically, these renormalization constants come from subintervals that are fully paired, and can be constructed following the BPHZ theory \cite{HairerBPHZ} as in \cite{GR26} (though for the linear Anderson model, they can be described by much simpler combinatorics, see Section \ref{sec.setup_3}).

Analytically, we need to prove upper bounds for these terms with order up to $|\log\eps|$, see (\ref{eq.setup_19}) in Proposition \ref{prop.main_1}. For this, note that these terms are summations over self-pairings which we denote by $\kappa$; now if we fix the set of different points described by a Hepp tree $\Tc$ (as we did in Section \ref{sec.idea_3} above), then the summation over $\kappa$ can be viewed as a summation over all permutations of these points but with each point having \emph{two} copies. In this way, the estimate of renormalization constants $\Cc_{2q}$ can fit into the same framework of estimates discussed in Section \ref{sec.idea_3}, and follows from similar proofs. We list this as

\textbf{Observation 3:} in general critical setting, the estimate for renormalization constants should follow from arguments that are similar to those in the estimate for main terms. 

\subsection{Comparison with other critical problems} Outside the realm of SPDEs, the analysis of critical equations has also played central roles in recent works in wave turbulence and kinetic theory, including \cite{DH21,DH21_2,DH23,DH23_2,DHM24}. In these cases, the critical equations still satisfy (2)--(3) in Section \ref{sec.intro_2} above, but the setting of the problem usually specifies \emph{one single scale}, in contrast with the SPDE setting where multiple different scales contribute uniformly.

These problems have many features in common with the current work: one needs to expand up to an order that is logarithmic in the scale (like \textbf{Observation 1}), and there is also a balance between integration/counting bounds and combinatorial/factorial bounds (like \textbf{Observation 2}), though no renormalization is required in dispersive settings. However, the classification of Feynman diagrams leading to \textbf{Observation 2} is different in the sense that one can define an index $\rho$ (analog of $|W^c|$ in Section \ref{sec.idea_3}) such that each index gains a \emph{power} of the scale as opposed to a \emph{logarithm} in this work. Of course this provides more flexibility in the choice of parameters and in some sense makes things easier, though the analytic and combinatorial estimates in \cite{DH21,DH23_2,DHM24} also have their own unique difficulties that are not present here.

It would be interesting to compare the similarities and differences of these problems (for example, there are strikingly similar diagrams and reduction procedures in \cite{DH23} and \cite{GR26} coming from seemingly unrelated places), which we plan to do in future works.

\subsection*{Concerning AI usage} We have used ChatGPT 5.4 Pro at early stages of this project to provide some ideas and inspirations, including e.g. a proof of a toy model version of Proposition \ref{prop.main_4} (which we decide not to include in this paper). The \textbf{Step 3} and \textbf{Step 4} (b) in the proof of Proposition \ref{prop.hepp_3} in Section \ref{sec.proof_hepp_3} are completed with the help of ChatGPT 5.5 Pro. All the other proofs in this paper are composed and written exclusively by the authors.

\subsection*{Acknowledgement} H.S. gratefully acknowledges support from NSF  CAREER DMS-2044415, and some helpful discussion with Simon Gabriel on \cite{GR26}. Y.D. is supported by NSF DMS-2531437. We also thank the SLMath program ``Recent Trends in Stochastic Partial Differential Equations'' in Fall 2025 where this work began.

The proof of Theorem \ref{thm.main} has been autoformalized in Lean 4 (assuming only Proposition \ref{prop.main_2}, which is the main result proved in \cite{GR26}), with the help from Yongzheng Jia and Peter Liang from UC Berkely and SAIR Foundation (\url{https://sair.foundation/}). See the Github page (\url{https://github.com/sairmath/anderson4d_lean}) for details.

\section{Notations and preliminaries}\label{sec.setup} 
We start by setting up some notations and preliminary results.
\subsection{Relevant notations} In the proof below, we understand that intervals $[a,b]$ consist of integers. We also identify a linear operator $K$ with its kernel $K(x,y)$ such that
\begin{equation}\label{eq.setup_0}K(x,y)=(K\dirac_y)(x),\qquad Kf(x)=\int_{\Tb^4}K(x,y)f(y)\,\mathrm{d}y.\end{equation} We shall use $C$ to denote (large) absolute constants that may vary from line to line. At one place we will need to fix one specific such constant, which we will refer to as $C_0$.

\begin{definition}\label{def.func} Define $\Ec$ to be the class of functions $f$ on $\Tb^4$ (or generally $(\Tb^4)^n$) that is symmetric in all variables and even in each variable; i.e.
\begin{equation}\label{eq.func} f(\cdots,x^i,\cdots,x^j,\cdots)=f(\cdots,x^j,\cdots,x^i,\cdots);\qquad f(\cdots,x^i,\cdots)=f(\cdots,-x^i,\cdots).
\end{equation}
\end{definition}
\begin{definition}\label{def.pairing} Given any finite index set $A$, define a \emph{partial pairing} $\kappa$ to be a partition of a subset $P\subseteq A$ into disjoint two-element subsets $\{i,j\}$ (called \emph{pairs}); denote by $S=A\backslash P$ to be the set of unpaired, or \emph{single}, indices. If $S=\varnothing$ (in particular $|A|$ is even), we say $\kappa$ is a \emph{full pairing}.
\end{definition}
\begin{definition}\label{def.interval}Let $A$ be a finite index set, with some linear order $<$. We will call it an \emph{interval} (such as $[1,m]=\{1,\cdots,m\}$). A \emph{subinterval} of $A$ is a subset of form $[a,b]=\{x\in A:a\leq x\leq b\}$. Given a partial pairing $\kappa$ (Definition \ref{def.pairing}), we say a subinterval $B$ is \emph{fully paired} if $\kappa$ induces a full pairing on $B$. Define $\kappa$ to be \emph{primitive} if no \emph{proper} subinterval is fully paired.
\end{definition}
\subsection{The smoothed out noise} Recall the smoothed out noise $\xi_\eps=\rho_\eps*\xi$ with $\rho\in\Ec$ (Definition \ref{def.func}). Define $\eta:=\rho*\rho\geq 0$ which is also smooth and compactly supported with $\eta\in\Ec$ and $\int\eta=1$, then by Wick's theorem we have
\begin{equation}\label{eq.setup_1}
\Eb\bigg(\prod_{j=1}^{m}\xi_\eps(x_j)\bigg)=\sum_{\kappa}\prod_{\{i,j\}\in\kappa}\eta_\eps(x_i-x_j),
\end{equation} where $\eta_\eps=\eps^{-4}\eta(\cdot/\eps)$ as above, and $\kappa$ runs over all full pairings of $[1,m]$. More generally, we have
\begin{equation}\label{eq.setup_2}\prod_{j=1}^{m}\xi_\eps(x_j)=\sum_{\kappa}\prod_{\{i,j\}\in\kappa}\eta_\eps(x_i-x_j) 
\Wick{\prod_{i\in S}\xi_\eps(x_i)}
\end{equation} 
where $\kappa$ runs over all partial pairings of $[1,m]$ and $S$ is the set of single indices (Definition \ref{def.pairing}), 
and $\Wick{-}$ denotes the Wick product, i.e.  projection to the highest Wiener chaos.

\section{Construction of parametrix}\label{sec.setup_new}
\subsection{Renormalized terms}\label{sec.setup_2}
Now consider the formal expansion of the kernel of the unrenormalized operator $(1-\Delta-\lambda_\eps\cdot\xi_\eps)^{-1}$ . For $m\geq 0$, the $m$-th order term (containing $m$ copies of the noise) is given by
\begin{equation}\label{eq.setup_3}
\Ic_m(x,y):=\lambda_\eps^m\int_{(\Tb^4)^m}\prod_{j=0}^m G(x_j-x_{j+1})\prod_{j=1}^m\xi_\eps(x_j)\,\mathrm{d}x_1\cdots\mathrm{d}x_m,
\end{equation} where $x_0:=x$ and $x_{m+1}:=y$. Applying (\ref{eq.setup_2}), we get
\begin{equation}\label{eq.setup_4}\Ic_m=\sum_{\kappa}\Ic_{m,\kappa};\qquad \Ic_{m,\kappa}:=\lambda_\eps^m\int_{(\Tb^4)^m}\prod_{j=0}^m G(x_j-x_{j+1})\prod_{\{i,j\}\in\kappa}\eta_\eps(x_i-x_j)
\Wick{\prod_{i\in S}\xi_\eps(x_i)}\,\mathrm{d}x_1\cdots\mathrm{d}x_m,
\end{equation} where $\kappa$ runs over all partial pairings of $[1,m]$ and $S$ is the set of single indices. 

For later use, we shall define the expression $\Ic_{m,\kappa}[G_0,\cdots,G_m]$ for given functions $G_j\in\Ec$. This is the same as in (\ref{eq.setup_4}), but with the factor $G(x_j-x_{j+1})$ in (\ref{eq.setup_4}) replaced by $G_j(x_j-x_{j+1})$ for $0\leq j\leq m$. In particular, the $\Ic_{m,\kappa}$ in (\ref{eq.setup_4}) is equal to $\Ic_{m,\kappa}[G,\cdots,G]$. Finally, in the above formulas, the index set $[1,m]$ may be replaced by any linearly ordered finite set.

\begin{definition}\label{def.renormterm} Given $m$ and $\kappa$ and  functions $(G_0,\cdots,G_m)$ with $G_j\in\Ec$, we define the renormalization $\Rc\Ic_{m,\kappa}$ of $\Ic_{m,\kappa}$, inductively, as follows:
\begin{enumerate}
\item If there is no fully paired subinterval of $[1,m]$ (and $[1,m]$ itself is not fully paired), then $\Rc\Ic_{m,\kappa}=\Ic_{m,\kappa}$;
\item Otherwise, choose the \emph{smallest (with respect to inclusion) and leftmost} fully paired subinterval $[p+1,p+2q]$ (this could be $[1,m]$). Then:
\begin{enumerate}
\item The indices in $[1,p]\cup[p+2q+1,m]$ are linearly ordered, with a natural partial pairing $\kappa'$, which is inherited from $\kappa$.
\item For $z,w$, denote $x_{p+1}:=z$ and $x_{p+2q}:=w$, and define
\begin{equation}\label{eq.setup_5}J(z-w):=\lambda_\eps^{2q}\int_{(\Tb^4)^{2q-2}}\prod_{j=p+1}^{p+2q-1} G_j(x_j-x_{j+1})\prod_{\{i,j\}\in\kappa}\eta_\eps(x_i-x_j)\,\mathrm{d}x_{p+2}\cdots\mathrm{d}x_{p+2q-1}.
\end{equation}
\end{enumerate}
\item Then, define (inductively)
\begin{align}\label{eq.setup_6}
\Rc\Ic_{m,\kappa}[G_0,\cdots,G_m]&:=\Rc\Ic_{m-2q,\kappa'}[G_0,\cdots,G_{p-1},\widetilde{G},G_{p+2q+1},\cdots G_m],\\
\label{eq.setup_7}
\widetilde{G}(x-y)&:=\int_{(\Tb^4)^2}G_p(x-z)\bigg(J(z-w)-\int_{\Tb^4}J \cdot \dirac(z-w)\bigg)G_{p+2q}(w-y)\,\mathrm{d}z\mathrm{d}w,
\end{align} with $J$ defined in (\ref{eq.setup_5}). Here it is easy to check that $\widetilde{G}\in\Ec$.
\end{enumerate}
\end{definition}
\begin{prop}\label{prop.renorm2} Given $\kappa$, we work as in Definition \ref{def.renormterm}, each time choosing a smallest and leftmost fully paired subinterval and reducing $\kappa$ to $\kappa'$ by removing that subinterval. In the end we get a set of fully paired intervals, whose \emph{left} and \emph{right} endpoints are uniquely determined indices in $[1,m]$, call them $\ell_i$ and $r_i$ for $1\leq i\leq s$. Then we have
\begin{multline}\label{eq.setup_7.5}
\Rc\Ic_{m,\kappa}=\lambda_\eps^m\int_{(\Tb^4)^m}\prod_{\substack{j=0\\j\not\in\{r_1,\cdots,r_s\}}}^m G(x_j-x_{j+1})\prod_{i=1}^s \big(G(x_{r_i}-x_{r_i+1})-G(x_{\ell_i}-x_{r_i+1})\big)\\\times\prod_{\{i,j\}\in\kappa}\eta_\eps(x_i-x_j)\Wick{\prod_{i\in S}\xi_\eps(x_i)}\,\mathrm{d}x_1\cdots\mathrm{d}x_m.
\end{multline}
\end{prop}
\begin{proof} Since $\Rc\Ic_{m,\kappa}$ is defined inductively, we consider the first step of reducing $\Ic_{m,\kappa}$ to $\Ic_{m-2q,\kappa'}$ as in Definition \ref{def.renormterm} (3). In (\ref{eq.setup_7}) we have
\begin{equation}
\label{eq.setup_new1}
\begin{aligned}
\widetilde{G}(x-y)&=\int_{(\Tb^4)^2}G_p(x-z)\bigg(J(z-w)-\int_{\Tb^4}J \cdot \dirac(z-w)\bigg)G_{p+2q}(w-y)\,\mathrm{d}z\mathrm{d}w\\
&=\int_{(\Tb^4)^2}G_p(x-z)J(z-w)\cdot\big(G_{p+2q}(w-y)-G_{p+2q}(z-y)\big)\,\mathrm{d}z\mathrm{d}w.
\end{aligned}
\end{equation}

If we consider $\Ic_{m-2q,\kappa'}[\cdots,\widetilde{G},\cdots]$ in (\ref{eq.setup_6}) and plug in the $\widetilde{G}$ defined by (\ref{eq.setup_new1}), then we get exactly $\Ic_{m,\kappa}$ but \emph{with the factor $G_{r}(x_r-x_{r+1})$ replaced by the difference $G_{r}(x_r-x_{r+1})-G_r(x_\ell-x_{r+1})$}, where $\ell=p+1$ and $r=p+2q$ are the left and right endpoints of the fully paired interval $[p+1,p+2q]$ chosen in Definition \ref{def.renormterm} (2).

As such, we know that $\Ic_{m-2q,\kappa'}$ equals the same expression $\Ic_{m,\kappa}$, but with $G_{r}(x_r-x_{r+1})$ replaced by $G_{r}(x_r-x_{r+1})-G_r(x_\ell-x_{r+1})$. We then proceed by renormalizing $\Ic_{m-2q,\kappa'}$ and so on, where each time we take one new difference; in the end this leads to (\ref{eq.setup_7.5}).
\end{proof}

\subsection{Renormalized equation}\label{sec.setup_3} By discussions in Section \ref{sec.setup_2}, if $m=2q$ is even and all indices in $[1,m]$ are fully paired under $\kappa$, then $\Ic_{m,\kappa}=\Ic_{m,\kappa}(x-y)$ is a deterministic function of $x-y$, and so is its renormalization $\Rc\Ic_{m,\kappa}$. Now, for each $m=2q$, define the renormalization constant $\Cc_{2q}$ as follows.

Let $\sigma$ be a full pairing of $\{1,\cdots,2q\}$, which cannot be formed by concatenating two fully paired subintervals. Equivalently, $\sigma$ is formed by starting with a \emph{primitive} pairing $(2p,\kappa_0)$ and inserting into it finitely many disjoint fully paired subintervals $(2p_i,\kappa_i)$ (these $\kappa_i$ may or may not be primitive). We then define
\begin{equation}\label{eq.setup_8} \Jc_{2q,\sigma}(z-w):=\lambda_{\eps}^{2p}\int_{(\Tb^4)^{2p-2}}\prod_{j=1}^{2p-1}G_j(x_j-x_{j+1})\prod_{\{i,j\}\in\kappa_0}\eta_\eps(x_i-x_j)\,\mathrm{d}x_2\cdots\mathrm{d}x_{2p-1},
\end{equation} where $x_1=z$ and $x_{2p}=w$, and 
\begin{equation}\label{eq.setup_9}
G_j=
\left\{
\begin{aligned}
&\Rc\Ic_{2p_i,\kappa_i},&&\textrm{if the pairing $\kappa_i$ is inserted between $j$ and $j+1$;}\\
&G,&&\textrm{otherwise.}
\end{aligned}
\right.
\end{equation} Now define
\begin{equation}\label{eq.setup_10}
\Cc_{2q,\sigma}:=\int_{\Tb^4}\Jc_{2q,\sigma}(z)\,\mathrm{d}z,\qquad \Cc_{2q}=\sum_{\sigma}\Cc_{2q,\sigma}.
\end{equation}

Finally, let $A=\lfloor|\log\eps|\rfloor$, we define the renormalization constant $\Cc_\eps$ in (\ref{eq.intro_1}) as
\begin{equation}\label{eq.setup_11}
\Cc_\eps:=\sum_{q\leq A}\Cc_{2q}.
\end{equation}
\begin{prop}\label{prop.renorm_3} Given $\sigma$ as in (\ref{eq.setup_8}). Consider the process in Definition \ref{def.renormterm} and the set of fully paired \emph{proper} subintervals obtained (i.e. not including $[1,2q]$ itself), with left and right endpoints $\ell_i$ and $r_i$ for $1\leq i\leq s$. Then for the $\Jc_{2q,\sigma}$ in (\ref{eq.setup_8}), we have
\begin{multline}\label{eq.setup_11.5}
\Jc_{2q,\sigma}(z-w)=\lambda_{\eps}^{2q}\int_{(\Tb^4)^{2q-2}}\prod_{\substack{j=1\\j\not\in\{r_1,\cdots,r_s\}}}^{2q-1}G(x_j-x_{j+1})\prod_{i=1}^s\big(G(x_{r_i}-x_{r_i+1})-G(x_{\ell_i}-x_{r_i+1})\big)\\\times\prod_{\{i,j\}\in\sigma}\eta_\eps(x_i-x_j)\,\mathrm{d}x_2\cdots\mathrm{d}x_{2q-1},
\end{multline} where $x_1:=z$ and $x_{2q}:=w$.
\end{prop}
\begin{proof} Note that $\sigma$ is formed by inserting all the $(2p_i,\kappa_i)$ to $(2p,\kappa_0)$, and $\kappa_0$ is primitive. Consider the collection of intervals $[\ell_i,r_i]$ in (\ref{eq.setup_11.5}), denote it by $A$. For each $(2p_i,\kappa_i)$ inserted, consider the expression $\Rc\Ic_{2p_i,\kappa_i}$ in (\ref{eq.setup_7.5}), which contains an associated collection of intervals $[\ell_*,r_*]$ (with subscripts $i$ replaced by $*$ to avoid repetitions); denote it by $A_i$. It is clear that $A=\sqcup_i A_i$, and that $\ell_i,r_i\not\in\{1,2q\}$. Now, the desired equality (\ref{eq.setup_11.5}) follows by plugging in (\ref{eq.setup_8}), then (\ref{eq.setup_9}) for each $G_j$, and then (\ref{eq.setup_7.5}) for each $\Rc\Ic_{2p_i,\kappa_i}$.
\end{proof}
\subsection{Parametrix of the renormalized operator}\label{sec.setup_4} Now we consider the operator 
\begin{equation}\label{eq.setup_12}\Lc_\eps=1-\Delta-\lambda_{\eps}\cdot\xi_\eps+\Cc_\eps
\end{equation}in (\ref{eq.intro_1}), with $\Cc_\eps$ defined in (\ref{eq.setup_11}). We will construct a parametrix $\Pc_\eps$, and associated error terms $\Rc_\eps$ and $\Rc_\eps'$, such that
\begin{equation}\label{eq.setup_13}\Lc_\eps \Pc_\eps=1+\Rc_\eps,\quad \Pc_\eps \Lc_\eps=1+\Rc_\eps';\qquad \|\Rc_\eps\|,\,\|\Rc_\eps'\|\ll 1
\end{equation} in a suitable space with high probability, which then allows us to define \[\Lc_\eps^{-1}=\Pc_\eps(1+\Rc_\eps)^{-1}=(1+\Rc_\eps')^{-1}\Pc_\eps.\]

We define the parametrix $\Pc_\eps$ as follows:
\begin{equation}\label{eq.setup_14}\Pc_\eps=\sum_{m=0}^A\Pc_m,\quad \Pc_m=\sum_{\kappa}\Rc\Ic_{m,\kappa}.
\end{equation} To calculate $\Lc_\eps\Pc_\eps$ and $\Pc_\eps\Lc_\eps$, we need the following proposition.
\begin{prop}\label{prop.renorm3}For any $m\geq 1$, we have
\begin{equation}\label{eq.setup_15}\lambda_\eps\xi_\eps\circ \Pc_{m-1}=(1-\Delta)\circ\Pc_m+\sum_{1\leq q\leq m/2}\Cc_{2q}\Pc_{m-2q};
\end{equation}
\begin{equation}\label{eq.setup_16}\Pc_{m-1}\circ \lambda_\eps\xi_\eps=\Pc_m\circ (1-\Delta)+\sum_{1\leq q\leq m/2}\Cc_{2q}\Pc_{m-2q}.
\end{equation}
\end{prop}
\begin{proof} We first prove (\ref{eq.setup_15}). We multiply both sides of (\ref{eq.setup_15}) by $(1-\Delta)^{-1}$ on the left, and expand the kernel of $(1-\Delta)^{-1}\circ\lambda_\eps\xi_\eps\circ \Pc_{m-1}$:
\[(1-\Delta)^{-1}\circ\lambda_\eps\xi_\eps\circ \Pc_{m-1}(x,y)=\lambda_\eps\sum_{\kappa}\int_{\Tb^4}G(x-x_1)\xi_\eps(x_1)\cdot\Rc\Ic_{m-1,\kappa}(x_1,y)\,\mathrm{d}x_1,\] where the index set in $\Rc\Ic_{m-1,\kappa}$ is $[2,m]$, and $\kappa$ runs over all partial pairings of this set. By multiplying $\xi_\eps(x_1)$ with $\Wick{\prod_{j\in S}\xi_\eps(x_j)}$ (where $S$ is the set of single indices under $\kappa$, see (\ref{eq.setup_4})) and applying Wick's theorem, we can divide this term into contributions consisting of the following 3 cases:
\begin{enumerate}
\item If $1$ is not paired, i.e. a single index. In this way we get exactly $\Rc\Ic_{m,\kappa}$, where $\kappa$ directly extends to a partial pairing of $[1,m]$ (with $1$ being a single index). Conversely, all partial pairings $\kappa$ of $[1,m]$ where $1$ is a single index can be uniquely obtained in this way.
\item If $1$ is paired (extending $\kappa$ to a partial pairing $\kappa'$ of $[1,m]$), and no subinterval $[1,2q]$ becomes fully paired in $\kappa'$. In this case the new pair involving $1$ does not lead to any additional renormalization, so this term is still exactly $\Rc\Ic_{m,\kappa'}$. Conversely, all partial pairings $\kappa'$ of $[1,m]$ where $1$ is paired but does not form any fully paired subinterval $[1,2q]$, can be uniquely obtained in this way.
\item If $1$ is paired (extending $\kappa$ to a partial pairing $\kappa'$ of $[1,m]$), and there is a fully paired subinterval $[1,2q]$, we choose a smallest $q$, which is then uniquely determined. Now let $\sigma_1$ be the set of pairs within $[1,2q]$, and $\sigma_2$ be those in $[2q+1,m]$, then by the smallest assumption of $q$, we know that $\sigma_1$ cannot be formed by concatenating two fully paired intervals.

This means, see Section \ref{sec.setup_3}, that $\sigma_1$ is formed by starting with a \emph{primitive} pairing $(2p,\kappa_0)$ and inserting into it finitely many disjoint (primitive or non-primitive) fully paired subintervals $(2p_i,\kappa_i)$. The pair involving $1$ must belong to $\kappa_0$, so all the $\kappa_i$ already occur in the partial pairing $\kappa$ of $[2,m]$.

Recall $\Jc_{2q,\sigma_1}$ as in (\ref{eq.setup_8}), then the contribution of this case to $(1-\Delta)^{-1}\circ\lambda_\eps\xi_\eps\circ \Pc_{m-1}$ is equal to
\begin{equation}\label{eq.setup_16.5}\int_{(\Tb^4)^2} G(x-x_1)\Jc_{2q,\sigma_1}(x_1-x_{2q})\Rc\Ic_{m-2q,\sigma_2}(x_{2q},y)\,\mathrm{d}x_1\mathrm{d}x_{2q}.
\end{equation} Recall that \[\int_{\Tb^4}\Jc_{2q,\sigma_1}=\Cc_{2q,\sigma_1},\qquad\Cc_{2q}=\sum_{\sigma_1}\Cc_{2q,\sigma_1},\] we can decompose
\begin{equation}\label{eq.setup_16.8}\Jc_{2q,\sigma_1}(x_1-x_{2q})=\big(\Jc_{2q,\sigma_1}(x_1-x_{2q})-\Cc_{2q,\sigma_1}\cdot\dirac(x_1-x_{2q})\big)+\big(\Cc_{2q,\sigma_1}\cdot\dirac(x_1-x_{2q})\big).\end{equation} If we plug in the first term in (\ref{eq.setup_16.8}), then the resulting contribution to (\ref{eq.setup_16.5}) is exactly equal to $\Rc\Ic_{m,\kappa'}$. Conversely, all partial pairings $\kappa'$ of $[1,m]$ where $[1,2q]$ is the first fully paired interval can be uniquely obtained in this way.

If we plug in the second term in (\ref{eq.setup_16.8}), then the resulting contribution to (\ref{eq.setup_16.5}) is
\[\Cc_{2q,\sigma_1}\cdot\int_{\Tb^4} G(x-x_1)\Rc\Ic_{m-2q,\sigma_2}(x_1,y)\,\mathrm{d}x_1=((1-\Delta)^{-1}\circ (\Cc_{2q,\sigma_1}\cdot\Rc\Ic_{m-2q,\sigma_2}))(x,y).\] 
Clearly summing over all $\kappa$ (and all possible pairs involving $1$) is equivalent to summing over all choices of $q$ and $(\sigma_1,\sigma_2)$, and vice versa. This contribution then gives 
\[\sum_q\sum_{\sigma_1,\sigma_2}(1-\Delta)^{-1}\circ(\Cc_{2q,\sigma_1}\cdot\Rc\Ic_{m-2q,\sigma_2})=\sum_q(1-\Delta)^{-1}\circ(\Cc_{2q}\cdot\Pc_{m-2q}).\]
\end{enumerate}
Putting together (1)--(3) above, we get
\[(1-\Delta)^{-1}\circ\lambda_\eps\xi_\eps\circ \Pc_{m-1}=\Pc_m+\sum_{1\leq q\leq m/2}(1-\Delta)^{-1}\circ(\Cc_{2q}\Pc_{m-2q});\] multiplying both sides by $(1-\Delta)$ on the left, we get
\[\lambda_\eps\xi_\eps\circ \Pc_{m-1}=(1-\Delta)\Pc_m+\sum_{1\leq q\leq m/2}\Cc_{2q}\Pc_{m-2q},\] which proves (\ref{eq.setup_15}). As for (\ref{eq.setup_16}), note that the formulas (\ref{eq.setup_5})--(\ref{eq.setup_7}) in Definition \ref{def.renormterm} and (\ref{eq.setup_8})--(\ref{eq.setup_11}) in Section \ref{sec.setup_3} are symmetric with respect to ``left" and ``right"; in Definition \ref{def.renormterm} (2) we are choosing the leftmost subinterval, but the final expression would not change if each time we choose the rightmost one. Therefore (\ref{eq.setup_16}) follows from the same proof with ``left" and ``right" switched.
\end{proof}
Now, using Proposition \ref{prop.renorm3}, we can calculate $\Lc_\eps\Pc_\eps$ and $\Pc_\eps\Lc_\eps$. By symmetry, we only calculate the former. Note that $\Pc_0=(1-\Delta)^{-1}$, and using (\ref{eq.setup_15}), we can expand that
\begin{equation}\label{eq.setup_17}
\begin{aligned}
\Lc_\eps\Pc_\eps&=\bigg(1-\Delta-\lambda_\eps\xi_\eps+\sum_{q\leq A}\Cc_{2q}\bigg)\circ\bigg(\sum_{m=0}^A\Pc_m\bigg)\\
&=(1-\Delta)\Pc_0+\sum_{m=1}^A(1-\Delta)\circ\Pc_m-\sum_{m=0}^A\lambda_\eps\xi_\eps\circ \Pc_m+\sum_{m,q\leq A}\Cc_{2q}\Pc_m\\
&=1+\sum_{m=1}^A \lambda_\eps\xi_\eps\circ \Pc_{m-1}-\sum_{m=1}^A\sum_{1\leq q\leq m/2}\Cc_{2q}\Pc_{m-2q}-\sum_{m=0}^A\lambda_\eps\xi_\eps\circ \Pc_m+\sum_{m,q\leq A}\Cc_{2q}\Pc_m\\
&=1-\lambda_\eps\xi_\eps\circ\Pc_A+\sum_{\substack{m,q\leq A\\m+2q>A}}\Cc_{2q}\Pc_m.
\end{aligned}
\end{equation} Therefore we get
\begin{equation}\label{eq.setup_18}\Rc_{\eps}=-\lambda_\eps\xi_\eps\circ\Pc_A+\sum_{\substack{m,q\leq A\\m+2q>A}}\Cc_{2q}\Pc_m;\qquad\textrm{similarly}\qquad \Rc_{\eps}'=-\Pc_A\circ\lambda_\eps\xi_\eps+\sum_{\substack{m,q\leq A\\m+2q>A}}\Cc_{2q}\Pc_m.
\end{equation}
\subsection{Main estimates}\label{sec.setup_6}
In this subsection, we state the main estimates for $\Pc_m$ and $\Cc_{2q}$, namely Propositions \ref{prop.main_1} and \ref{prop.main_2}, which we will prove in later sections. Then, we show that these estimates are enough to prove Theorem \ref{thm.main}.
\begin{prop}\label{prop.main_1} For each $1\leq q\leq A$, we have
\begin{equation}\label{eq.setup_19}
|\Cc_{2q}|\leq \eps^{-2}|\log\eps|^{-1}\cdot(C\lambda)^{2q}.
\end{equation}Moreover, define the Fourier coefficients
\begin{equation}\label{eq.ft}\widehat{F}(\alpha,\beta)=\int_{(\Tb^4)^2}e^{i(\alpha\cdot x+\beta\cdot y)}F(x,y)\,\mathrm{d}x\mathrm{d}y,\quad\alpha,\beta\in\Zb^4.
\end{equation} Then for each $1\leq m\leq A$ and each $(\alpha,\beta)$, we have
\begin{equation}\label{eq.setup_20}\Eb\big(|\widehat{\Pc_m}(\alpha,\beta)|^2\big)\leq \lambda_\eps^2\cdot C(C\lambda)^{2m-2}\cdot\min\big(1,\eps^{-8}\langle \alpha\rangle^{-4}\langle \beta\rangle^{-4}\langle \eps^2(\alpha+\beta)\rangle^{-8}\big).
\end{equation}
\end{prop}
\begin{prop}\label{prop.main_2} Define
\begin{equation}\label{eq.setup_21}H(x_1,y_1,x_2,y_2):=\int_{\Tb^4}G(x_1-z)G(y_1-z)G(x_2-z)G(y_2-z)\,\mathrm{d}z.
\end{equation} Also fix $B$ and $r$. Then, for each $1\leq m_j\leq B$ and $(\alpha_j,\beta_j)_{j\leq r}$, we have 
\begin{equation}\label{eq.setup_22}\bigg|\Eb\bigg(\prod_{j=1}^r\widehat{\Pc_{m_j}}(\alpha_j,\beta_j)\bigg)-\sum_{\kappa}\prod_{\{i,j\}\in\kappa}\Eb\big(\widehat{\Pc_{m_i}}(\alpha_i,\beta_i)\cdot\widehat{\Pc_{m_j}}(\alpha_j,\beta_j)\big)\bigg|\leq\frac{(\lambda_\eps)^r}{|\log\eps|}\cdot O_{B,r,\alpha_j,\beta_j}(1),
\end{equation} where $\kappa$ in (\ref{eq.setup_22}) runs over all full pairings of $[1,r]$.

Moreover, there exists a function $\Xf(m_1,m_2)$, which is bounded by $|\Xf|\leq C(C\lambda)^{m_1+m_2-2}$, such that for any $m_1,m_2\leq B$ and $(\alpha_1,\beta_1,\alpha_2,\beta_2)$, we have
\begin{gather}\label{eq.setup_23}
\big|\Eb\big(\widehat{\Pc_{m_1}}(\alpha_1,\beta_1)\cdot\widehat{\Pc_{m_2}}(\alpha_2,\beta_2)\big)-\lambda_\eps^2\cdot\Xf(m_1,m_2)\cdot\widehat{H}(\alpha_1,\beta_1,\alpha_2,\beta_2)\big|
\leq\frac{\lambda_\eps^2}{|\log\eps|}\cdot O_{B,\alpha_j,\beta_j}(1);\\
\label{eq.setup_23.5}
\sum_{m_1+m_2=m}\Xf(m_1,m_2)=\left\{
\begin{aligned}
&(\lambda/\sqrt{2}\pi)^{m-2},&&\textrm{for $m$ even},\\
&0,&&\textrm{for $m$ odd}.
\end{aligned}
\right.
\end{gather}
\end{prop}
\begin{proof}[Proof of Theorem \ref{thm.main} assuming Propositions \ref{prop.main_1}--\ref{prop.main_2}] We now prove Theorem \ref{thm.main}, under the assumptions of Propositions \ref{prop.main_1}--\ref{prop.main_2}. That is, we need to prove
\begin{equation}\label{eq.setup_24}\Hc_\eps(x,y):=\lambda_\eps^{-1}\cdot (G_\eps(x,y)-G(x,y))\xrightarrow{\mathrm{law}}\Hc(x,y)\end{equation} in the sense of convergence in law for random fields. The proof below is divided into 3 steps.

\textbf{Step 1.} First recall the operators $\Pc_\eps$ and $\Rc_{\eps},\Rc_{\eps}'$ (and associated kernels) as in (\ref{eq.setup_14}) and (\ref{eq.setup_18}). We will measure them in $L^2\to L^2$ norm. As smoothed out white noise, we have \[
\Eb\|\xi_\eps\|_{L^2\to L^2}^2=\Eb\|\xi_\eps\|_{L^\infty}^2\lesssim \eps^{-5}.\] On the other hand, by Plancherel we have
\begin{equation}\label{eq.setup_24.5}\|K\|_{L^2\to L^2}^2\sim \|\widehat{K}(\alpha,\beta)\|_{\ell^2\to\ell^2}^2\lesssim \eps^{-8}\sum_{\alpha,\beta}|\widehat{K}(\alpha,\beta)|^2\cdot \langle \eps^2(\alpha+\beta)\rangle^{6}.
\end{equation} Combining this with (\ref{eq.setup_20}) (using the $\eps^{-8}\langle \alpha\rangle^{-4}\langle \beta\rangle^{-4}\langle \eps^2(\alpha+\beta)\rangle^{-8}$ upper bound), we get 
\begin{equation}\label{eq.setup_24.8}\Eb\|\Pc_m\|_{L^2\to L^2}^2\leq (C\lambda)^{2m}\eps^{-20},\quad \forall 1\leq m\leq A.
\end{equation} Note that $A\sim |\log\eps|$, so $\lambda^m\ll\eps^{100}$ for any $m\geq A/10$, as $\lambda$ is chosen sufficiently small.

Putting these together and using (\ref{eq.setup_19}), it is easy to see that
\begin{equation}\label{eq.setup_24.9}\Eb\|\Pc_\eps\|_{L^2\to L^2}\leq\eps^{-12},\qquad \Eb\|\Rc_\eps\|_{L^2\to L^2}+\Eb\|\Rc_\eps'\|_{L^2\to L^2}\leq\eps^{30}.
\end{equation} By Chebyshev, after excluding a set $Z_\eps$ of probability $\Pb(Z_\eps)\leq \eps^{2}$, we have \[\|\Pc_\eps\|_{L^2\to L^2}\leq \eps^{-14},\quad\|\Rc_\eps\|_{L^2\to L^2}+\|\Rc_\eps'\|_{L^2\to L^2}\leq\eps^{28}.\] Then, outside of $Z_\eps$, the inverse $\Lc_\eps^{-1}=\Pc_\eps(1+\Rc_\eps)^{-1}$ will exist in $\Bc(L^2,L^2)$, and satisfies that \begin{equation}\label{eq.setup_24new}\|\Lc_\eps^{-1}-\Pc_\eps\|_{L^2\to L^2}\leq\|\Pc_\eps\|_{L^2\to L^2}\cdot \|\Rc_\eps\|_{L^2\to L^2}\cdot\|(1+\Rc_\eps)^{-1}\|_{L^2\to L^2} \leq\varepsilon^{12}.\end{equation} 

\textbf{Step 2.} Since $\Pb(Z_\eps)\to 0$, in order to prove convergenece in law for $\Hc_\eps$, it suffices to prove the same for $\Hc_\eps\cdot\mathbf{1}_{Z_\eps^c}$. Note that $\Pc_0=G$; define $\Qc_\eps:=\lambda_\eps^{-1}\cdot(\Pc_\eps-\Pc_0)=\lambda_\eps^{-1}\sum_{1\leq m\leq A}\Pc_m$, then by (\ref{eq.setup_24new}) we have that $\|\Hc_\eps\cdot\mathbf{1}_{Z_\eps^c}-\Qc_\eps\cdot\mathbf{1}_{Z_\eps^c}\|_{L^2\to L^2}\leq\eps^{10}$. Using (\ref{eq.setup_20}) and summing over $m$, it is easy to see that $\Eb|\widehat{\Qc_\eps}(\alpha,\beta)|$ is uniformly bounded in $\eps$ for each fixed $(\alpha,\beta)$. This implies that the set of laws of distribution of the random fields $\Qc_\eps\cdot\mathbf{1}_{Z_\eps^c}$ is tight, and the same for $\Hc_\eps\cdot\mathbf{1}_{Z_\eps^c}$ (as their distance belongs to a fixed compact set in $\Sc'(\Tb^4)$). By Prokhorov, we then only need to prove convergence in law for any finitely many Fourier modes of $\Hc_\eps\cdot\mathbf{1}_{Z_\eps^c}$; that is, for any fixed $\alpha_j,\beta_j\in\Zb^4$ and $c_j\in\Rb$ for $1\leq j\leq s$, we only need to show
\begin{equation}\label{eq.setup_25}\Eb(e^{i\Xs(\Hc_\eps\cdot\mathbf{1}_{Z_\eps^c})})\to\Eb(e^{i\Xs(\Hc)}),\qquad \Xs(F):=\sum_{j=1}^s c_j\widehat{F}(\alpha_j,\beta_j).
\end{equation} Since $\|\Hc_\eps-\Qc_\eps\|_{L^2\to L^2}\to 0$ in the support of $\mathbf{1}_{Z_\eps^c}$ (and $|\Xs(F)|\lesssim\|F\|_{L^2\to L^2}$), and also $\Pb(Z_\eps)\to 0$, we only need to prove (\ref{eq.setup_25}) for $e^{i\Xs(\Qc_\eps)}$ instead of $e^{i\Xs(\Hc_\eps\cdot\mathbf{1}_{Z_\eps^c})}$.

\textbf{Step 3.} For each large constant $B$ (independent of $\eps$), define $\Qc_{B,\eps}:=\lambda_\eps^{-1}\sum_{1\leq m\leq B}\Pc_{m}$ (actually these $\Pc_m$ also depend on $\eps$, but this does not matter in the proof below). By using (\ref{eq.setup_20}) as above, we can estimate 
\begin{equation}\label{eq.setup_26}
\begin{aligned}
\Eb\big|(e^{i\Xs(\Qc_\eps)}-e^{i\Xs(\Qc_{B,\eps})})\big|&\leq \Eb|\Xs(\Qc_{\eps})-\Xs(\Qc_{B,\eps})|\\&\leq \sum_{j=1}^s|c_j|\cdot\Eb\big|\widehat{\Qc_{\eps}}(\alpha_j,\beta_j)-\widehat{\Qc_{B,\eps}}(\alpha_j,\beta_j)\big|
\leq O_{c_j,\alpha_j,\beta_j}(1)\cdot(C\lambda)^{B}.
\end{aligned}
\end{equation} 

Moreover, for fixed $B$, define $\Hc_B$ to be a centered Gaussian random field with covariant function 
\[\Eb(\Hc_B(x_1,y_1)\Hc_B(x_2,y_2))=\Xf_B\cdot H(x_1,y_2,x_2,y_2);\qquad \Xf_B:=\sum_{m_1,m_2\leq B}\Xf(m_1,m_2).\] By using (\ref{eq.setup_22})--(\ref{eq.setup_23}), we get that
\begin{equation}\label{eq.setup_27}\Eb((\Xs(\Qc_{B,\eps}))^2)\to\Xf_B\cdot\sum_{i,j=1}^s c_ic_j\widehat{H}(\alpha_i,\beta_i,\alpha_j,\beta_j)=\Eb((\Xs(\Hc_B))^2)\qquad (\eps\to 0)
\end{equation} (which in particular implies that $\Xf_B\geq 0$), as well as 
\begin{equation}\label{eq.setup_28}\Eb((\Xs(\Qc_{B,\eps}))^r)\to \mathbf{1}_{\textrm{$r$ is even}}\cdot(r-1)!!\cdot(\Eb((\Xs(\Hc_B))^2))^{r/2}=\Eb((\Xs(\Hc_B))^r)
\end{equation} for each fixed $r$. For each fixed $B$, as $\Xs(\Hc_B)$ is a Gaussian random variable, we know that $\Xs(\Qc_{B,\eps})\to\Xs(\Hc_B)$ in law as $\eps\to 0$. This implies that
\begin{equation}\label{eq.setup_29}\big|\Eb(e^{i\Xs(\Qc_{B,\eps})})-\Eb(e^{i\Xs(\Hc_B)})\big|\leq \Rs_B(\eps)\end{equation}
with $\lim_{\eps\to 0}\Rs_B(\eps)=0$ for fixed $B$. 

On the other hand, by (\ref{eq.setup_23.5}) we have $\Xf_B\to \Xf:=2\pi^2/(2\pi^2-\lambda^2)$ as $B\to\infty$, which implies that $\big|\Eb(e^{i\Xs(\Hc_B)})-\Eb(e^{i\Xs(\Hc)})\big|\leq \Rs(B)$ with $\lim_{B\to\infty}\Rs(B)=0$. Putting together, we get
\begin{equation}\label{eq.setup_30}\big|\Eb(e^{i\Xs(\Qc_\eps)})-\Eb(e^{i\Xs(\Hc)})\big|\leq O_{c_j,\alpha_j,\beta_j}(1)\cdot(C\lambda)^{B}+\Rs_B(\eps)+\Rs(B).
\end{equation} By taking $\varepsilon\to 0$ first and then $B\to\infty$, this proves (\ref{eq.setup_25}) for $\Qc_\eps$. Therefore, Theorem \ref{thm.main} is proved.
\end{proof}
\subsection{Proof of Proposition \ref{prop.main_2}}\label{sec.bb} It turns out that, Proposition \ref{prop.main_2} directly follows from the proofs in \cite{GR26}. In fact, as \cite{GR26} treats the term-by-term convergence, all the estimates in \cite{GR26} are allowed to have constants of arbitrary size depending on the order of expansion. In the case of Proposition \ref{prop.main_2}, the situation is the same in view of the $O_{B,r,\alpha_j,\beta_j}(1)$ and $O_{B,\alpha_j,\beta_j}(1)$ factors on the right hand side of (\ref{eq.setup_22})--(\ref{eq.setup_23}) (with $e^{i\sum_j(\alpha_j\cdot x_j+\beta_j\cdot y_j)}$ playing the role of the test function $\varphi$ in \cite{GR26}). With this point, we then see that (\ref{eq.setup_22})--(\ref{eq.setup_23}) are consequences of Lemma 3.3, Lemma 3.4, Propositions 3.9 and 3.12 in \cite{GR26}.
\section{Reduction to primitive pairing estimates}\label{sec.primi} In the rest of this paper we will prove Proposition \ref{prop.main_1}. In this section we reduce this estimate to one for primitive pairings (Definition \ref{def.interval}), stated in Proposition \ref{prop.primi_1} below.

Recall $G=G(x-y)$ is the kernel of $(1-\Delta)^{-1}$, which has the following properties for $z\in\Tb^4$:
\begin{equation}\label{eq.primi_0}G\in\Ec,\quad|\nabla^k G(z)|\lesssim |z|^{-2-k},\quad \nabla^k\bigg(G(z)-\frac{1}{4\pi^2|z|^2}\bigg)\lesssim |z|^{-1-k}.
\end{equation}
\begin{prop}\label{prop.primi_1} Fix $1\leq n\lesssim|\log\eps|$, and functions $G_j(z)\in\Ec$ satisfying $|G_j(z)|\leq |z|^{-2}$ for $1\leq j\leq 2n-1$. Consider the following expression
\begin{equation}\label{eq.primi_1}\Jc_{2n,\mathrm{prim}}(z-w):=\lambda_\eps^{2n}\sum_{\kappa}\int_{(\Tb^4)^{2n-2}}\prod_{j=1}^{2n-1}G_j(x_j-x_{j+1})\prod_{\{i,j\}\in\kappa}\eta_\eps(x_i-x_j)\,\mathrm{d}x_2\cdots\mathrm{d}x_{2n-1},
\end{equation} where $x_1:=z$ and $x_{2n}:=w$, and $\kappa$ runs over all \emph{primitive} full pairings of $[1,2n]$. Then the function $\Jc_{2n,\mathrm{prim}}(z)\in\Ec$, and satisfies the following estimate
\begin{equation}\label{eq.primi_2}|\Jc_{2n,\mathrm{prim}}(z)|\leq(C\lambda)^{2n}\bigg(\frac{\eps^{-4}}{|\log\eps|}\cdot|z|^{-2}\mathbf{1}_{|z|\lesssim\eps}+\frac{1}{|\log\eps|^2}\cdot(|z|^2+\eps^2)^{-3}\bigg).
\end{equation}
Moreover, if we change $\Jc_{2n,\mathrm{prim}}$ into $\widetilde{\Jc}_{2n,\mathrm{prim}}$ by inserting an additional factor $\eps^2+\max_{i,j}|x_i-x_j|^2$ in the integrand in (\ref{eq.primi_1}), then we have
\begin{equation}\label{eq.primi_2.5}|\widetilde{\Jc}_{2n,\mathrm{prim}}(z)|\leq(C\lambda)^{2n}\bigg(\frac{\eps^{-2}}{|\log\eps|}\cdot|z|^{-2}\mathbf{1}_{|z|\lesssim\eps}+\frac{1}{|\log\eps|^2}\cdot(|z|^2+\eps^2)^{-2}\bigg).
\end{equation}
\end{prop}
Now we apply Proposition \ref{prop.primi_1} to finish the proof of Propositions \ref{prop.main_1}. In all arguments in the rest of this section, we will assume Proposition \ref{prop.primi_1}, in particular (\ref{eq.primi_2}) and (\ref{eq.primi_2.5}), to be true.

\subsection{Proof of (\ref{eq.setup_19})}\label{sec.4.1} 
We start by proving (\ref{eq.setup_19}). Recall $\Cc_{2q}=\sum_{\sigma}\Cc_{2q,\sigma}$, where $\sigma$ is as in (\ref{eq.setup_8}), and $\Cc_{2q,\sigma}$ is defined in (\ref{eq.setup_10}) with $\Jc_{2q,\sigma}$ given by (\ref{eq.setup_11.5}), which we copy here: 
\begin{multline}\label{eq.primi_3}
\Jc_{2q,\sigma}(z-w)=\lambda_{\eps}^{2q}\int_{(\Tb^4)^{2q-2}}\prod_{\substack{j=1\\j\not\in\{r_1,\cdots,r_s\}}}^{2q-1}G(x_j-x_{j+1})\prod_{i=1}^s\big(G(x_{r_i}-x_{r_i+1})-G(x_{\ell_i}-x_{r_i+1})\big)\\\times\prod_{\{i,j\}\in\sigma}\eta_\eps(x_i-x_j)\,\mathrm{d}x_2\cdots\mathrm{d}x_{2q-1}.
\end{multline}

\textbf{Step 1.} To start, we first locate the endpoints $(\ell_i,r_i)$ of the fully paired subintervals in (\ref{eq.primi_3}). First, fixing $s$ and the set $\cup_i\{\ell_i,r_i\}$ leads to at most $\sum_s\binom{2q}{2s}\leq 2^{2q}$ choices.
Once this is fixed, note that all the subintervals $[\ell_i,r_i]$ have a nested structure (any two of them are either disjoint, or one contains the other), so if we put a left bracket ``\{" at each $\ell_i$ and right bracket ``\}" at each $r_i$, then they form a \emph{Dyck word} (i.e. a word in which left and right brackets are fully paired), which has $\frac{1}{s+1}\binom{2s}{s}\leq 2^{2q}$ choices. 

Once this Dyck word is also fixed, then the set $\{\{\ell_i,r_i\}_{1\leq i\leq s}\}$ is also fixed. By the symmetry of (\ref{eq.primi_3}), we may fix $\{\ell_i,r_i\}$ for each $1\leq i\leq s$, by each time choosing the smallest and leftmost subinterval. For each $i$, define $r_i-\ell_i+1:=2p_i$, and 
\[I_i:=[\ell_i,r_i],\qquad \widetilde{I}_i:=I_i\backslash\cup\{I_j:I_j\subsetneq I_i\},\qquad \widetilde{I}_0:=[1,2q]\backslash(\cup_i I_i).\] Note that each $\widetilde{I}_i$ is a linearly ordered set, and $\{1,2q\}\subseteq\widetilde{I}_0$. Then, choosing a full pairing $\sigma$ in (\ref{eq.primi_3}) with all the $(\ell_i,r_i)$ fixed, is equivalent to choosing a primitive pairing $\kappa_i$ for each $\widetilde{I}_i$ (including $i=0$; beware that these $\kappa_i$ have different meanings from those in Section \ref{sec.setup_3}). As such, with all the $(\ell_i,r_i)$ fixed, we can rewrite (\ref{eq.primi_3}) as follows:
\begin{multline}\label{eq.primi_4}\sum_{\sigma:\,(\ell_i,r_i)\,\textrm{fixed}}\Jc_{2q,\sigma}(z-w)=\lambda_\eps^{2q}\sum_{(\kappa_0,\cdots,\kappa_s)}\int_{(\Tb^4)^{2q-2}}\prod_{j\in\widetilde{I}_0\backslash\{1,2q\}}\,\mathrm{d}x_j\cdot\prod_{i=1}^s\prod_{j\in \widetilde{I}_i}\,\mathrm{d}x_j\\
\times \prod_{\substack{j=1\\j\not\in\{r_1,\cdots,r_s\}}}^{2q-1}G(x_j-x_{j+1})\prod_{i=1}^s\big(G(x_{r_i}-x_{r_i+1})-G(x_{\ell_i}-x_{r_i+1})\big)\prod_{\{i,j\}\in\sigma}\eta_\eps(x_i-x_j),
\end{multline} here note that $\sigma=\kappa_0\cup\kappa_1\cup\cdots\cup\kappa_s$.

\textbf{Step 2}. Now, to estimate (\ref{eq.primi_4}), we proceed in the increasing order of $i$ (but with $i=0$ left to the end), and each time we sum over $\kappa_i$ and integrate in the variables $x_j$ for $j\in \widetilde{I}_i$ (or for $j\in \widetilde{I}_0\backslash\{1,2q\}$ when $i=0$). By our ordering of the intervals, it is easy to check that any factor $G(x_j-x_{j+1})$ where $j<\ell_i-1$ or $j>r_i$, and any factor $G(x_{r_{i'}}-x_{r_{i'}+1})-G(x_{\ell_{i'}}-x_{r_{i'}+1})$ where $i'>i$, \emph{does not involve any variable $x_j$ for $j\in I_i$}. So, if we start with $i=1$, then the part of expression involving the integrated variables $x_j\,(j\in \widetilde{I}_1=I_1)$ will be exactly
\begin{multline}\label{eq.primi_5}\widetilde{G}(x_{\ell_1-1}-x_{r_1+1}):=\lambda_\eps^{2p_1}\sum_{\kappa_1}\int_{(\Tb^4)^{2p_1}}\prod_{j=\ell_1-1}^{r_1-1}G(x_j-x_{j+1})\cdot\big(G(x_{r_1}-x_{r_1+1})-G(x_{\ell_1}-x_{r_1+1})\big)\\\times\prod_{\{i,j\}\in\kappa_1}\eta_\eps(x_i-x_j)\,\mathrm{d}x_{\ell_1}\cdots \mathrm{d}x_{r_1},
\end{multline} where $\kappa_1$ runs over all primitive pairings of $[\ell_1,r_1]$. By comparing (\ref{eq.primi_5}) with (\ref{eq.primi_1}), we see that
\begin{equation}\label{eq.primi_6}
\widetilde{G}(x-y)=\int_{(\Tb^4)^2} G(x-z)\cdot\Jc_{2p_1,\mathrm{prim}}(z-w)\cdot\big(G(z-y)-G(w-y)\big)\,\mathrm{d}z\mathrm{d}w,
\end{equation} where $(x,y,z,w):=(x_{\ell_1-1},x_{r_1+1},x_{\ell_1},x_{r_1})$ and $\Jc_{2p_1,\mathrm{prim}}$ is as in (\ref{eq.primi_1}) with all $G_j$ replaced by $G$. By Proposition \ref{prop.primi_1} we know that $\Jc_{2p_1,\mathrm{prim}}\in\Ec$, and so does $\widetilde{G}$ by (\ref{eq.primi_6}).

Next we find an upper bound for $\widetilde{G}$. There are three cases:
\begin{enumerate}
\item If $|z-w|\ll|z-y|\sim|w-y|$, then by Taylor expansion and (\ref{eq.primi_0}) we have
\begin{equation}\label{eq.primi_7}G(z-y)-G(w-y)=-(z-w)\cdot\nabla G(z-y)+O(1)|z-w|^2\cdot |z-y|^{-4}.\end{equation} As $\Jc=\Jc_{2p_1,\mathrm{prim}}\in\Ec$ (thus is even), we have $\int_{\Tb^4}\Jc(z-w)\cdot (z-w)\,\mathrm{d}w=0$ for each $z$, thus the contribution of the first term in (\ref{eq.primi_7}) vanishes. Plugging in the second term, and using (\ref{eq.primi_2}), we then get 
\begin{multline}\label{eq.primi_8}
|\widetilde{G}(x-y)|\lesssim (C\lambda)^{2p_1}\int_{|z-w|\lesssim|z-y|\sim|w-y|}|x-z|^{-2}|z-y|^{-4}\cdot|z-w|^2\\\times\bigg(\frac{\eps^{-4}}{|\log\eps|}\cdot|z-w|^{-2}\mathbf{1}_{|z-w|\lesssim\eps}+\frac{1}{|\log\eps|^2}\cdot(|z-w|^2+\eps^2)^{-3}\bigg)\,\mathrm{d}z\mathrm{d}w.
\end{multline} By first integrating in $w$ and then in $z$, and using the denominators $1/|\log\eps|$ and $1/|\log\eps|^2$ to cancel possible log losses, we can verify that $|\widetilde{G}(x-y)|\leq (C\lambda)^{2p_1} |x-y|^{-2}$.
\item If $|z-y|\lesssim|z-w|\sim|w-y|$, then we have
\begin{multline}\label{eq.primi_9}
|\widetilde{G}(x-y)|\lesssim (C\lambda)^{2p_1}\int_{|z-y|\lesssim|z-w|\sim|w-y|}|x-z|^{-2}|z-y|^{-2}\\\times\bigg(\frac{\eps^{-4}}{|\log\eps|}\cdot|z-w|^{-2}\mathbf{1}_{|z-w|\lesssim\eps}+\frac{1}{|\log\eps|^2}\cdot(|z-w|^2+\eps^2)^{-3}\bigg)\,\mathrm{d}z\mathrm{d}w.
\end{multline} Again, by integrating in $w$ and then in $z$, and exploiting the denominators, we get that $|\widetilde{G}(x-y)|\leq (C\lambda)^{2p_1}|x-y|^{-2}$.
\item If $|w-y|\lesssim|z-w|\sim|z-y|$, then we have
\begin{multline}\label{eq.primi_10}
|\widetilde{G}(x-y)|\lesssim (C\lambda)^{2p_1}\int_{|z-y|\lesssim|z-w|\sim|w-y|}|x-z|^{-2}|w-y|^{-2}\\\times\bigg(\frac{\eps^{-4}}{|\log\eps|}\cdot|z-w|^{-2}\mathbf{1}_{|z-w|\lesssim\eps}+\frac{1}{|\log\eps|^2}\cdot(|z-w|^2+\eps^2)^{-3}\bigg)\,\mathrm{d}z\mathrm{d}w.
\end{multline} Using that $|z-w|\sim|z-y|$, we can replace the $|z-w|$ in the second line of (\ref{eq.primi_10}) by $|z-y|$; then we first integrate in $w$ and then in $z$, and exploit the denominators, to see that $|\widetilde{G}(x-y)|\leq (C\lambda)^{2p_1} |x-y|^{-2}$.
\end{enumerate}

In summary, in any case we have proved $|\widetilde{G}(z)|\leq (C\lambda)^{2p_1}|z|^{-2}$.

\textbf{Step 3.} Now, we apply the expression of $\widetilde{G}$ from (\ref{eq.primi_5}) to simplify (\ref{eq.primi_4}). Analytically, this gets rid of the summation in $\kappa_1$ (i.e. the primitive pairing in $I_1=\widetilde{I}_1$) and the integration in the corresponding $x_j$ variables. Combinatorially, this amounts to removing the fully paired subinterval $I_1=[\ell_1,r_1]$. After this, the remaining indices (in $[1,\ell_1-1]\cup[r_1+1,2q]$) are still linearly ordered with a full pairing $\sigma':=\sigma\backslash\kappa_1=\kappa_0\cup\kappa_2\cup\cdots\cup\kappa_s$, and the corresponding fully paired subintervals are just $I_i\backslash I_1\,(i\geq 2)$, which we view as the new $I_i$. The definition of $\widetilde{I}_i$ and $\widetilde{I}_0$ remain the same, and now $I_2$ is the new smallest and leftmost interval, so $\widetilde{I}_2=I_2$.

Now let $H=(C\lambda)^{-2p_1}\cdot\widetilde{G}$, also recall $\sigma'=\kappa_0\cup\kappa_2\cup\cdots\cup\kappa_s$, then the above discussions lead to
\begin{multline}\label{eq.primi_11}\sum_{\sigma:\,(\ell_i,r_i)\,\textrm{fixed}}\Jc_{2q,\sigma}(z-w)=(C\lambda)^{2p_1}\lambda_\eps^{2(q-p_1)}\sum_{(\kappa_0,\kappa_2,\cdots,\kappa_s)}\int_{(\Tb^4)^{2(q-p_1)-2}}\prod_{j\in\widetilde{I}_0\backslash\{1,2q\}}\,\mathrm{d}x_j\cdot\prod_{i=2}^s\prod_{j\in \widetilde{I}_i}\,\mathrm{d}x_j\\
\times \prod_{\substack{j=1,\,j\not\in[\ell_1-1,r_1]\\j\not\in\{r_2,\cdots,r_s\}}}^{2q-1}G(x_j-x_{j+1})\cdot H(x_{\ell_1-1}-x_{r_1+1})\cdot\prod_{i=2}^s\big(G(x_{r_i}-x_{r_i+1})-G(x_{\ell_i}-x_{r_i+1})\big)\prod_{\{i,j\}\in\sigma'}\eta_\eps(x_i-x_j),
\end{multline}

The expression (\ref{eq.primi_11}) has the same form as the right hand side of (\ref{eq.primi_4}), except that we do not have the summation and integration associated with $I_1$, and that the function $G(x_{\ell_1-1}-x_{r_1+1})$ is replaced by $H(x_{\ell_1-1}-x_{r_1+1})$. But we also have $H\in\Ec$ and $|H(z)|\leq |z|^{-2}$ up to an absolute constant, and Proposition \ref{prop.primi_1} depends only on these two properties, so next we can sum and integrate in $\kappa_2$ and $x_j\,(j\in \widetilde{I}_2=I_2)$, extract an expression of form (\ref{eq.primi_5}), estimate it using Proposition \ref{prop.primi_1}, and reduce to the same expression (\ref{eq.primi_11}) but without $I_2$, and so on.

\textbf{Step 4.} At the end of reduction, we have removed all the $I_i\,(1\leq i\leq s)$, and now we have $I_0=\widetilde{I}_0$ which is a linearly ordered set (generalized interval) of $2p$ elements, with a primitive pairing $\kappa_0$. This gives that
\begin{equation}\label{eq.primi_12}\sum_{\sigma:\,(\ell_i,r_i)\,\textrm{fixed}}\Jc_{2q,\sigma}(z-w)=(C\lambda)^{2(q-p)}\lambda_\eps^{2p}\sum_{\kappa_0}\int_{(\Tb^4)^{2p-2}} \prod_{j\in I_0\backslash\{2q\}}G_j(x_j-x_{j+1})\prod_{\{i,j\}\in\kappa_0}\eta_\eps(x_i-x_j)\prod_{j\in I_0\backslash\{1,2q\}}\,\mathrm{d}x_i.
\end{equation} Here note that $\{r_1,\cdots,r_s\}$ have been removed (and the associated $G(\cdot)-G(\cdot)$ factors integrated out) in previous steps; moreover each $G_j$ is either $G$ (if no fully paired subinterval is inserted between $j$ and $j+1$) or one of the $H$ functions coming from the reduction steps in \textbf{Steps 2--3}. In either case they belong to $\Ec$ and are bounded by $|H(z)|\leq |z|^{-2}$ up to an absolute constant, so by using (\ref{eq.primi_2}) one last time, we get
\begin{equation}\label{eq.primi_13}
\bigg|\sum_{\sigma:\,(\ell_i,r_i)\,\textrm{fixed}}\Jc_{2q,\sigma}(z)\bigg|\leq (C\lambda)^{2q}\bigg(\frac{\eps^{-4}}{|\log\eps|}\cdot|z|^{-2}\mathbf{1}_{|z|\lesssim\eps}+\frac{1}{|\log\eps|^2}\cdot(|z|^2+\eps^2)^{-3}\bigg).
\end{equation} By integrating this, and then summing over all $(\ell_i,r_i)$ (which provides at most $O(C^q)$ factor, see \textbf{Step 1}), we finish the proof of (\ref{eq.setup_19}).

\subsection{Proof of (\ref{eq.setup_20})}\label{sec.proof.3.24} There are two things to prove: when we take $1$ on the right hand side of (\ref{eq.setup_20}), or when we take $\eps^{-8}\langle \alpha\rangle^{-4}\langle \beta\rangle^{-4}\langle \eps^2(\alpha+\beta)\rangle^{-8}$. We shall prove the former in \textbf{Steps 1--3} below, and the latter, which only requires minor changes, in \textbf{Step 4}.

Start with the following expression, which follows from (\ref{eq.setup_7.5}), (\ref{eq.setup_14}) and (\ref{eq.ft}):
\begin{multline}\label{eq.Pm-hat}
\widehat{\Pc_{m}}(\alpha,\beta)
=\lambda_\eps^m\sum_\kappa\int_{(\Tb^4)^{m+2}}e^{i(\alpha\cdot x+\beta\cdot y)}\prod_{\substack{j=0\\j\not\in\{r_1,\cdots,r_s\}}}^mG(x_j-x_{j+1})\\
\times\prod_{i=1}^s\big(G(x_{r_i}-x_{r_i+1})-G(x_{\ell_i}-x_{r_i+1})\big)\prod_{\{i,j\}\in\kappa}\eta_\eps(x_i-x_j)
\Wick{\prod_{i\in S}\xi_\eps(x_i)}\,\mathrm{d}x_1\cdots\mathrm{d}x_m\mathrm{d}x\mathrm{d}y.
\end{multline} Here in (\ref{eq.Pm-hat}) we denote $(x_0,x_{m+1}):=(x,y)$, and $\kappa$ runs over all partial pairings of $[1,m]$, and $S$ is the set of single indices under $\kappa$.

\textbf{Step 1.} First consider the terms where $\kappa$ in (\ref{eq.Pm-hat}) is a full pairing (so $S=\varnothing$), which are deterministic terms. For each full pairing $\kappa$, consider the intervals $[\ell_i,r_i]$, and the one among them that is maximal (with respect to inclusion) and rightmost; since $\kappa$ is a full pairing, this must have form $I_*=[a,m]$ for some $a=a(\kappa)\in [1,m-1]$ (of course we may fix $a$ below).

As in Section~\ref{sec.4.1} above, we can reduce the integral (\ref{eq.Pm-hat}) by successively removing the fully paired subintervals $[\ell_i,r_i]$ to the left of $I_*$, and then those inside $I_*$, until we are only left with $I_*$ itself. Each such reduction, using Proposition \ref{prop.primi_1}, gains a power of $C\lambda$ and replaces the contribution of $[\ell_i,r_i]$ by a new input function $H$ satisfying $|H(z)| \lesssim |z|^{-2}$, see (\ref{eq.primi_11}). After these reductions, we are left with the new interval $I_*\backslash\{I_j:I_j\subsetneq I_*\}$, (which we assume has $2p$ elements) and a primitive full pairing $\kappa_*$. So we can write
\begin{equation}\label{eq.4.2after-red}
\begin{aligned}
(\textrm{This part of $\widehat{\Pc_m}(\alpha,\beta)$})&=\sum_{\kappa_*} (C\lambda)^{m-2p} \lambda_\eps^{2p} \int_{(\Tb^4)^{2p+2}}e^{i(\alpha\cdot x+\beta\cdot y)}
G_0 (x-x_a) \prod_{j=1}^{2p-1} G_j(u_j-u_{j+1}) 
\\
&\times\prod_{\{i,j\}\in \kappa_*} \eta_\eps(u_i-u_j)\cdot\big(
G(x_m-y) - G(x_a-y)
\big)\,\mathrm{d}u_1\cdots\mathrm{d}u_{2p}\mathrm{d}x \mathrm{d}y
\\
&=(C\lambda)^{m-2p}
\int_{(\Tb^4)^{4}}e^{i(\alpha\cdot x+\beta\cdot y)}
G_0 (x-x_a) \Jc_{2p,\mathrm{prim}}(x_a-x_m)\\&\times
\big(
G(x_m-y) - G(x_a-y)
\big)\, \mathrm{d}x_a\mathrm{d}x_m\mathrm{d}x \mathrm{d}y
\end{aligned}
\end{equation}
with $(u_1,u_{2p}):=(x_a,x_m)$ (see the notations of Proposition \ref{prop.primi_1}), where $\kappa_*$ runs over primitive pairings, and $|G_j(z)|\lesssim|z|^{-2}$.
Now the integral in $y$ produces
\[
\int_{\Tb^4} e^{i\beta\cdot y}
\bigl(G(x_m-y)-G(x_a-y)\bigr)\,dy
=
\langle \beta\rangle^{-2}\cdot e^{i\beta\cdot x_a}
\big(e^{i\beta\cdot (x_m-x_a)}-1\big),
\]
noticing that $\widehat{G}(\beta)=\langle\beta\rangle^{-2}$ up to constants. If we take the imaginary part $\sin(\beta\cdot (x_m-x_a))$ above, then the contribution to the whole integral vanishes because $\Jc_{2p,\mathrm{prim}}\in \Ec$, so we are left with the contribution of the real part, which can be bounded by
\[
\langle \beta\rangle^{-2} \cdot|\cos(\beta\cdot(x_m-x_a))-1|
\le 
\langle \beta\rangle^{-2} \cdot |\beta|^2 |x_m-x_a|^2
\le
\eps^2+\max_{i,j\in I_*}|x_i-x_j|^2.
\]
With this factor, we can effectively replace $\Jc_{2p,\mathrm{prim}}$ by $\widetilde{\Jc}_{2p,\mathrm{prim}}$, see Proposition \ref{prop.primi_1}; now  \eqref{eq.primi_2.5} implies
\[
|\widetilde{\Jc}_{2p,\mathrm{prim}}|\leq(C\lambda)^{2p} \bigg(\frac{\eps^{-2}}{|\log\eps|}\cdot|z|^{-2}\mathbf{1}_{|z|\lesssim\eps}+\frac{1}{|\log\eps|^2}\cdot(|z|^2+\eps^2)^{-2}\bigg),
\qquad z=x_m-x_a.
\] Using this bound, and that $|G_0(z)|\lesssim |z|^{-2}$, we can bound $\widehat{\Pc_m}(\alpha,\beta)$ by first integrating $\widetilde{\Jc}_{2p,\mathrm{prim}}$ in $x_m$, then integrating $G_0$ in $x$, and finally integrating in $x_a$. THis leads to $|\widehat{\Pc_m}(\alpha,\beta)|\leq (C\lambda)^m|\log\eps|^{-1}$, hence (\ref{eq.setup_20}).

\textbf{Step 2.} Now suppose $\kappa$ is not a full pairing in (\ref{eq.Pm-hat}). To compute the second moment, we sum over pairings $\kappa_+$ of $[1,m]$ and $\kappa_-$ of $[m+1,2m]$ (say with sets of single indices being $S\pm$), and also sum over bijections $\pi:S_+\to S_-$, due to Wick's theorem, i.e.
\[
\Eb\Big(
\Wick{\prod_{i\in S_{+}}\xi_\eps(x_i)}\;
\Wick{\prod_{j\in S_{-}}\xi_\eps( x_j)}\Big)
= \sum_{\pi:S_{+} \to S_{-}}\prod_{i\in S_{+}}\eta_\eps(x_i- x_{\pi(i)}).
\]

Note that summing over $(\kappa_+,\kappa_-,\pi)$ is equivalent to summing over full pairings $\kappa'$ of $[1,2m]$, where $\kappa'=\kappa_+\cup\kappa_-\cup(\cup_i\{i,\pi(i)\}).$ For each $\kappa'$, we may identify all the fully paired subintervals \emph{within $[1,m]$ or within $[m+1,2m]$} as in Definition \ref{def.renormterm} and Proposition \ref{prop.renorm2}, and list them as $I_i:=[\ell_i,r_i]$; define $\widetilde{I}_i$ as before, and $\widetilde{I}_0:=[1,2m]\backslash(\cup_i I_i)$. Arguing similarly to \textbf{Step 1} in Section \ref{sec.4.1}, we may fix the positions of these subintervals at $O(C^m)$ cost. Then, similar to (\ref{eq.setup_7.5}), we can write (upon fixing these subintervals):
\begin{align}
(\textrm{This part of }\Eb|\widehat{\Pc_m}(\alpha,\beta)|^2)&=\lambda_\eps^{2m}
\sum_{\kappa_0,\cdots ,\kappa_{s}}
\int_{(\Tb^4)^{2m+4}}e^{i \alpha\cdot (x-z)+i\beta\cdot (y-w)}
G(x-x_1)G(x_m-y)G(z-x_{m+1})
\nonumber\\&\times 
\prod_{\substack{j=1\\ j\not\in\{m,r_1,\cdots,r_{s}\}}}^{2m-1}
G(x_j-x_{j+1})\prod_{i=1}^{s+t}\big(G(x_{r_i}-x_{r_i+1})-G(x_{\ell_i}-x_{r_i+1})\big)
\nonumber\\
&\times G(x_{2m}-w)\prod_{\{i,j\}\in\kappa}\eta_\eps(x_i-x_j)\cdot
\prod_{j\in \widetilde{I}_0}\mathrm{d}x_j\cdot
\prod_{i=1}^{s+t} 
\prod_{j\in \widetilde{I}_i}\mathrm{d}x_j\cdot
\mathrm{d}x\mathrm{d}y\mathrm{d}z\mathrm{d}w
\label{eq.3.24_new}
\end{align}
where $\kappa_i$ is a pairing of $\widetilde{I}_i$ (primitive for $i\geq 1$) and $\kappa'=\kappa_0\cup\cdots\cup\kappa_{s}$ (if $m$ or $2m$ is involved in a fully paired subinterval, then the factors $G(x_m-y)$ or $G(x_{2m}-w)$ should be replaced by corresponding $G(\cdot)-G(\cdot)$ differences, but this does not affect the arguments below).

Now we argue as in Section \ref{sec.4.1} to successively sum over the primitive pairings $\kappa_i\,(i\geq 1)$ and remove each $I_i$, gaining powers of $C\lambda$ and introducing new input functions $H(z)$ satisfying $|H(z)|\lesssim|z|^{-2}$; in the end we are left with $\widetilde{I}_0$ and the pairing $\kappa_0$ (which may \emph{not} be primitive, see below). By renaming the variables, we can refer to the remaining variables coming from $\{x_1,\cdots,x_m\}$ as $\{x_1,\cdots,x_p\}$, and those coming from $\{x_{m+1},\cdots,x_{2m}\}$ as $\{x_{p+1},\cdots,x_{p+q}\}$. Now we can take absolute values, using the upper bounds for the new input functions $H$, and bound
\begin{multline}\label{eq.3.24_new2}
(\textrm{This part of }\Eb|\widehat{\Pc_m}(\alpha,\beta)|^2)\leq(C\lambda)^{2m-p-q}\lambda_\eps^{p+q}
\sum_{\kappa_0} \int_{\Tb^{4(p+q+4)}}
 \prod_{\substack{j=1\\j\neq p}}^{p+q-1}|x_j-x_{j+1}|^{-2} \prod_{\{i,j\}\in \kappa_0}\eta_\eps(x_i-x_j)\\\times|x-x_1|^{-2}\cdot |x_p-y|^{-2}\cdot  |z-x_{p+1}|^{-2} \cdot |x_{p+q}-w|^{-2}\,\mathrm{d}x_1\cdots\mathrm{d}x_{p+q}\cdot \mathrm{d}x\mathrm{d}y\mathrm{d}z\mathrm{d}w.
\end{multline}

\textbf{Step 3.} In (\ref{eq.3.24_new2}) we first integrate in $(x,y,z,w)$ yielding constants, leading to
\begin{equation}\label{eq.missing-p}
\textrm{LHS of (\ref{eq.3.24_new2})}\leq(C\lambda)^{2m-p-q}\lambda_\eps^{p+q}
\sum_{\kappa_0} \int_{\Tb^{4(p+q)}}   \prod_{j=1}^{p+q-1} |x_j-x_{j+1}|^{-2} \cdot |x_p-x_{p+1}|^2
\prod_{\{i,j\}\in \kappa_0}\eta_\eps(x_i-x_j)
 \prod_{j=1}^{p+q}\mathrm{d}x_j.
\end{equation}
Here, the full pairing $\kappa_0$ of $[1,p+q]$ may have fully paired proper subintervals; however, by our construction, the partial pairing that $\kappa_0$ induces on $[1,p]$ and on $[p+1,p+q]$ must both be primitive and not full. This means that, all fully paired subintervals of $\kappa$ \emph{must} contain $\{p,p+1\}$, so these intervals must have the form $[a_i,b_i]$ where
\[
1\le a_t < a_{t-1} <\cdots <a_1 \le p < p+1 \le b_1 <b_2 <\cdots<b_t \le p+q.
\]
Now we perform successive reductions for these subintervals. Define $k_1:=b_1-a_1+1$. Note that the full pairing $\kappa_1$ of $[a_1,b_1]$ induced by $\kappa_0$ is primitive, so
by \eqref{eq.primi_2.5} in Proposition~\ref{prop.primi_1}, the expression
\[
\lambda_\eps^{k_1}\sum_{\kappa_1}\int\prod_{j=a_1}^{b_1-1}|x_j-x_{j+1}|^{-2}\cdot|x_p-x_{p+1}|^2
\prod_{\{i,j\}\in \kappa_1}\eta_\eps(x_i-x_j)
\,\mathrm{d}x_{a_1+1}\cdots \mathrm{d}x_{b_1-1},
\]
which is just $\widetilde{\Jc}_{k_1,\mathrm{prim}}(x_{a_1}-x_{b_1})$, is bounded by 
\[
(C\lambda)^{k_1}\bigg(\frac{\eps^{-2}}{|\log\eps|}\cdot|z|^{-2}\mathbf{1}_{|z|\lesssim\eps}+\frac{1}{|\log\eps|^2}\cdot(|z|^2+\eps^2)^{-2}\bigg),
\qquad z=x_{a_1}-x_{b_1}.
\]
It is then elementary to show that
\[
\int_{\Tb^{8}} |x_{a_1-1} - x_{a_1}|^{-2} \cdot\bigg(\frac{\eps^{-2}}{|\log\eps|}\cdot|z|^{-2}\mathbf{1}_{|z|\lesssim\eps}+\frac{1}{|\log\eps|^2}\cdot(|z|^2+\eps^2)^{-2}\bigg)\cdot
|x_{b_1} - x_{b_1+1}|^{-2}
\mathrm{d}x_{a_1} \mathrm{d}x_{b_1} \le C
\]
with $z:=x_{a_1}-x_{b_1}$, so we can bound \eqref{eq.missing-p} by $(C\lambda)^{2m-p-q+k_1}\lambda_\eps^{p+q-k_1}$ times
\[
\sum_{ \kappa_0'} \int  
\Big(\prod_{j=1}^{a_1-2} |x_j-x_{j+1}|^{-2} \Big)
\Big(\prod_{j=b_1+1}^{p+q-1} |x_j-x_{j+1}|^{-2} \Big)
\prod_{\{i,j\}\in\kappa_0'}\eta_\eps(x_i-x_j)
 \prod_{j=1}^{a_1-1}\mathrm{d}x_j \prod_{j=b_1+1}^{p+q}\mathrm{d}x_j
\]
where $\kappa_0':=\kappa_0\backslash\kappa_1$, which precisely reassembles the structure of \eqref{eq.missing-p}.
Then we reduce the subinterval $[a_2,b_2]$, etc. Eventually, we obtain
from $\kappa_0$ a primitive pairing $\kappa_t$ of $[1,a_t-1]\cup [b_t+1,p+q]$, then we can apply \eqref{eq.primi_2.5} in Proposition~\ref{prop.primi_1} one last time, to bound $\Eb|\widehat{\Pc_m}(\alpha,\beta)|^2$ by the integral of the right hand side of \eqref{eq.primi_2.5}, which is bounded by $(C\lambda)^{2m} |\log\eps|^{-1} \le \lambda_\eps^2 \cdot C(C\lambda)^{2m-2}$. This proves (\ref{eq.setup_20}).

\textbf{Step 4.} Finally, we consider (\ref{eq.setup_20}) when we take $\eps^{-8}\langle \alpha\rangle^{-4}\langle \beta\rangle^{-4}\langle \eps^2(\alpha+\beta)\rangle^{-8}$ on the right hand side. First we treat the decay factor in $\alpha+\beta$; note that $\Pc_m$ is the renormalized version of a linear operator which is composition of at most $|\log\eps|$ factors, each being either convolution by $G$ (which is a multiplication on Fourier side) or multiplication by $\xi_\eps$ (which is convolution by $\widehat{\rho_\eps}\cdot\widehat{\xi}$ on Fourier side). If this operator shifts the frequency by a distance of $|\alpha+\beta|\sim L\geq \eps^{-2}$, then one of the factors must shift the frequency by $\geq \eps^{1/2}L$, so for this factor we can replace $\xi_\eps$ by $\widetilde{\xi}_\eps$ which is the projection of $\xi_\eps$ to frequencies $\gtrsim \eps^{1/2}L$ and carries a decay factor of $\langle \eps^2L\rangle^{-8}$. In this way we get the decay factor $\langle \eps^2(\alpha+\beta)\rangle^{-8}$.

Now to deal with the decay factors in $\alpha$ and $\beta$, we exploit the oscillation in $(x,y,z,w)$ in (\ref{eq.3.24_new}). Consider for example the integral in $y$. If $m$ is not involved in a fully paired subinterval, then we can integrate in $y$ after removing each $I_i$ in (\ref{eq.3.24_new}) and before taking absolute value, to pick up a factor of $|\widehat{G}(\beta)|\leq\langle \beta\rangle^{-2}$. If $m$ is involved in a fully paired subinterval, say $[a,m]$ as in \textbf{Step 1}, then the same calculations there lead to the factor \[\langle \beta\rangle^{-2} \cdot|\cos(\beta\cdot(x_m-x_a))-1|.\] We now bound it by $2\langle \beta\rangle^{-2}$ to keep the decay in $\beta$, although this gives up a factor $|x_m-x_a|^2$ which forces us to consider $\Jc_{2p,\mathrm{prim}}$ instead of $\widetilde{\Jc}_{2p,\mathrm{prim}}$; upon integration this leads to a loss of $\eps^{-2}$. By doing the same for each one of $(x,y,z,w)$, we can secure the desired $\langle \alpha\rangle^{-4}\langle \beta\rangle^{-4}$ factor at the price of losing $\eps^{-8}$. This fnishes the proof of (\ref{eq.setup_20}).
\section{Proof of primitive pairing estimates}\label{sec.proof}
In this section we prove Proposition \ref{prop.primi_1}. The fact that $\Jc_{2n,\mathrm{prim}}\in\Ec$ easily follows because each $G_j\in\Ec$. Now we prove (\ref{eq.primi_2})--(\ref{eq.primi_2.5}). First, when $n=1$, by (\ref{eq.primi_1}) we have \[\Jc_{2,\mathrm{prim}}(z)=\lambda^2|\log\eps|^{-1}\cdot\eps^{-4}G(z)\cdot\eta(z/\eps),\] which matches the first term on the right hand side of (\ref{eq.primi_2}); same for $\widetilde{\Jc}_{2,\mathrm{prim}}$ and (\ref{eq.primi_2.5}). Below we assume $n\geq 2$; note also that 
\[\eps^2+\max_{i,j}|x_i-x_j|^2\geq \eps^2+|x_1-x_{2n}|^2=\eps^2+|z-w|^2,\] so to prove (\ref{eq.primi_2})--(\ref{eq.primi_2.5}), it suffices to prove that
\begin{multline}\label{eq.proof_1}
\sum_\kappa\int_{(\Tb^4)^{2n-2}}\big(\eps^2+\max_{i,j}|x_i-x_j|^2\big)\cdot\prod_{j=1}^{2n-1}|x_j-x_{j+1}|^{-2}\prod_{\{i,j\}\in\kappa}\eta_\eps(x_i-x_j)\,\mathrm{d}x_2\cdots\mathrm{d}x_{2n-1}\\\leq C^n|\log\eps|^{n-2}\cdot (|z-w|^2+\eps^2)^{-2}.
\end{multline} \subsection{First reductions}\label{sec.reduce} In the rest of this section we will prove (\ref{eq.proof_1}). In this subsection we first perform some reductions. First, define $\Ac:=\{\min(i,j):\{i,j\}\in\kappa\}$, then $\Ac$ is an $n$-element subset of $[1,2n]$, which has at most $\binom{2n}{n}\leq 4^n$ choices. In view of the $C^n$ factor on the right hand side of (\ref{eq.proof_1}), in the rest of the proof we may fix the choice of $\Ac$. This type of arguments (applied to any quantity/object that has $O(C^n)$ choices) will be frequently used below.

Once $\Ac$ is fixed, $\kappa$ is then equivalent to a bijection between $\Ac$ and $\Ac^c:=[1,2n]\backslash \Ac$, which we still denote by $\kappa$. Next, for each fixed $\kappa$, we decompose the domain of integration in (\ref{eq.proof_1}) (where each $x_j\in\Tb^4=[-\pi,\pi]^4$), by fixing $y_j:=\lfloor\eps^{-1}x_j\rfloor$ for each $j\in\Ac$, where $\lfloor z\rfloor$ is defined by taking integer part (floor function) of each coordinate of $z$. By definition we then have \[y_j\in\Zb_M^4,\qquad \Zb_M:=\Zb\cap[-M,M],\] where $M\sim\eps^{-1}$ is a fixed dyadic number (i.e. $M\in 2^\Nb$). Moreover for $j\in\Ac$ we have $|x_j-\eps y_j|\lesssim\eps$; in view of the factor $\eta_\eps(x_i-x_j)$ in (\ref{eq.proof_1}) and the compact support of $\eta$, we also have $|x_{\kappa(j)}-\eps y_j|\lesssim\eps$ for $j\in \Ac$. As such, we naturally define $y_{\kappa(j)}=y_j$ for $j\in \Ac$. This leds to
\begin{equation}\label{eq.reduce_1}
\textrm{LHS of }(\ref{eq.proof_1})\leq C^n\sup_\Ac\sum_\kappa\sum_{(y_j:j\in\Ac)}\eps^2\max_{i,j}\langle y_i-y_j\rangle^2\int_{|x_j-\eps y_j|\lesssim\eps}\eps^{-4n}\prod_{j=1}^{2n-1}|x_j-x_{j+1}|^{-2}\,\mathrm{d}x_2\cdots\mathrm{d}x_{2n-1}.
\end{equation} Next we claim that the integral in (\ref{eq.reduce_1}) is bounded by
\begin{equation}\label{eq.reduce_2}
\int_{|x_j-\eps y_j|\lesssim\eps}\prod_{j=1}^{2n-1}|x_j-x_{j+1}|^{-2}\,\mathrm{d}x_2\cdots\mathrm{d}x_{2n-1}\leq C^n\eps^{4n-6}\prod_{j=1}^{2n-1}\langle y_j-y_{j+1}\rangle^{-2}.
\end{equation} In fact, if $|y_j-y_{j+1}|\lesssim 1$ for each $j$, then $|x_j-x_{j+1}|\lesssim\eps$ for each $j$, so by using this support of $|x_j-x_{j+1}|$ and rescaling, we can then bound
\begin{equation}\label{eq.reduce_3} \textrm{LHS of }(\ref{eq.reduce_2})\leq \eps^{4n-6}\cdot\big\|(|x|^{-2}\mathbf{1}_{|x|\lesssim1})^{\oast (2n-1)}\big\|_{L^\infty}\leq C^n\eps^{4n-6}
\end{equation} which implies (\ref{eq.reduce_2}), where $\oast(2n-1)$ means $(2n-1)$ fold convolution; the proof of (\ref{eq.reduce_3}) follows from the elementary facts that \[2n-1\geq 3,\quad (|x|^{-2}\mathbf{1}_{|x|\lesssim1})^{\oast3}\in L^\infty,\quad \textrm{and}\quad\|f *(|x|^{-2}\mathbf{1}_{|x|\lesssim1})\|_{L^\infty}\lesssim\|f\|_{L^\infty}.\] On the other hand, if $|y_i-y_{i+1}|\gg 1$ for some $i$, then for this $i$ we have $|x_i-x_{i+1}|\sim \eps\langle y_i-y_{i+1}\rangle$ pointwise; moreover, for each $1\leq j<i$ we have
\[\int_{|x_{j+1}-\eps y_{j+1}|\lesssim\eps}|x_j-x_{j+1}|^{-2}\,\mathrm{d}x_{j+1}\lesssim\eps^2\langle y_j-y_{j+1}\rangle^{-2}\] uniformly in $x_j$ under the assumption $|x_j-\eps y_j|\lesssim\eps$, and the same is true for $i<j\leq 2n-1$ but with the roles of $j$ and $j+1$ switched, so (\ref{eq.reduce_2}) follows by integrating in the order $x_{i+1}\to x_{i+2}\cdots\to x_{2n-1}\to x_i\to x_{i-1}\cdots\to x_2$.

Now, inserting (\ref{eq.reduce_2}) to (\ref{eq.reduce_1}), we get
\begin{equation}\label{eq.reduce_4}
\textrm{LHS of }(\ref{eq.proof_1})\leq C^n\eps^{-4}\sup_\Ac\sum_\kappa\sum_{(y_j)}\max_{i,j}\langle y_i-y_j\rangle^2\prod_{j=1}^{2n-1}\langle y_j-y_{j+1}\rangle^{-2}.
\end{equation} It is important to quantify the summation in (\ref{eq.reduce_4}): with $\Ac$ and $\Ac^c=[1,2n]\backslash\Ac$ fixed, this summation is taken over all bijections $\kappa:\Ac\to \Ac^c$, and all sequences $(y_1,\cdots,y_{2n})\in(\Zb_M^4)^{2n}$, which satisfy $y_j=y_{\kappa(j)}$ for all $j\in \Ac$. Moreover, since $z=x_1$ and $w=x_{2n}$ are fixed, the values of $y_1$ and $y_{2n}$ are also fixed in this summation. Note also that $\kappa(1)\neq 2n$ (because $n\geq 2$ and $\kappa$ is primitive).

\subsection{The Hepp tree} To estimate the summation in (\ref{eq.reduce_4}), it is important to classify the scales defined by the distances $|y_i-y_j|$. For this purpose, we introduce the notion of \emph{Hepp trees} below.
\begin{definition}
\label{def.hepp} Define a \emph{Hepp tree} $\Tc$ to be a rooted tree where each branching node has at least two children. Define a \emph{marking} of $\Tc$ to be a set of dyadic numbers $N_\nf\in 2^{\Nb}$ for each branching node $\nf$ such that $N_{\nf^+}>N_\nf>1$ where $\nf^+$ is the parent of $\nf$. Also, define a choice of \emph{multiplicities} of $\Tc$ to be a set of positive integers $m_\lf\geq 2$ for each \emph{leaf} $\lf$.
\end{definition}
\begin{lemma}\label{lem.hepp_2} The number of (unmarked) Hepp trees of $r$ leaves is bounded by $C^r$.
\end{lemma}
\begin{proof} We are just counting the number of (unlabeled) rooted trees of $r$ leaves, such that each branching node has at least two children.  Let this number be $A_r$, then $A_1=1$, and 
\[A_r=\sum_{q=2}^\infty\sum_{r_1+\cdots +r_q=r}A_{r_1}\cdots A_{r_q}\quad(r\geq 2).\] Considering the generating function $f(x)=\sum_r A_rx^r$, we get
\[f(x)=x+\sum_{q=2}^\infty f(x)^q=x+\frac{f(x)^2}{1-f(x)}\Rightarrow f(x)=\frac{x+1-\sqrt{x^2-6x+1}}{4}\] which is analytic in $|x|<3-2\sqrt{2}$, it then follows $A_r\lesssim 6^r$.
\end{proof}
\begin{definition}\label{def.hepp_2} Let $\Tc$ be a Hepp tree (Definition \ref{def.hepp}), possibly with a marking $(N_\nf)$ and multiplicities $(m_\lf)$. We define the following notations:
\begin{enumerate}[{(1)}]
\item Recall $\nf^+$ is the parent of $\nf$. Define $\Cc_\nf$ to be the set of children of $\nf$, and $\Tc_\nf$ to be the subtree rooted at $\nf$ (with $\Tc_\nf=\{\nf\}$ when $\nf$ is a leaf), and write $\mf\leq\nf$ if $\mf\in \Tc_\nf$. Define $\Bc$ to be the set of branching nodes, $\Lc$ the set of leaves, and $\rf$ the root. We also write $\Lc_\nf=\Lc\cap \Tc_\nf$ and $\Bc_\nf=\Bc\cap\Tc_\nf$ etc. 
\item For $\nf\in\Bc$, define $\gamma_\nf=|\Cc_\nf|$ to be the number of children of $\nf$ (leaves are counted \emph{without} multiplicity; later we also need to treat leaves \emph{with multiplicity}, which will be clearly indicated below).
\item Finally, for $\mf,\mf'\in\Tc$, define $\mf\vee\mf'$ to be the lowest common ancestor of $\mf$ and $\mf'$. Note that $\mf\vee\mf'\in\Lc$ if and only if $\mf=\mf'\in\Lc$; otherwise $\mf\vee\mf'\in\Bc$.
\end{enumerate}

We usually denote $r$ to be the total number of leaves of $\Tc$ without multiplicity; it is easy to see that $\sum_{\nf\in\Bc}(\gamma_\nf-1)=r-1$. We may also denote by $m=\sum_{\lf\in\Lc}m_\lf$ the total number of leaves of $\Tc$ with multiplicity (each leaf $\lf$ counted $m_\lf$ times).
\end{definition}
\begin{definition}\label{def.hepp_3} Let $\Tc$ be a Hepp tree with marking $(N_\nf)$. A mapping $(z_\lf):\Lc\to\Zb_M^4$ is called \emph{admissible}, if it is injective (i.e. $z_\lf$ are different for different $\lf$), and satisfies that
\begin{enumerate}[{(a)}]
\item $|z_\lf-z_{\lf'}|\geq \frac{1}{2}N_{\lf\vee\lf'}$ for $\lf,\lf'\in\Lc$ with $\lf\neq\lf'$;
\item For any $\nf\in\Bc$, all its children $\nf_1,\cdots,\nf_s$ can be connected by some links between them, such that for each link $(j,k)$, there exists $\lf\in \Lc_{\nf_j}$ and $\lf'\in \Lc_{\nf_k}$ such that $|z_\lf-z_{\lf'}|\leq N_{\nf}$.
\end{enumerate} 
Let $Z$ be a (non-repeating) subset of $\Zb_M^4$ with $|Z|=r$ (the number of leaves of $\Tc$), we define $Z\leftrightarrow (\Tc,N_\nf)$, if there exists a bijective mapping $(z_\lf):\Lc\to Z$ that is admissible.

Now, suppose we are also given the multiplicities $m_\lf$, and let $m=\sum_{\lf\in\Lc} m_\lf$. For any $(y_1,\cdots,y_{m})\in(\Zb_M^4)^{m}$, we define $(y_1,\cdots,y_{m})\leftrightarrow (\Tc,N_\nf,m_\lf)$, if there exists an admissible $(z_\lf)$ such that the \emph{multiset} $\{y_1,\cdots,y_{m}\}=\cup_{\lf\in\Lc}\{\textrm{$m_\lf$ copies of $z_\lf$}\}$. We also write $(y_1,\cdots,y_{m})\leftrightarrow(\Tc,N_\nf)$ (resp. $(y_1,\cdots,y_{m})\leftrightarrow\Tc$), if there exists some choice of $(m_\lf)$ (resp. some choice of $(N_\nf,m_\lf)$) such that $(y_1,\cdots,y_{m})\leftrightarrow (\Tc,N_\nf,m_\lf)$.
\end{definition}
\begin{lemma}\label{lem.hepp} For any (non-repeating) finite subset $Z\subseteq\Zb_M^4$ there always exists $(\Tc,N_\nf)\leftrightarrow Z$ with $1<N_\nf\lesssim M$. For any vector $(y_1,\cdots,y_{2n})\in (\Zb_M^4)^{2n}$ such that $y_{j}=y_{\kappa(j)}$ for some bijection $\kappa:\Ac\to\Ac^c$, there always exists $(\Tc,N_\nf,m_\lf)\leftrightarrow (y_1,\cdots,y_{2n})$ with $1<N_\nf\lesssim M$ and $m_\lf$ even.
\end{lemma}
\begin{proof} Let the set of distinct points in $(y_j)$ be $Z=\{z_1,\cdots,z_r\}$; by the existence of $\kappa$, each $z_i$ occurs an even number of times, say $m_i$ times. We only need to construct $(\Tc,N_\nf)\leftrightarrow Z$, then by Definition \ref{def.hepp_3} we will automatically have $(\Tc,N_\nf,m_\lf)\leftrightarrow (y_1,\cdots,y_{2n})$, where $m_\lf=m_i$ for $z_\lf=z_i$.

Start from $r$ leaves $\lf$ representing these $z_i$, which we view as a forest formed by $r$ (trivial) trees. Define $N_0\in 2^\Nb$ such that
\[\min\big\{|z_\lf-z_{\lf'}|:\lf\neq\lf'\big\}\in [N_0/2,N_0),\] then $N_0\geq 2$. Consider all pairs $(\lf,\lf')$ satisfying $|z_\lf-z_{\lf'}|\in [N_0/2,N_0)$, and the connected components of the resulting graph defined by the leaves and these leaf pairs. For each connected component of $\geq 2$ leaves (at least one of them exists by definition of $N_0$), draw a branching node $\nf$ whose children are the leaves in the this component, and mark it by $N_\nf=N_0$. By construction, Definition \ref{def.hepp_3} (a)--(b) are satisfied by each $\nf$.

Now, we have turned the forest of $r$ trees into a smaller forest of $q<r$ trees; moreover for leaves $\lf,\lf'$ in different trees we must have $|z_\lf-z_{\lf'}|\geq N_0$. Then define $N_1\in 2^\Nb$ such that
\[\min\big\{|z_\lf-z_{\lf'}|:\textrm{$\lf,\lf'$ in different trees}\big\}\in [N_1/2,N_1),\] then $N_1\geq 2N_0$. We then consider all pairs of trees $(\Tc,\Tc')$ such that $|z_\lf-z_{\lf'}|\in [N_1/2,N_1)$ for some leaf $\lf$ of $\Tc$ and $\lf'$ of $\Tc'$, and the connected components of the resulting graph defined by the trees and these tree pairs. For each connected component of $\geq 2$ trees (there is at least one of them by definition of $N_1$), draw a branching node $\mf$ whose children are the roots of the trees in the corresponding component, and mark it by $N_\mf=N_1$. By construction, Definition \ref{def.hepp_3} (a)--(b) are satisfied by each $\mf$.

Repeating this, in the end we get a single tree $\Tc$ with a marking $(N_\nf)$, for which $(z_\lf):\Lc\to Z$ is bijective and admissible. Therefore we get $Z\leftrightarrow(\Tc,N_\nf)$, which completes the proof.
\end{proof}

{\it Example.} Consider $\{z_1,\cdots,z_{100}\}$ with $z_i=(i,0,0,0)$. The resulting Hepp tree consists of a root with $100$ children all being leaves. Indeed, $N_0=2$, and in the first step every leaf pair $(\lf_j,\lf_{j+1})$ is linked yielding one single connected component. This is a possible configuration which would correspond to \eqref{eq.idea_9}. Note that our construction, compared to the earlier construction of binary Hepp trees, guarantees the \emph{strict} inequality $N_{\nf^+}>N_{\nf}$, which is important in controlling the combinatorial factors, see (\ref{eq.hepp_sum}) and \cite[(A.15)]{HairerQuastel}.

\medskip

Now start from (\ref{eq.reduce_4}). By Lemma \ref{lem.hepp}, for each $\kappa$ and each $(y_1,\cdots,y_{2n})$ (which we may abbreviate as $(y_j)$), there must exist a Hepp tree $\Tc$ such that $(y_j)\leftrightarrow\Tc$. This $\Tc$ has $O(C^n)$ choices by Lemma \ref{lem.hepp_2}, so we may fix the choice of $\Tc$ from now on. Once $\Tc$ is fixed and $(y_j)\leftrightarrow\Tc$, there must exist some $(N_\nf)$ and even $(m_\lf)$ such that $(y_j)\leftrightarrow (\Tc,N_\nf,m_\lf)$. This leads to the following decomposition
\begin{equation}\label{eq.vol_1}
\sum_{(y_j)\leftrightarrow \Tc}\sum_\kappa\ (\cdots)=\sum_{(N_\nf,m_\lf)}\ \sum_{(y_j)\leftrightarrow (\Tc,N_\nf,m_\lf)}\sum_\kappa\ \frac{(\cdots)}{\#\{(N_\nf',m_\lf'):(y_j)\leftrightarrow (\Tc,N_\nf',m_\lf')\}}
\end{equation} (which follows by switching the order of summation), where $(\cdots)$ is any expression in $(y_1,\cdots,y_{2n})$, and the denominator is the \emph{number of} $(N_\nf',m_\lf')$ with $m_\lf'$ even such that $(y_1,\cdots,y_{2n})\leftrightarrow (\Tc,N_\nf',m_\lf')$.

In (\ref{eq.vol_1}), define $Z$ to be the set of distinct points in $(y_j)$, then $(y_j)\leftrightarrow (\Tc,N_\nf,m_\lf)$ implies $Z\leftrightarrow(\Tc,N_\nf)$. 
Therefore we can further decompose (\ref{eq.vol_1}) as
\begin{equation}\label{eq.vol_1.5}
\sum_{(y_j)\leftrightarrow \Tc}\sum_\kappa\ (\cdots)=\sum_{(m_\lf)}\sum_{(N_\nf)}\sum_{Z\leftrightarrow (\Tc,N_\nf)}\sum_{\substack{(y_j)\leftrightarrow (\Tc,N_\nf,m_\lf)\\ Z\textrm{ fixed}}}\sum_{\kappa}\frac{(\cdots)}{\#\{(N_\nf',m_\lf'):(y_j)\leftrightarrow (\Tc,N_\nf',m_\lf')\}}.
\end{equation} For fixed $Z=\{z_1,\cdots,z_r\}$, consider the \emph{multiset} $Y=\{y_1,\cdots,y_{2n}\}$; this $Y$ is uniquely determined by a permutation $(m_\lf)_{\lf\in\Lc}$ into $(m_1,\cdots,m_r)$ (denote it by $\chi$), such that $Y=Z_\chi:=\cup_{i=1}^r\{\textrm{$m_i$ copies of $z_i$}\}$. Therefore we can further decompose (\ref{eq.vol_1.5}) as
\begin{equation}\label{eq.vol_1.8}
\sum_{(y_j)\leftrightarrow \Tc}\sum_\kappa\ (\cdots)=\sum_{(m_\lf)}\sum_\chi\sum_{(N_\nf)}\sum_{Z\leftrightarrow (\Tc,N_\nf)}\sum_{\substack{(y_j):\textrm{ permutation}\\\textrm{of  $Y=Z_\chi$}}}\sum_{\kappa}\frac{(\cdots)}{\#\{(N_\nf',m_\lf'):(y_j)\leftrightarrow (\Tc,N_\nf',m_\lf')\}}.
\end{equation}

Note that the number of choices of $(m_\lf)$ with $\sum_\lf m_\lf=2n$, is bounded by $2^{2n}$. Also for fixed $(m_\lf)$, if the value $\ell\,(1\leq \ell\leq q)$ occurs $b_\ell$ times among these $m_\lf$, where $\sum_\ell b_\ell=r$ and $\sum_\ell \ell b_\ell=\sum_\lf m_\lf=2n$, then the number of choices of $\chi$ is
\begin{equation}\label{eq.perm}\frac{r!}{b_1!\cdots b_q!}=\frac{(b_1+\cdots +b_q)!}{b_1!\cdots b_q!}=\prod_{\ell=1}^q\binom{b_\ell+\cdots +b_q}{b_\ell}\leq\prod_{\ell=1}^q 2^{b_\ell+\cdots +b_q}=2^{b_1+\cdots +qb_q}=2^{2n}.\end{equation} As such, in the proof below we will fix the choices of $(m_\lf)$ and $\chi$.

Now consider (\ref{eq.vol_1.8}) with $(m_\lf)$ and $\chi$ fixed. For each fixed $Z$, the multiset $Y=Z_\chi$ is also fixed; we denote it by $Y=\{y_1^*,\cdots,y_{2n}^*\}$. Then we have
\begin{equation}\label{eq.vol_1_new}
\sum_{\substack{(y_j):\textrm{ permutation}\\\textrm{of  $Y$}}}\sum_{\kappa}\,\mathrm{Func}(y_1,\cdots,y_{2n})\leq\prod_\lf \frac{(m_\lf/2)!}{m_\lf!}\sum_{\pi}\,\mathrm{Func}(y_{\pi(1)}^*,\cdots,y_{\pi(2n)}^*)
\end{equation} for any function $\mathrm{Func}(y_1,\cdots,y_{2n})\geq 0$, where $\pi$ runs over permutations of $[1,2n]$ (\emph{not} of $Y$). This is because each permutation $(y_j)$ of $Y$ corresponds to $\prod_\lf m_\lf!$ permutations of $[1,2n]$ (as each $z_\lf$ has $m_\lf$ copies); also for fixed $(y_j)$, the number of choices of $\kappa$ is bounded by $\prod_\lf (m_\lf/2)!$ (because $y_k=y_{\kappa(j)}$, so $\kappa$ is determined by a collection of bijections between two sets of cardinality $m_\lf/2$), so (\ref{eq.vol_1_new}) is true.

{\it Example.} We demonstrate the above re-sum with $(y_1,\cdots,y_8)$, namely $n=4$. First, fix $\Ac=\{1,2,4,6\}$ and $\Ac^c=\{3,5,7,8\}$. Then, fix $\Tc$ to be the tree whose root $\rf$ has three children all being leaves $\lf_1,\lf_2,\lf_3$. Then, fix $N_{\rf}=32$, $(m_{\lf_1},m_{\lf_2},m_{\lf_3})=(4,2,2)$. With these fixed, one can then fix $Z=\{a,b,c\}$ and $\chi=(4,2,2)$, which determines the multiset $Y=\{a,a,a,a,b,b,c,c\}$, and then $(y_j)$ runs over all permutations of $Y$. 
(Note that if we instead fix $Z=\{a,b,c\}$ and $\chi=(2,4,2)$, then $Y=\{a,a,b,b,b,b,c,c\}$.) With this $Y$, if $(y_j)=(a,b,a,a,b,c,a,c)$, then there exists a bijection $\kappa:\Ac\to \Ac^c$ given by $1\mapsto 3$, $2\mapsto 5$, $4\mapsto 7$, $6\mapsto 8$ such that $y_{\kappa(j)}=y(j)$. On the other hand if $(y_j)=(a,a,b,a,b,a,c,c)$ there is no such $\kappa$.

Clearly $(y_j)=(a,b,a,a,b,c,a,c)$ corresponds to $4!2!2!$ permutations of $[1,8]$ by swapping identical letters, and it also admits $2!1!1!$ choices of $\kappa$ since one needs to decide how the $a$ at locations $\{1,4\}\subset \Ac$ are mapped to the $a$ at locations $\{3,7\}$, etc. Note that when upper bounding in terms of $\pi\in S_8$,  the case $(y^*_{\pi(j)})=(a,a,a,a,b,c,b,c)$ (for instance) should be avoided, since all the copies of $a$ form a single block in $\pi$ which would mean that $\kappa$ has a fully paired interval, i.e. not primitive, see (c) below \eqref{eq.vol_1_fin}. 
\medskip

Putting everything together, by fixing $(m_\lf,\chi)$, applying (\ref{eq.vol_1_new}) to (\ref{eq.vol_1.8}) and inserting the summand in (\ref{eq.reduce_4}) into the above $(\cdots)$, we get that (note also that the denominators in (\ref{eq.vol_1})--(\ref{eq.vol_1.8}) are symmetric in $(y_j)$)
\begin{multline}\label{eq.vol_1_fin}
\textrm{LHS of }(\ref{eq.proof_1})\leq C^n\eps^{-4}\sup_{(\Ac,\Tc,m_\lf,\chi)}\sum_{(N_\nf)}\sum_{Z\leftrightarrow (\Tc,N_\nf)}\prod_\lf(m_\lf!)^{-1/2}\cdot\frac{1}{\#\{(N_\nf',m_\lf'):(y_j^*)\leftrightarrow (\Tc,N_\nf',m_\lf')\}}\\
\times\sum_{\pi}^{(*)}\max_{i,j}\langle y_i^*-y_j^*\rangle^2\prod_{j=1}^{2n-1}\langle y_{\pi(j)}^*-y_{\pi(j+1)}^*\rangle^{-2},
\end{multline} where $Y=\{y_1^*,\cdots,y_{2n}^*\}=Z_\chi$ as defined above and $\pi\in S_{2n}$ (this $\pi$ will be identified with a permutation of $m_\lf$ \emph{distinct} copies of each leaf $\lf$). Here in (\ref{eq.vol_1_fin}) we note the following conditions in the summation:
\begin{enumerate}[{(a)}]
\item We have $1\leq N_\nf\lesssim M$ for each $\nf$;
\item Note that the values of $y_1$ and $y_{2n}$ are fixed, so in the summation in $Z$, these two values (or one value if $y_1=y_{2n}$) must belong to $Z$;
\item Since $\kappa$ is primitive, this implies that in the summation in $\pi$, this $\pi$ must satisfy the following condition: for \emph{any proper subset} $\Nc$ of leaves in $\Tc$, all the copies of leaves $\lf\in\Nc$ \emph{cannot} form a single block (i.e. image of a subinterval) in $\pi$. 
\end{enumerate}

The goal is now to estimate (\ref{eq.vol_1_fin}) under the assumptions (a)--(c) above. To do this we will treat the summation in $\pi$ (with fixed $Z$), and the summation in $Z$ separately, in the following two propositions:
\begin{prop}\label{prop.hepp_3} 
For any Hepp tree $\Tc$ and marking $(N_\nf)$, let $\mathrm{Aut}(\Tc)$ be the automorphism group of $\Tc$  (preserving the root), and $\mathrm{Aut}(\Tc,N_\nf)$ be the automorphism group of the marked tree $\Tc$ (preserving the root and the marks). Then, for any $(m_\lf)$ and any $(y_j)\leftrightarrow (\Tc,N_\nf,m_\lf)$ we have
\begin{equation}\label{eq.hepp_2}
\#\{(N_\nf',m_\lf'):(y_j)\leftrightarrow (\Tc,N_\nf',m_\lf')\}\geq\frac{|\mathrm{Aut}(\Tc)|}{|\mathrm{Aut}(\Tc,N_\nf)|}.
\end{equation}
In addition, for any $z_0,z_1\in\Zb_M^4$ we have
\begin{equation}\label{eq.vol_2}\frac{|\mathrm{Aut}(\Tc,N_\nf)|}{|\mathrm{Aut}(\Tc)|}\sum_{\substack{Z\leftrightarrow (\Tc,N_\nf)\\z_0\in Z}}1\leq C^n\prod_{\nf\in\Bc} N_\nf^{4(\gamma_\nf-1)},\end{equation}
\begin{equation}\label{eq.vol_2.1}
\frac{|\mathrm{Aut}(\Tc,N_\nf)|}{|\mathrm{Aut}(\Tc)|}\sum_{\substack{Z\leftrightarrow (\Tc,N_\nf)\\z_0,z_1\in Z}}1\leq C^n\cdot\langle z_0-z_1\rangle^{-4}\cdot\prod_{\nf\in\Bc} N_\nf^{4(\gamma_\nf-1)},
\end{equation}
where $\gamma_\nf$ is the number of children of $\nf$ as in Definition \ref{def.hepp_2}.
\end{prop}
\begin{prop}\label{prop.main_3} Fix a Hepp tree $\Tc$ and $(N_\nf,m_\lf)$ with $m_\lf$ even, and a vector $(y_1,\cdots,y_{2n})$ such that $(y_j)\leftrightarrow (\Tc,N_\nf,m_\lf)$ as in Definition \ref{def.hepp_3}. Then we have
\begin{equation}\label{eq.vol_3}\sum_{\pi}^{(*)}\prod_{j=1}^{2n-1}\frac{1}{\langle y_{\pi(j)}-y_{\pi(j+1)}\rangle^2}\leq C^n\sum_{W\subseteq\Bc\backslash\{\rf\}}(n-|W|)!\prod_{\lf\in\Lc}(m_\lf!)^{1/2}\cdot\prod_{\nf\in\Bc} N_\nf^{-4(\gamma_\nf-1)}\cdot N_\rf^{-2}\prod_{\nf\not\in W}\frac{N_\nf}{N_{\nf^+}}.
\end{equation} Here $\pi\in S_{2n}$ in (\ref{eq.vol_3}) (which is identified with a permutation of $m_\lf$ \emph{distinct} copies of each leaf $\lf$), and satisfies condition that, for any \emph{proper} subset of leaves, all their copies cannot form a single block in $\pi$ (a single block is the image of a subinterval under $\pi$). Below we still call this the \emph{primitivity} condition for $\pi$.
\end{prop}

In the rest of this subsection we prove Proposition \ref{prop.primi_1} under the assumption of Propositions \ref{prop.hepp_3} and \ref{prop.main_3}. The proofs of Propositions \ref{prop.hepp_3} and \ref{prop.main_3} are left to Sections \ref{sec.proof_hepp_3} and \ref{sec.proof_main_3} below.

\begin{proof}[Proof of Proposition \ref{prop.primi_1} assuming Propositions \ref{prop.hepp_3} and \ref{prop.main_3}] We only need to prove (\ref{eq.proof_1}). Start from (\ref{eq.vol_1_fin}). Note that, if $(y_1,\cdots,y_{2n})\leftrightarrow(\Tc,N_\nf,m_\lf)$, then by Definition \ref{def.hepp_3} (b), it is easy to see that $\langle y_i-y_j\rangle\leq O(n^2)\cdot N_\rf$ uniformly in $(i,j)$; for a precise argument see \textbf{Step 4} in the proof of Proposition \ref{prop.hepp_3} in Section \ref{sec.proof_hepp_3} below. This allows us to replace the factor $\max_{i,j}\langle y_i^*-y_j^*\rangle^2$ on the right hand side of (\ref{eq.vol_1_fin}) by $N_\rf^2$ up to loss $O(n^2)\leq C^n$.

After this, by estimating the inner sum (in $\pi$) using (\ref{eq.vol_3}), and estimating the middle sum (in $Z$) using (\ref{eq.hepp_2}) and (\ref{eq.vol_2.1}), and cancelling out various factors, we get
\begin{equation}\label{eq.vol_4}
\textrm{LHS of }(\ref{eq.proof_1})\leq C^n\eps^{-4}\sup_{(\Ac,\Tc,m_\lf,\chi)}\langle y_1-y_{2n}\rangle^{-4}\sum_{W\subseteq\Bc\backslash\{\rf\}}(n-|W|)!\sum_{(N_\nf)}\prod_{\nf\not\in W}\frac{N_\nf}{N_{\nf^+}}.
\end{equation} Since $(x_1,x_{2n})=(z,w)$ and $|x_j-\eps y_j|\lesssim\eps$, we get that \[\eps^{-4}\langle y_1-y_{2n}\rangle^{-4}\sim (|z-w|^2+\eps^2)^{-2}.\] 

Finally, consider the summation over $(N_\nf)$ on the right hand side of (\ref{eq.vol_4}). This is taken for all $\nf\in\Bc$; note that $r-1=\sum_{\nf\in\Bc}(\gamma_\nf-1)$ and $\gamma_\nf\geq 2$ implies $|\Bc|\leq r-1$. If we take the summation in the order of $N_\nf$ with $\nf$ from low to high, then for each $\nf\in (\Bc\backslash\{\rf\})\backslash W$, the value of $N_{\nf^+}$ has been fixed when we sum over $N_\nf$, so the summation in $N_\nf$ gives $\sum_{N_\nf\leq N_{\nf^+}}(N_\nf/N_{\nf^+})\leq C$; of course, for each $\nf\in W\cup\{\rf\}$ the summation in $N_\nf$ gives $\sum_{1\leq N_\nf\lesssim M}1\leq C|\log\eps|$. This implies that
\begin{equation}\label{eq.vol_5}
\sum_{(N_\nf)}\prod_{\nf\not\in W}\frac{N_\nf}{N_{\nf^+}}\leq C^n|\log\eps|^{|W|+1}.
\end{equation} 

Note that the exponent on the right hand side of (\ref{eq.vol_5}) is $\leq r-1\leq n-1$. If equality holds we must have $W=\Bc\backslash\{\rf\}$ and $r=n$ (and thus $m_\lf=2$ for each $\lf$ in view of $2n=\sum_\lf m_\lf$), in particular $y_1\neq y_{2n}$ (as $\kappa(1)\neq 2n$). Up to $O(n^2)$ loss, we may fix leaves $\ff_0\neq\ff_1$ such that $y_{1}=z_{\ff_0}$ and $y_{2n}=z_{\ff_1}$; in this case, with $y_1$ and $y_{2n}$ fixed, we have 
\[N_{\ff_0\vee \ff_1}/2\leq \langle y_1-y_{2n}\rangle\leq O(n^2)\cdot N_{\ff_0\vee \ff_1}\] by Definition \ref{def.hepp_3} (a)--(b) and the same proof quoted above, so the summation in the variable $N_{\ff_0\vee\ff_1}$ in (\ref{eq.vol_5}) is bounded by $O(\log n)\leq C^n$ rather than $O(|\log\eps|)$. In summary, we have that
\[\textrm{LHS of }(\ref{eq.vol_5})\leq C^n|\log\eps|^{\min(|W|+1,n-2)};\] inserting this into (\ref{eq.vol_4}) and using that $n\lesssim|\log\eps|$, we get
\[
\begin{aligned}\textrm{LHS of }(\ref{eq.proof_1})&\leq C^n\cdot (|z-w|^2+\eps^2)^{-2}\cdot\sum_{W\subseteq\Bc\backslash\{\rf\}}(n-|W|)!\cdot |\log\eps|^{\min(|W|+1,n-2)}\\
&\leq C^n\cdot (|z-w|^2+\eps^2)^{-2}\cdot\sup_{0\leq s\leq r-2}(n-s)!\cdot|\log\eps|^{\min(s+1,n-2)}\\
&\leq  C^n\cdot (|z-w|^2+\eps^2)^{-2}\cdot|\log\eps|^{n-2}
\end{aligned}\] using the inequality
\[(n-s)!\cdot|\log\eps|^{\min(s+1,n-2)}\leq C^n\cdot |\log\eps|^{n-2},\quad \forall \,0\leq s\leq n-2\] which follows from elementary calculus. This completes the proof.
\end{proof}

\subsection{Proof of Proposition \ref{prop.hepp_3}}\label{sec.proof_hepp_3} In this subsection we prove Proposition \ref{prop.hepp_3}. The proof is divided into 6 steps. In the proof below note $r\leq n$.

\textbf{Step 1.} First we prove (\ref{eq.hepp_2}). This is easy, because for each $\pi\in\mathrm{Aut}(\Tc)$, let $N_\nf':=N_{\pi(\nf)}$ and $m_\lf':=m_{\pi(\lf)}$, then by Definition \ref{def.hepp_3} we also have $(y_1,\cdots,y_n)\leftrightarrow (\Tc,N_\nf',m_\lf')$. Note that this defines an action of $\mathrm{Aut}(\Tc)$ on the set of $(N_\nf)$, so the orbit-stabilizer theorem implies that the number of choices for such $(N_\nf')$ (which is not more than the number of choices for $(N_\nf',m_\lf')$ in (\ref{eq.hepp_2})), is equal to $|\mathrm{Aut}(\Tc)|/|\mathrm{Aut}(\Tc,N_\nf)|$, which proves (\ref{eq.hepp_2}).

\textbf{Step 2.} For the rest of the proof we focus on (\ref{eq.vol_2})--(\ref{eq.vol_2.1}). Let $r$ be the number of leaves without multiplicity, and consider all the admissible mappings $(z_\lf):\Lc\to\Zb_M^4$ in Definition \ref{def.hepp_3}. By definition, for each admissible $(z_\lf)$ and each $\pi\in\mathrm{Aut}(\Tc,N_\nf)$, the mapping $(z_{\pi(\lf)})$ is also admissible, and for different $\pi$ the $(z_{\pi(\lf)})$ are also different (as the $z_\lf$ are distinct). This leads to
\begin{equation}\label{eq.new_2}\sum_{\textrm{$(z_\lf)$ admissible}}1\geq|\mathrm{Aut}(\Tc,N_\nf)|\cdot\sum_{Z\leftrightarrow(\Tc,N_\nf)}1.
\end{equation} Note also that, if $z_0\in Z$ (or $z_0,z_1\in Z$) we may fix leaves $\ff_0$ (or $\ff_0,\ff_1$) such that $z_{\ff_j}=z_j$ are fixed, with a loss of $O(r^2)\leq C^n$.

As such, to prove (\ref{eq.vol_2})--(\ref{eq.vol_2.1}) we only need to prove (for fixed leaves $(\ff_0,\ff_1)$ and fixed $(z_0,z_1)$) that
\begin{equation}\label{eq.count_2}\sum_{\substack{\textrm{$(z_\lf)$ admissible}\\z_{\ff_0}=z_0}}1\leq C^n\cdot|\mathrm{Aut}(\Tc)|\cdot\prod_{\nf\in\Bc} N_\nf^{4(\gamma_\nf-1)},
\end{equation}
\begin{equation}\label{eq.count_2.1}\sum_{\substack{\textrm{$(z_\lf)$ admissible}\\z_{\ff_0}=z_0,z_{\ff_1}=z_1}}1\leq C^n\cdot\langle z_0-z_1\rangle^{-4}\cdot|\mathrm{Aut}(\Tc)|\cdot\prod_{\nf\in\Bc} N_\nf^{4(\gamma_\nf-1)}.
\end{equation}

\textbf{Step 3.} In this step we prove an auxiliary estimate, namely that \footnote{As examples, if the root of $\Tc$ has $r$ children which are all leaves, $|\mathrm{Aut}(\Tc)|=\prod_{\nf\in\Bc}\gamma_\nf! = r!$; if $\Tc$ is a binary tree and every branching node has at least one child being a leaf, $|\mathrm{Aut}(\Tc)|=2$, $\prod_{\nf\in\Bc}\gamma_\nf!=2^{r-1}$. }
\begin{equation}\label{eq.count_3}
|\mathrm{Aut}(\Tc)|\geq C^{-r}\prod_{\nf\in\Bc}\gamma_\nf!.
\end{equation} 

To prove (\ref{eq.count_3}), for each $\nf\in\Bc$ consider possible linear orderings $\prec_\nf$ of the children of $\nf$; let the set of combinations $(\prec_\nf)$ of all such linear orderings be $\mathrm{Ord}(\Tc)$, then $|\mathrm{Ord}(\Tc)|=\prod_{\nf\in\Bc}\gamma_\nf!$. Now $\mathrm{Aut}(\Tc)$ naturally acts on $\mathrm{Ord}(\Tc)$ by defining \[(\pi\cdot\prec)_\nf=\prec_\nf';\qquad\pf\prec_{\nf}'\qf\Leftrightarrow \pi^{-1}(\pf)\prec_{\pi^{-1}(\nf)}\pi^{-1}(\qf),\] and this action is free, because if $\pi$ preserves a given $(\prec_\nf)$, we can start from the root $\rf$ and keep taking children in the given order, to prove inductively that $\pi$ must be identity.

From this, we know that the set of orbits of $\mathrm{Ord}(\Tc)$ under this action has size $(\prod_{\nf\in\Bc}\gamma_\nf!)\cdot|\mathrm{Aut}(\Tc)|^{-1}$. Moreover, for each $(\prec_\nf)\in\mathrm{Ord}(\Tc)$, we can construct a tree $\Tc'$ by listing the children (and associated subtrees) of each branching node $\nf$ from left to right according to $\prec_\nf$. Then two combinations of linear orderings $(\prec_\nf)$ and $(\prec_\nf')$ yield the same tree $\Tc'$ if and only if $(\prec_\nf)$ and $(\prec_\nf')$ differ by an element of $\mathrm{Aut}(\Tc)$, so the above gives an injection from the set of orbits of $\mathrm{Ord}(\Tc)$ to the set of trees $\Tc'$ with $r$ leaves and each branching node having $\geq 2$ children. Since the latter set has size $\leq C^r$ by Lemma \ref{lem.hepp_2}, this proves (\ref{eq.count_3}).

\textbf{Step 4.} In this step we prove two auxiliary results for the distribution of points $(z_\lf)$ for all $\lf\in\Tc$. Define
\begin{equation}\label{eq.count_4}\widetilde{N}_\nf:=\sum_{\mf\leq\nf}\gamma_\mf\cdot N_\mf.
\end{equation} Suppose $(z_\lf)$ satisfy Definition \ref{def.hepp_3} (b), then we have that
\begin{enumerate}[{(a)}]
\item $|z_\lf-z_{\lf'}|\leq \widetilde{N}_\nf$ for any leaves $\lf,\lf'\in\Lc_\nf$;
\item For any $R$, the set $\{z_\lf:\lf\in\Lc_\nf\}$ can be covered by at most $3(1+\widetilde{N}_\nf/R)$ balls of radius $R$.
\end{enumerate}

To prove these results, note that if a branching node $\mf\leq\nf$ has two children $\mf_i$ and $\mf_j$ which are linked as in Definition \ref{def.hepp_3} (b), then there exist leaves $\lf\in\Lc_{\mf_i}$ and $\lf'\in\Lc_{\mf_j}$ such that $|z_\lf-z_{\lf'}|\leq N_\mf$. We may then draw a new edge between $\lf$ and $\lf'$ (and the corresponding line segment between $z_\lf$ and $z_{\lf'}$ in $\Rb^4$) which has length $\leq N_\mf$. These new edges turn the \emph{set of leaves} $\lf\in\Lc_\nf$ into a new tree (whose nodes are represented by those points $z_\lf\in\Zb^4\subseteq\Rb^4$) where each new edge (represented by the line segment in $\Rb^4$) has a length, and the \emph{total length} of all new edges is at most $\sum_{\mf\leq\nf}(\gamma_\mf-1)N_\mf\leq \widetilde{N}_\nf$. With this construction, (a) is now trivial by triangle inequality, because any two points $z_\lf$ and $z_{\lf'}$ are connected by a sequence of line segments (as the new edges form a tree) whose total length is at most $\widetilde{N}_\nf$.

To prove (b), choose a maximal subset $Q\subseteq\{z_\lf\}$ such that the distance between any two points is $\geq R$, then by definition $\{z_\lf\}$ is covered by the balls of radius $R$ centered at points in $Q$, so we only need to prove $|Q|\leq 3(1+\widetilde{N}_\nf/R)$. To see this, draw a ball of radius $R/3$ centered at each $z\in Q$, then these balls are disjoint by the definition of $Q$. Since all the $z\in Q$ are connected by the line segments, for each $z\in Q$ there must exist part of the line segments connecting $z$ to a point on the boundary of this ball (unless $|Q|=1$ in which case there is nothing the prove). These line segments are disjoint and with total length $\geq |Q|\cdot R/3$ by triangle inequality, so $|Q|\cdot R/3\leq \widetilde{N}_\nf$ and thus $|Q|\leq 3\cdot \widetilde{N}_\nf/R$, as desired.

\textbf{Step 5.} Now we complete the proof of (\ref{eq.count_2})--(\ref{eq.count_2.1}). We proceed by an inductive construction. Define $\Jc_0(\Tc,N_\nf)$ and $\Jc_{01}(\Tc,N_\nf)$ to be the left hand side of (\ref{eq.count_2}) and (\ref{eq.count_2.1}) associated with $(\Tc,N_\nf)$ (we omit the subscript $(\ff_0,\ff_1)$ for simplicity). Let the root of $\Tc$ be $\rf$, and the children of $\rf$ be $\rf_1,\cdots,\rf_{q}$ (with $q=\gamma_\rf$). Let $\ff_0\in \Lc_{\rf_a}$ and $\ff_1\in\Lc_{\rf_b}$, we define $\lf_a=\ff_0$ and $\lf_b=\ff_1$; for $j\not\in\{a,b\}$ define $\lf_j$ to be the leftmost leaf in $\Lc_{\rf_j}$. Let the links between $\rf_j$ form a tree (if not then take a spanning tree) $\Hc$, we may regard $\Hc$ as fixed below. We then fix the sequence $(j_1,\cdots,j_q)$ as follows: $j_1=a$, and $j_{s}$ is such that $\rf_{j_s}$ has exactly one link with $(\rf_{j_t}:t<s)$.

Now, to bound $\Jc_0(\Tc,N_\nf)$, we count the number of choices of $(z_\lf)$. With $z_{\lf_a}=z_0$ fixed, we first determine $z_\lf$ for each $\lf\in\Lc_{\rf_a}$, which has $\Jc_0(\Tc_{\rf_a},N_\nf)$ choices. Suppose the values of $z_\lf$ for each $\lf\in\Lc_{\rf_{j_t}}\,(t<s)$ have been fixed, we next determine the values of $z_{\lf_{j_s}}$. By definition, there exists $t<s$ and a link between $\rf_{j_s}$ and $\rf_{j_t}$. For simplicity write $j_t:=i$ and $j_s:=j$, then using \textbf{Step 4} (a) and noting that $\lf_{j}\in\Lc_{\rf_{j}}$, we see that $z_{\lf_{j}}$ belongs to a ball of radius $N_\rf+\widetilde{N}_{\rf_{j}}$, centered at some point in $\{z_\lf:\lf\in\Lc_{\rf_{i}}\}$. Note that the values $z_\lf$ for $\lf\in\Lc_{\rf_{i}}$ have been fixed, using also \textbf{Step 4} (b), we see that $z_{\lf_{j}}$ belongs to the union of at most $3(1+\widetilde{N}_{\rf_i}/N_\rf)$ \emph{fixed} balls of radius $2(N_\rf+\widetilde{N}_{\rf_{j}})$. Therefore the number of choices of $z_{\lf_j}$ is at most 
\begin{equation}\label{eq.count_5}
CN_\rf^4\cdot\bigg(1+\frac{\widetilde{N}_{\rf_i}}{N_\rf}\bigg)\bigg(1+\frac{\widetilde{N}_{\rf_j}}{N_\rf}\bigg)^4.
\end{equation} Once $z_{\lf_j}$ is fixed, we then proceed to determine the values of $z_\lf$ for each $\lf\in\Lc_{\rf_j}$, which has $\Jc_0(\Tc_{\rf_j},N_\nf)$ choices. We next determine the value of $z_{\lf_{j_{s+1}}}$, and so on. Altogether we get that
\begin{equation}\label{eq.count_6}
\Jc_0(\Tc,N_\nf)\leq\prod_{j=1}^q\Jc_0(\Tc_{\rf_j},N_\nf)\cdot C^{q-1}N_\rf^{4(q-1)}\cdot\prod_{j=1}^q\bigg(1+\frac{\widetilde{N}_{\rf_j}}{N_\rf}\bigg)^4\cdot\sum_{\Hc}\prod_{i=1}^q\bigg(1+\frac{\widetilde{N}_{\rf_i}}{N_\rf}\bigg)^{d(i)},
\end{equation} where the summation is taken over all link trees $\Hc$, and $d(i)$ is the degree of $\rf_i$ in $\Hc$. Here note that each factor $(1+ \widetilde{N}_{\rf_j}/N_\rf)^4$ occurs at most once in (\ref{eq.count_6}), because the values $j=j_s$ in (\ref{eq.count_5}) is different for each iteration; moreover each factor $(1+ \widetilde{N}_{\rf_i}/N_\rf)$ occurs at most $d(i)$ times in (\ref{eq.count_6}), as this factor occurs in (\ref{eq.count_5}) only if $i=j_t$ is such that $\rf_i$ has a link with $\rf_j=\rf_{j_s}$ in $\Hc$.

Now we proceed with the factors in (\ref{eq.count_6}). We may replace $d(i)$ by $d(i)-1$ in the summation in (\ref{eq.count_6}), and replace $(1+ \widetilde{N}_{\rf_j}/N_\rf)^4$ in (\ref{eq.count_6}) by $(1+ \widetilde{N}_{\rf_j}/N_\rf)^5$. First
\[\prod_{j=1}^q\bigg(1+\frac{\widetilde{N}_{\rf_j}}{N_\rf}\bigg)^5\leq \exp\bigg(5\sum_{j=1}^q\frac{\widetilde{N}_{\rf_j}}{N_\rf}\bigg),\] then by the weighted Cayley Theorem we have
\[\sum_\Hc\prod_{i=1}^q\bigg(1+\frac{\widetilde{N}_{\rf_i}}{N_\rf}\bigg)^{d(i)-1}
=\bigg(q+\sum_{i=1}^q\frac{\widetilde{N}_{\rf_i}}{N_\rf}\bigg)^{q-2}
\leq
q^q\bigg(1+\frac{1}{q}\sum_{i=1}^q\frac{\widetilde{N}_{\rf_i}}{N_\rf}\bigg)^{q}\leq q^q\exp\bigg(\sum_{j=1}^q\frac{\widetilde{N}_{\rf_j}}{N_\rf}\bigg).\] Putting together (using also $q\geq 2$), we get that
\begin{equation}\label{eq.count_7}
\Jc_0(\Tc,N_\nf)\leq\prod_{j=1}^q\Jc_0(\Tc_{\rf_j},N_\nf)\cdot q!\cdot C^{q-1}N_\rf^{4(q-1)}\cdot \exp\bigg(6\sum_{j=1}^q\frac{\widetilde{N}_{\rf_j}}{N_\rf}\bigg).
\end{equation} 

\textbf{Step 6.} In \textbf{Step 5} we discussed how to bound $\Jc_0(\Tc,N_\nf)$. Bounding $\Jc_{01}(\Tc,N_\nf)$ is similar. Recall $\ff_0\in \Lc_{\rf_a}$ and $\ff_1\in\Lc_{\rf_b}$; if $a=b$ then we get the same (\ref{eq.count_7}) but with $\Jc_{0}(\Tc_{\rf_a},N_\nf)$ replaced by $\Jc_{01}(\Tc_{\rf_a},N_\nf)$. If $a\neq b$, then at step $s$ where $j_s=b$, the number of choices for $z_{\lf_b}$ is $1$ instead of (\ref{eq.count_5}) because $z_{\lf_b}=z_1$ is fixed. This leads to the same (\ref{eq.count_7}) but with $N_\rf^{4(q-1)}$ replaced by $N_\rf^{4(q-1)}\cdot N_\rf^{-4}$.

By iteration, it then follows that (where $r$ is the number of leaves without multiplicity)
\begin{equation}\label{eq.count_8}\Jc_0(\Tc,N_\nf)\leq C^{r-1}\prod_{\nf\in\Bc}\gamma_\nf!\cdot\prod_{\nf\in\Bc} N_\nf^{4(\gamma_\nf-1)}\cdot\exp\bigg(6\sum_{\nf\in\Bc}\sum_{j}\frac{\widetilde{N}_{\nf_j}}{N_\nf}\bigg),
\end{equation} where $\nf_j$ in the last summation are all the child nodes of $\nf$ (if some $\nf_j$ is a leaf then the corresponding term disappears). Note that $\prod_\nf \gamma_\nf!\leq C^r|\mathrm{Aut}(\Tc)|$ by \textbf{Step 3}, and (note that $\gamma_\mf\geq 2$)
\begin{equation}\label{eq.hepp_sum}\sum_{\nf\in\Bc}\sum_{j}\frac{\widetilde{N}_{\nf_j}}{N_\nf}=\sum_{\nf\in\Bc}\sum_j\sum_{\mf\leq \nf_j}\frac{\gamma_\mf\cdot N_\mf}{N_\nf}=\sum_{\nf\in\Bc}\sum_{\mf<\nf}\frac{\gamma_\mf\cdot N_\mf}{N_\nf}=\sum_{\mf\in\Bc}\gamma_\mf\sum_{\nf>\mf}\frac{N_\mf}{N_\nf}\leq \sum_{\mf\in\Bc}\gamma_\mf\leq \sum_{\mf\in\Bc}2(\gamma_\mf-1)\leq 2r.\end{equation} Here, note that for fixed $\mf$, all the ancestors $\nf>\mf$ form an increasing chain, and the corresponding $N_\nf$ form a strictly increasing chain of powers of $2$ (recall that $N_{\nf^+}>N_\nf$ in Definition \ref{def.hepp}), so $\sum_{\nf>\mf}\frac{N_\mf}{N_\nf}\leq 1$ by geometric summation. Putting together, we get
\begin{equation}\label{eq.count_9}\Jc_0(\Tc,N_\nf)\leq C^{r-1}C^re^{12r}\cdot|\mathrm{Aut}(\Tc)|\cdot\prod_{\nf\in\Bc} N_\nf^{4(\gamma_\nf-1)},
\end{equation} which proves (\ref{eq.count_2}).

Finally, for $\Jc_{01}(\Tc,N_\nf)$ we may assume $\ff_0\neq\ff_1$. This can be bounded by the same iteration, except that we have one gain $N_\nf^{-4}$ at exactly one step of the iteration (namely, when we look at the tree rooted at $\ff_0\vee\ff_1$), so instead of (\ref{eq.count_9}) we get
\begin{equation}\label{eq.count_10}\Jc_{01}(\Tc,N_\nf)\leq C^{r-1}C^re^{12r}\cdot|\mathrm{Aut}(\Tc)|\cdot\prod_{\nf\in\Bc} N_\nf^{4(\gamma_\nf-1)}\cdot N_{\ff_0\vee \ff_1}^{-4}.
\end{equation} But by \textbf{Step 4} (a) we have \[\langle z_0-z_1\rangle=\langle z_{\ff_0}-z_{\ff_1}\rangle\leq 1+\widetilde{N}_{\ff_0\vee\ff_1}\leq O(r^2)\cdot N_{\ff_0\vee\ff_1},\] and the $O(r^2)$ factor can be absorbed by $C^r\leq C^n$, so this proves (\ref{eq.count_2.1}).

{\it Example.} To illustrate the counting in \textbf{Step 5}, fix $\Tc$ where $\rf$ has three children  $\rf_1,\rf_2,\rf_3$, and each $\rf_i$ has $6$ children all being leaves. Fix $N_{\rf}=8$, $N_{\rf_i}=4$.  
Let $(\lf_1=f_0,\lf_2,\lf_3)$ be  the given or the leftmost leaves. Let the link tree $\Hc$ be $\rf_2 - \rf_1 - \rf_3$.
Let $(j_1,j_2,j_3)=(1,2,3)$, namely, we first ``place'' the $6$ leaves of $\Tc_{\rf_1}$ into $\Zb^4$, then $\Tc_{\rf_2}$, and then $\Tc_{\rf_3}$. In the following, suppose the values of $z_{\lf}$ for each $\lf\in \Lc_{\rf_1}$ have been fixed (the left cluster), 
covered by at most $3(1+\frac{\widetilde N_{\rf_1}}{N_\rf})$ (showing two in the picture) balls of radius $N_{\rf}$;
we next determine $z_{\lf_2}$. With $\Hc$, there exists some $z_{\lf_2'}$  ($\lf_2'\in \Lc_{\rf_2}$), within distance $N_\rf$ from some point in the left cluster, and $|z_{\lf_2} - z_{\lf_2'}| \le \widetilde N_{\rf_2}$. Clearly, the number of choices of $z_{\lf_2}$ is at most \eqref{eq.count_5} (with $i=1$, $j=2$ therein).
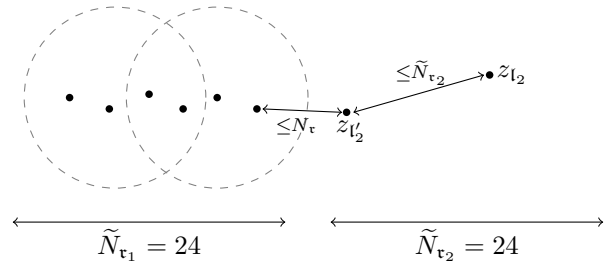
\[
\begin{tikzpicture}[scale=0.15]
\node[dot] at (-19,1) {};
\node[dot] at (-15.5,0) {};
\node[dot] at (-12,1.3) {};
\node[dot] at (-9,0) {};
\node[dot] at (-6,1) {};
\node[dot] at (-2.5,0) {};
\draw[gray,dashed] (-15,1) circle (8);
\draw[gray,dashed] (-6,1) circle (8);
\draw[<->] (-24,-10) -- (0,-10); \node at (-12,-12) {$\widetilde N_{\rf_1}=24$};
\draw[<->] (4,-10) -- (28,-10); \node at (16,-12) {$\widetilde N_{\rf_2}=24$};

\node[dot] at (5.4, -0.3) {};  \node at (5.8, -2) {$z_{\lf_2'}$};  
\draw[<->] (-2,0) -- (5, -0.3); 
\node at (1, -1.5) {\scriptsize $\le \!\! N_{\rf}$};
\node[dot] at (18, 3) {}; 
\draw[<->] (6,-0.3) -- (17.5, 3); 
\node at (12, 3.2) {\scriptsize $\le \!\! \widetilde N_{\rf_2}$};
\node at (20, 3) {$z_{\lf_2}$};  
\end{tikzpicture}
\]

\subsection{Proof of Proposition \ref{prop.main_3}}\label{sec.proof_main_3}  Recall that $\gamma_\nf$ in (\ref{eq.vol_3}) is the number of children of $\nf$ (Definition \ref{def.hepp_2}) without counting multiplicities of leaves. We first introduce some variants of it where full or partial multiplicities are counted. Also, the proof of Proposition \ref{prop.main_3} is done by induction, where we will reduce $\Tc$ to smaller trees. In this process we will need to introduce certain new notations, including the \emph{simple} and \emph{compound} leaves, and associated quantities.
\begin{definition}\label{def.mul} Consider a Hepp tree $\Tc$ with multiplicities $m_\lf$. We shall divide its leaves into two \emph{types}, \emph{simple} and \emph{compound}, and denote the simple and compound leaf sets by $\Lc^S$ and $\Lc^C$. Define the quantities
\begin{equation}\label{eq.mul_1}\gamma_\nf^2:=\gamma_\nf+\sum_{\lf\in\Lc^S\cap\Cc_\nf}1+\sum_{\lf\in\Lc^C\cap\Cc_\nf}(m_\lf-1);\qquad \gamma_\nf^\infty:=\gamma_\nf+\sum_{\lf\in\Lc\cap\Cc_\nf}(m_\lf-1)
\end{equation} (recall $\Cc_\nf$ is the set of children of $\nf$). In other words, $\gamma_\nf^2$ is the number of children of $\nf$ with each simple leaf counted with \emph{multiplicity 2} and each compound leaf counted with full multiplicities $m_\lf$, while $\gamma_\nf^\infty$ is the number of children of $\nf$ with each leaf (simple or compound) counted with full multiplicities $m_\lf$.

For the initial Hepp tree in Proposition \ref{prop.main_3} we take $\Lc^C=\varnothing$, so in $\gamma_\nf^2$ we just count every leaf with multiplicity $2$. Note that in Proposition \ref{prop.main_3} the multiplicity $m_\lf$ is always even; throughout the proof below, we will introduce new leaves whose multiplicity $m_\lf$ may not be even, but they will always be at least two (Definition \ref{def.hepp}). In particular, we always have $\gamma_\nf^2\leq\gamma_\nf^\infty$ in (\ref{eq.mul_1}).
\end{definition}

It is easy to verify that, for $\Lc^C=\varnothing$, the product $\prod_{\nf\in\Bc} N_\nf^{-4(\gamma_\nf-1)}\cdot N_\rf^{-2}$ in (\ref{eq.vol_3}) can be rewritten as
\begin{equation}\label{eq.sum_1}\prod_{\nf\in\Bc} N_\nf^{-4(\gamma_\nf-1)}\cdot N_\rf^{-2}=\prod_{\nf\in\Bc}N_\nf^{-2(\gamma_\nf^2-1)}\prod_{\nf\in\Bc\backslash\{\rf\}}\bigg(\frac{N_\nf}{N_{\nf^+}}\bigg)^2
\end{equation} (the meaning of each factor in (\ref{eq.sum_1}) will become clear later in the proof). The proof of (\ref{eq.sum_1}) follows from counting the power of each $N_\nf$ on both sides, which are
\[-4(\gamma_\nf-1)-2\cdot\mathbf{1}_{\nf=\rf}\quad\textrm{for LHS};\qquad-2(\gamma_\nf-1)-2\cdot\#(\Lc\cap\Cc_\nf)+2-2\cdot\mathbf{1}_{\nf=\rf}-2\cdot\#(\Bc\cap\Cc_\nf)\quad\textrm{for RHS},\] and noticing that they are equal because $\gamma_\nf=|\Cc_\nf|$.

In the proof of Proposition \ref{prop.main_3} below, we will start from the initial Hepp tree (with all simple leaves), and each time collapse a suitable subtree rooted at a branching node $\nf$ into a \emph{compound} leaf with certain multiplicity. Therefore, to prove Proposition \ref{prop.main_3} inductively, we need to generalize it to a version which allows for compound leaves; such generalization is stated below.

\begin{prop}\label{prop.main2_2}
Let $(\Tc,N_\nf,m_\lf)$ be a fixed marked Hepp tree with multiplicities $m_\lf\geq 2$, and with both simple and compound leaves (Definition \ref{def.mul}). Let $(y_1,\cdots,y_m)$ be a fixed vector (where $m=\sum_\lf m_\lf$), such that the multiset $\{y_1,\cdots,y_m\}=\cup_\lf \{\textrm{$m_\lf$ copies of $z_\lf$}\}$, where $(z_\lf)$ satisfies Definition \ref{def.hepp_3} (a); moreover, instead of Definition \ref{def.hepp_3} (b), we assume that for each $\nf\in\Bc$ and two leaves $\lf,\lf'\in\Lc_\nf$, we have \begin{equation}\label{eq.gap_3}|z_{\lf}-z_{\lf'}|\leq\sum_{\mf\leq\nf}\gamma_\mf^\infty\cdot N_\mf.
\end{equation}

Fix absolute constants $C_0>1000$ and $D=\exp(C_0^{10})$. Then, for any $(\Tc,N_\nf,m_\lf)$, and under all the above assumptions, we have
\begin{multline}\label{eq.vol_3_2}\sum_{\pi}^{(*)}\prod_{j=1}^{m-1}\frac{1}{\langle y_{\pi(j)}-y_{\pi(j+1)}\rangle^2}\leq C_0^mD^{|\Bc|}\sum_{W\subseteq\Bc\backslash\{\rf\}}((m-2|W|)!)^{1/2}\\\times\prod_{\lf\in\Lc^S}(m_\lf!)^{1/2}\prod_{\lf\in\Lc^C}(m_\lf!)^{3/4}\prod_{\nf\in\Bc} N_\nf^{-2(\gamma_\nf^2-1)}\prod_{\nf\in\Bc\backslash\{\rf\}}\bigg(\frac{N_\nf}{N_{\nf^+}}\bigg)^2\prod_{\nf\not\in W}\frac{N_\nf}{N_{\nf^+}}.
\end{multline} Here $\pi\in S_m$ in (\ref{eq.vol_3_2}) (which is identified with a permutation of $m_\lf$ distinct copies of each leaf $\lf$) and satisfies the primitivity condition in Proposition \ref{prop.main_3}. In addition, we restrict that there is no $j$ such that $\pi(j)$ and $\pi(j+1)$ are two copies of \emph{the same leaf}.
\end{prop}

First, let us prove that Proposition \ref{prop.main_3} follows from Proposition \ref{prop.main2_2}.
\begin{proof}[Proof of Proposition \ref{prop.main_3} assuming Proposition \ref{prop.main2_2}] We shall take $\Lc^C=\varnothing$ in Proposition \ref{prop.main2_2}, then (\ref{eq.gap_3}) follows from Definition \ref{def.hepp_3} (b) and \textbf{Step 4} in the proof of Proposition \ref{prop.hepp_3} in Section \ref{sec.proof_hepp_3} (and $\gamma_\nf\leq\gamma_\nf^\infty$). The only issue now is that in the summation in (\ref{eq.vol_3}), compared to (\ref{eq.vol_3_2}), it is possible for two copies of the same leaf to be adjacent. This can be solved by applying a simple argument: if several copies of the same leaf form a single block in the permutation, we can merge them into one big copy of this simple leaf (and decrease the multiplicity). More precisely, we argue as follows.

For each leaf $\lf$, let the set of its different copies be $A_\lf$ with $|A_\lf|=m_\lf$ (and $\sum_\lf m_\lf=2n$). Each permutation $\pi$ leads to a unique unordered partition $\Pc_\lf$ of $A_\lf$ (with the cardinalities forming an unordered partition $\pf_\lf$ of $m_\lf$), by collecting all different subsets of $A_\lf$ that form a continuous block in $\pi$. The (numerical) partition $\pf_\lf$ has $\leq 2^{m_\lf}$ choices. Once $\pf_\lf$ is fixed for each $\lf$, the permutation $\pi$ can be determined by the following inputs:
\begin{enumerate}[{(i)}]
\item The (set) partition $\Pc_\lf$ for each $\lf$, whose number of choices equals
\[\frac{m_\lf!}{a_{1}!\cdots a_{ s}!}\cdot\frac{1}{b_1!\cdots b_q!},\] where $s=s_\lf$ depends on $\lf$, and $(a_{1}, \cdots,a_{s})$ are the terms in the partition $\pf_\lf$ (so $\sum_j a_{j}=m_\lf$; these $a_j$ also depend on $\lf$ but we omit the $\lf$ subscript for simplicity), and $b_\ell$ is the number of terms $a_{j}$ that equal $\ell$ (so $\sum_\ell b_\ell=s$ and $\sum_\ell \ell b_\ell=m_\lf$).
\item 
The permutation $\widetilde{\pi}$ of the leaves of the same tree $\Tc$ with multiplicity, but with the copies of $\lf$ replaced by the \emph{sets of copies} in the partition $\Pc_\lf$ (so the multiplicity of $\lf$ is now equal to $s=s_\lf$). Moreover, in the permutation $\widetilde{\pi}$, by definition, no two copies of $\lf$ can be adjacent. \footnote{For example, consider three leaves $\alpha,\beta,\gamma$ with multiplicities $(m_\alpha,m_\beta,m_\gamma)=(4,2,2)$, and $\pi=(\alpha_1,\alpha_2,\beta_1,\gamma_1,\alpha_3,\beta_2,\gamma_2,\alpha_4)$ in which $\alpha_1,\alpha_2$ form a continuous block. Then in this case $\Pc_\alpha$ partitions $A_\alpha=\{\alpha_1,\alpha_2,\alpha_3,\alpha_4\}$ into $\{\{\alpha_1,\alpha_2\},\{\alpha_3\},\{\alpha_4\}\}$ and $\pf_\alpha$ partitions the integer $m_\alpha=4$ into $(2,1,1)$ (unordered). We have $\tilde\pi =(\tilde\alpha_{12},\beta_1,\gamma_1,\tilde\alpha_3,\beta_2,\gamma_2,\tilde\alpha_4)$ and new multiplicities $(s_\alpha,s_\beta,s_\gamma)=(3,2,2)$. Of course, there exist different choices of $\pi$ yielding different $\Pc_\alpha$ but with the same $\pf_\alpha$. Fixing $\pf_\alpha=(2,1,1)$ the number of choices of $\Pc_\alpha$ is $\frac{4!}{2!1!1!}\cdot\frac{1}{2!1!}=6$ since $(a_1,\cdots,a_{s_\alpha})=(2,1,1)$ which has $b_1=2$ incidences of $1$ and $b_2=1$ incidences of $2$; this is consistent with the fact that there are $6$ ways of choosing two elements from $A_\alpha$ to form $\Pc_\alpha$ for this fixed $\pf_\alpha$.}
\item 
Finally, a permutation of elements in each set $B\in\Pc_\lf$. The number of choices equals $a_1!\cdots a_s!$. Note that the product of factors in (i) and (iii) gives
\begin{equation}\label{eq.chain_7}\frac{m_\lf!}{b_1!\cdots b_q!}= \frac{m_\lf!}{s_\lf!}\cdot\frac{s!}{b_1!\cdots b_q!}\leq 2^{m_\lf}\frac{m_\lf!}{s_\lf!}
\end{equation} by the same proof as (\ref{eq.perm}) (with $2n$ and $r$ in (\ref{eq.perm}) replaced by $m_\lf$ and $s$).
\end{enumerate}

Putting altogether, we see that (note also that $\langle y_{\pi(j)}-y_{\pi(j+1)}\rangle=1$ if $\pi(j)$ and $\pi(j+1)$ are copies of the same leaf):
\begin{equation}\label{eq.chain_71}
\textrm{LHS of (\ref{eq.vol_3})}\leq 4^n\sup_{(\pf_\lf)}\prod_{\lf\in\Lc}\frac{m_\lf!}{s_\lf!}\cdot\sum_{\widetilde{\pi}}^{(*)}\prod_{j=1}^{m-1}\frac{1}{\langle y_{\widetilde{\pi}(j)}-y_{\widetilde{\pi}(j+1)}\rangle^2},
\end{equation} where $\widetilde{\pi}$ is the same as $\pi$ but with the multiplicity of $\lf$ being $s_\lf$ instead of $m_\lf$, and $m=\sum_\lf s_\lf$ (so $\widetilde{\pi}\in S_m$), and $(*)$ ensures no two copies of the same leaf in $\widetilde{\pi}$ are adjacent. Moreover, $\pi$ being primitive implies that $\widetilde{\pi}$ is also primitive (as each copy of $\lf$ in $\widetilde{\pi}$ is a continuous block of copies of $\lf$ in $\pi$); also since all copies of $\lf$ cannot form a single block in $\pi$, we know that $s_\lf\geq 2$ for each $\lf$.

Now, by applying Proposition \ref{prop.main2_2} with $m_\lf$ replaced by $s_\lf$, we get
\begin{multline}\label{eq.chain_6copy}
\sum_{\widetilde{\pi}}^{(*)}\prod_{j=1}^{m-1}\frac{1}{\langle y_{\widetilde{\pi}(j)}-y_{\widetilde{\pi}(j+1)}\rangle^2}\leq C^{m}\sup_{(\pf_\lf)}\sum_{W\subseteq\Bc\backslash\{\rf\}}((m-2|W|)!)^{1/2}\\
\times\prod_{\lf\in\Lc}(s_\lf!)^{1/2}\prod_{\nf\in\Bc} N_\nf^{-2(\gamma_\nf^2-1)}\prod_{\nf\in\Bc\backslash\{\rf\}}\bigg(\frac{N_\nf}{N_{\nf^+}}\bigg)^2\prod_{\nf\not\in W}\frac{N_\nf}{N_{\nf^+}}.
\end{multline} Note that the powers of $N_\nf$ on the right hand side match those in (\ref{eq.vol_3}) in view of (\ref{eq.sum_1}), so we only need to compare the factorial factors for fixed choices of $(\pf_\lf)$ and $W$. For this, note that $m=\sum_\lf s_\lf$ and $2n=\sum_\lf m_\lf$; moreover from $|W|\leq |\Bc|\leq r$ (where $r$ is the number of leaves without multiplicity) and $s_\lf\geq 2$ we know that $m\geq 2|W|$. We then have (using $\binom{a}{b}\leq 2^a$)
\begin{align}
((m-2|W|)!)^{1/2}\prod_\lf (s_\lf!)^{1/2}\prod_\lf\frac{m_\lf!}{s_\lf!}&=((2n-2|W|)!)^{1/2}\prod_\lf (m_\lf!)^{1/2}\cdot\prod_\lf\bigg(\frac{m_\lf!}{s_\lf!}\bigg)^{1/2}\cdot\bigg(\frac{(m-2|W|)!}{(2n-2|W|)!}\bigg)^{1/2}\nonumber\\
&\leq C^n\cdot(n-|W|)!\prod_\lf (m_\lf!)^{1/2}\cdot\bigg(\prod_\lf (m_\lf-s_\lf)!\cdot\frac{(m-2|W|)!}{(2n-2|W|)!}\bigg)^{1/2}\nonumber\\
\label{eq.sum_2}&\leq C^n\cdot(n-|W|)!\prod_\lf (m_\lf!)^{1/2}.
\end{align}
By putting together (\ref{eq.chain_71}), (\ref{eq.chain_6copy}) and (\ref{eq.sum_2}), we have proved (\ref{eq.vol_3}).
\end{proof}
\subsubsection{The inductive step}\label{sec.induct} For the rest of this subsection we prove Proposition \ref{prop.main2_2}. As discussed above, we shall induct by collapsing a suitable subtree $\Tc_\nf$ into a compound leaf with certain multiplicity, thus reducing $\Tc$ to a smaller tree $\Tc'$. For this to work, we still need an estimate controlling the summation within $\Tc_\nf$, which is provided below.

\begin{prop}\label{prop.main_4} Recall $C_0$ in Proposition \ref{prop.main2_2}. In the same setting there, let $(\Tc,N_\nf,m_\lf)$, and the simple/compound leaf types, be fixed. Also fix $(y_1,\cdots,y_m)$ such that $m=\sum_\lf m_\lf$ and the multiset $\{y_1,\cdots,y_m\}=\cup_\lf \{\textrm{$m_\lf$ copies of $z_\lf$}\}$, where $(z_\lf)$ satisfies Definition \ref{def.hepp_3} (a). In addition, assume that
\begin{enumerate}[{(a)}]
\item For each $\nf\in\Bc\backslash\{\rf\}$ we have (recall $\gamma_\nf^\infty$ and $\gamma_\nf^2$ defined in (\ref{eq.mul_1}))
\begin{equation}\label{eq.local_2}
N_{\nf^+}\leq 8N_\nf':=8\sum_{\mf\leq\nf}\gamma_\mf^\infty\cdot N_\mf;
\end{equation}
\item We also fix a dyadic number $R$ such that $R\geq N_\rf'$ where $\rf$ is the root of $\Tc$ and $N_\rf'$ is defined as in (\ref{eq.local_2});
\item We also fix a set of \emph{skipped} indices $O\subseteq\{1,\cdots,m-1\}$, with $|O|=s$.
\end{enumerate} Then, we have
\begin{multline}\label{eq.local_3}\sum_{\pi}^{(*)}\prod_{\substack{j=1\\j\not\in O}}^{m-1}\frac{1}{\langle y_{\pi(j)}-y_{\pi(j+1)}\rangle^2}\leq C_0^{m/2}((m-s)!)^{1/2}\prod_{\lf\in\Lc^S}(m_\lf!)^{1/2}\prod_{\lf\in\Lc^C}(m_\lf!)^{3/4}\cdot\prod_{\nf\in\Bc} N_\nf^{-2(\gamma_\nf^2-1)}\cdot R^{2s}\\
\times\prod_{\nf\in\Bc\backslash\{\rf\}}\bigg(\frac{N_\nf}{N_{\nf^+}}\bigg)^{3}\cdot (N_\rf/R)^{\min(2s,3)}.
\end{multline} Here $\pi\in S_m$ (identified with a permutation of $m_\lf$ distinct copies of each leaf $\lf$) and satisfies that, there is no $j\not\in O$ such that $\pi(j)$ and $\pi(j+1)$ are
two copies of the same leaf. We do not require $\pi$ to be primitive.
\end{prop}
\begin{proof}[Proof of Proposition \ref{prop.main2_2} assuming Proposition \ref{prop.main_4}] We prove by induction. Suppose Proposition \ref{prop.main2_2} is true for smaller trees $\Tc'$ (the base case follows from the same argument), now we prove it for $\Tc$. This is divided into 3 steps.

\textbf{Step 1.} Given $(\Tc,N_\nf,m_\lf)$ and $(y_j)$, we choose a \emph{lowest} non-root branching node $\nf$ such that
\begin{equation}\label{eq.gap_4}\sum_{\mf\leq\nf}\gamma_\mf^\infty\cdot N_\mf\leq N_{\nf^+}/8.
\end{equation} This $\nf$ is uniquely determined by $(\Tc,N_\nf,m_\lf)$ and $(y_j)$ (if there exist more than one then we may fix one of them). Note that if $\nf$ does not exist, then we can argue similarly as the main proof below but apply Proposition \ref{prop.main_4} to the whole tree $\Tc$ with $O=\varnothing$, which will directly imply Proposition \ref{prop.main2_2} (choosing $W=\varnothing$ in (\ref{eq.vol_3_2})).

Consider the subtree $\Tc_\nf$ rooted at $\nf$, then the lowest assumption of $\nf$ (so that (\ref{eq.gap_4}) cannot hold for anything strictly lower than $\nf$) implies that (\ref{eq.local_2}) holds for each non-root branching node for this subtree. Moreover, if we choose $R=N_{\nf^+}$ in the assumption (b) below (\ref{eq.local_2}), then (\ref{eq.gap_4}) also implies that $R\geq N_\rf'$ (where $\rf=\nf$ is the root of $\Tc_\nf$).

In addition, consider any leaves $\lf,\lf'\in\Tc_\nf$ and $\yf\not\in\Tc_\nf$, then Definition \ref{def.hepp_3} (a) and (\ref{eq.gap_3}) and (\ref{eq.gap_4}) imply that \[|z_\yf-z_\lf|\geq N_{\nf^+}/2\geq 4\sum_{\mf\leq \nf}\gamma_\mf^\infty\cdot N_\mf\geq4|z_\lf-z_{\lf'}|,\] from which we have that 
\begin{equation}\label{eq.local_4}
\frac{1}{2}|z_\yf-z_{\lf}|\leq|z_\yf-z_{\lf'}|\leq 2|z_\yf-z_{\lf}|.
\end{equation}

\textbf{Step 2.} Now consider a permutation $\pi$ of all leaves in $\Lc$ (with multiplicities $m_\lf$), which is \emph{primitive} and has \emph{no adjacent copies} of the same leaf. Such a permutation can be formed in the following three steps:
\begin{enumerate}[{(1)}]
\item Choose a permutation $\pi_1$ of all leaves in $\Lc_\nf$ (with multiplicities $m_\lf$);
\item Choose a skipped set $O\subseteq\{1,\cdots,n-1\}$ with $|O|=s$ (where $n=\sum_{\lf\in\Lc_\nf}m_\lf$ is the total number of leaves in $\Lc_\nf$ with multiplicities), and break $\pi_1$ into $s+1$ blocks by separating those before $j$ and those after $j+1$ for each $j\in O$;
\item Finally, insert the leaves in $\Lc\backslash\Lc_\nf$ into the spaces between these $s+1$ blocks.
\end{enumerate}

Suppose $\pi_1$ and $O$ are fixed, we can define a tree $\Tc'$ with leaf set $\Lc'$, and a permutation $\pi_2$ of leaves in $\Lc'$ with multiplicities, as well as the values $(z_\lf)_{\lf\in\Lc'}$, as follows:
\begin{itemize}
\item First, $\Tc'$ is formed by collapsing $\Tc_\nf$ to a single leaf $\nf$, and the marking $(N_\mf)$ of $\Tc'$ is inherited from $\Tc$. The multiplicity and type for leaves in $\Lc'\backslash\{\nf\}$ are inherited from $\Tc$, and $\nf$ is defined to be a compound leaf in $\Tc'$, with multiplicity $s+1$.
\item The permutation $\pi_2$ of all leaves in $\Lc'$ with multiplicities, is formed by collapsing each of the $s+1$ blocks in (2) above into the single leaf $\nf\in\Lc'$. 
\item Note that in $\pi_2$, the $s+1$ different copies of $\nf$  are \emph{viewed as identical} (so switching between them does not change $\pi_2$).
\item Finally, the value of $z_\lf$ for leaves $\lf\neq\nf$ in $\Tc'$ is inherited from $\Tc$, while $z_\nf$ is defined to be $z_{\lf_*}$, where $\lf_*$ is the leftmost leaf in $\Tc_\nf$. 
\end{itemize}
We then have, see the proofs below, that:
\begin{enumerate}[{(1)}]
\item The $\pi$ is in bijection with $(\pi_1,O,\pi_2)$, and we must have $s\geq 1$ (because $\pi$ is primitive, so all the copies of leaves in $\Lc_\nf$ \emph{cannot} form one single block). 
\item For $\pi_1$ there is no $j\not\in O$ such that $\pi_1(j)$ and $\pi_1(j+1)$ are two copies of the same leaf (this is because $\pi$ has no adjacent copies of the same leaf), and for $\pi_2$ there is no $j$ such that $\pi_2(j)$ and $\pi_2(j+1)$ are two copies of the same leaf (for leaves other than $\nf$ this follows from the same assumption for $\pi$, while for the compound leaf $\nf$ this follows from the definition of $O$). Moreover $\pi_2$ is primitive (as $\pi$ is).
\item The vectors $(z_\lf):\lf\in\Lc'$ satisfy Definition \ref{def.hepp_3} (a) and (\ref{eq.gap_3}) for $\Tc'$.
\item If we define
\begin{equation}\label{eq.local_5}
\Pi[(y_j),\pi,O]:=\prod_{j\not\in O}\langle y_{\pi(j)}-y_{\pi(j+1)}\rangle^{-2}
\end{equation} as in the summand on the left hand side of (\ref{eq.local_3}), then we have
\begin{equation}\label{eq.local_6}\Pi[(y_j),\pi,\varnothing]\leq 4^n\cdot\Pi[(y_j^1),\pi_1,O]\cdot\Pi[(y_j^2),\pi_2,\varnothing],
\end{equation} where $(y_j)$, $(y_j^1)$ and $(y_j^2)$ are the vectors corresponding to the trees $\Tc$, $\Tc_\nf$ and $\Tc'$ and the permutations $\pi$, $\pi_1$ and $\pi_2$ respectively.
\end{enumerate}

In fact, in the above statements, (1)--(2) are already evident from above, so we will only prove (3)--(4). For (3), Definition \ref{def.hepp_3} (a) is obvious because $(z_\lf)$ for $\Lc'$ is essentially inherited from that of $\Lc$ (and that $z_\nf$ is defined to be $z_{\lf_*}$ while also $\lf_*\in\Tc_\nf$ in $\Tc$). As for (\ref{eq.gap_3}) for $\Tc'$, say we need to prove it for leaves $\lf,\lf'$ below a branching node $\pf$ of $\Tc'$. If $\nf\not\in\Tc_\pf$, then  (\ref{eq.gap_3}) follows from the same estimate for $\Tc$ because of the inheritance. If $\nf\in\Tc_\pf$, then using the assumption (\ref{eq.gap_3}) for $\Tc$ together with (\ref{eq.gap_4}), we get that
\[
\begin{aligned}|z_{\lf}-z_{\lf'}|&\leq \sum_{\substack{\mf\leq\pf\\(\textrm{before})}}(\gamma_\mf^\infty)_{\mathrm{before}}\cdot N_\mf\\
&=\sum_{\substack{\mf\leq\pf\\(\textrm{after})}}(\gamma_\mf^\infty)_{\mathrm{before}}\cdot N_\mf+\sum_{\mf\leq\nf}(\gamma_\mf^\infty)_{\mathrm{before}}\cdot N_\mf\\
&\leq  \sum_{\substack{\mf\leq\pf\\(\textrm{after})}}(\gamma_\mf^\infty)_{\mathrm{before}}\cdot N_\mf+N_{\nf^+}/8\leq  \sum_{\substack{\mf\leq\pf\\(\textrm{after})}}(\gamma_\mf^\infty)_{\mathrm{after}}\cdot N_\mf,
\end{aligned}
\] because for $\mf=\nf^+$, its child $\nf$ is a branching node (contributing $1$) before, and becomes a compound leaf of multiplicity $s+1\geq 2$ after. This proves (3). 

Finally, for (4), we consider the different factors occurring on both sides of (\ref{eq.local_6}). For a permutation $\pi$ of $(y_j)$ of all leaves in $\Tc$, if $\pi(j)$ and $\pi(j+1)$ both belong to $\Lc_\nf$ or both belong to $\Lc\backslash\Lc_\nf$, then this factor occurs in the same form on both sides of (\ref{eq.local_6}). Now suppose $\pi(j)=\lf\in\Lc_\nf$ and $\pi(j+1)=\yf\in\Lc\backslash\Lc_\nf$ (or the other way round which is the same), then this occurs at most $2(s+1)\leq 2n$ times. For each occurrence, we have a factor $\langle z_\yf-z_\lf\rangle^{-2}$ on the left hand side of (\ref{eq.local_6}), and a factor $\langle z_\yf-z_{\lf_*}\rangle^{-2}$ on the right hand side of (\ref{eq.local_6}) due to our choice $z_\nf:=z_{\lf_*}$. However these two factors are comparable up to a multiple of $2$ thanks to (\ref{eq.local_4}). By making this comparison for all relevant factors (there are at most $2n$ of them), this proves (\ref{eq.local_6}).

\textbf{Step 3.} We can now prove Proposition \ref{prop.main2_2}, i.e. prove (\ref{eq.vol_3_2}) for $\Tc$. Recall $n$ is the number of leaves in $\Lc_\nf$ with multiplicity. Let $m$ be the number of leaves in $\Lc\backslash\Lc_\nf$ counted with multiplicity. By the above statements (1)--(4), we have
\begin{equation}\label{eq.local_7}
\textrm{LHS of (\ref{eq.vol_3_2})}=\sum_{\pi}^{(*)}\Pi[(y_j),\pi,\varnothing]\leq 4^n\sum_{(\pi_1,O,\pi_2)}\Pi[(y_j^1),\pi_1,O]\cdot\Pi[(y_j^2),\pi_2,\varnothing].
\end{equation} Now, for fixed $\pi_1$ and $O$, we may estimate $\sum_{\pi_2}\Pi[(y_j^2),\pi_2,\varnothing]$ using the induction hypothesis (as $\pi_2$ is primitive and has no adjacent copies of the same leaf, see (2) above); on the other hand, for fixed $O$, we can estimate $\sum_{\pi_1}\Pi[(y_j^1),\pi_1,O]$ using Proposition \ref{prop.main_4}. This then leads to
\begin{align}
\textrm{LHS of (\ref{eq.vol_3_2})}&\leq (4C_0^{1/2})^n\sum_O((n-s)!)^{1/2}\prod_{\lf\in\Lc_\nf^S}(m_\lf!)^{1/2}\prod_{\lf\in\Lc_\nf^C}(m_\lf!)^{3/4}\prod_{\mf\in\Bc_\nf}N_\mf^{-2(\gamma_\mf^2-1)}\cdot N_{\nf^+}^{2s}\prod_{\mf\in\Bc_\nf\backslash\{\nf\}}\bigg(\frac{N_\mf}{N_{\mf^+}}\bigg)^{3}\nonumber\\
&\times (N_\nf/N_{\nf^+})^{2+\mathbf{1}_{s\geq 2}}\cdot ((s+1)!)^{-1}\cdot C_0^{m+s+1}D^{|\Bc|-1}\sum_{W\subseteq\Bc\backslash (\{\rf\}\cup\Bc_\nf)}((m+s+1-2|W|)!)^{1/2}\nonumber\\
&\times\prod_{\lf\in\Lc^S\backslash\Lc_\nf^S}(m_\lf!)^{1/2}\cdot((s+1)!)^{3/4}\prod_{\lf\in\Lc^C\backslash\Lc_\nf^C}(m_\lf!)^{3/4}\prod_{\mf\in\Bc\backslash\Bc_\nf}N_\mf^{-2(\gamma_\mf^2-1)}\cdot N_{\nf^+}^{-2s}\nonumber\\
&\times \prod_{\mf\in\Bc\backslash (\{\rf\}\cup\Bc_\nf)}\bigg(\frac{N_\mf}{N_{\mf^+}}\bigg)^2\prod_{\mf\in\Bc\backslash (\{\rf\}\cup\Bc_\nf\cup W)}\frac{N_\mf}{N_{\mf^+}}\label{eq.local_8}
\end{align}
where $s=|O|$ and $\Lc_\nf^S=\Lc^S\cap\Lc_\nf$ and $\Lc_\nf^C=\Lc^C\cap\Lc_\nf$. Here we note that $\min(2s,3)=2+\mathbf{1}_{s\geq 2}$ for $s\geq 1$, and the followings:
\begin{enumerate}[{(i)}]
\item In $\pi_2$, the $s+1$ different copies of the leaf $\nf$ are \emph{viewed as identical}, while in considering the induction hypothesis the different copies of this leaf in $\Lc'$ are \emph{viewed as distinct}. This leads to an additional $((s+1)!)^{-1}$ factor in the sum $\sum_{\pi_2}\Pi[(y_j^2),\pi_2,\varnothing]$ compared to the induction hypothesis.
\item The total number of leaves of $\Tc'$ with multiplicity is equal to $m+s+1$, explaining the factors $C_0^{m+s+1}$ and $(m+s+1-2|W|)!$ in (\ref{eq.local_8}); the new leaf $\nf$ of $\Tc'$ is compound and has multiplicity $s+1$, explaining the factor $((s+1)!)^{-3/4}$ in (\ref{eq.local_8}). Also the number of branching nodes for $\Tc'$ is at most $|\Bc|-1$ (as $\nf$ becomes a leaf), explaining the factor $D^{|\Bc|-1}$ in (\ref{eq.local_8}).
\item Because of the new compound leaf $\nf$ in $\Tc'$ of multiplicity $s+1$, the value of $\gamma_{\nf^+}^2$ in $\Tc'$ is increased by $s$ compared to the same value in $\Tc$, explaining the factor $N_{\nf^+}^{-2s}$ in (\ref{eq.local_8}). All the other factors in (\ref{eq.local_8}) simply come from plugging in (\ref{eq.vol_3_2}) and (\ref{eq.local_3}) and are easily checked.
\end{enumerate}

Now, simplifying the right hand side of (\ref{eq.local_8}) we get
\begin{equation}\label{eq.local_9}
\begin{aligned}
\textrm{LHS of (\ref{eq.vol_3_2})}&\leq C_0^{m}(4C_0^{1/2})^nD^{|\Bc|}\cdot D^{-1}C_0^{s+1}\cdot ((s+1)!)^{-1}((s+1)!)^{3/4}\sum_O\sum_W((n-s)!)^{1/2}\\&\times(m+s+1-2|W|)^{1/2}\prod_{\lf\in\Lc^S}(m_\lf!)^{1/2}\prod_{\lf\in\Lc^C}(m_\lf!)^{3/4}\prod_{\mf\in\Bc}N_\mf^{-2(\gamma_\mf^2-1)}\prod_{\mf\in\Bc\backslash\{\rf\}}\bigg(\frac{N_\mf}{N_{\mf^+}}\bigg)^2\prod_{\mf\not\in W'}\frac{N_\mf}{N_{\mf^+}},
\end{aligned}
\end{equation} where \[W'=\left\{
\begin{aligned}
&W,&\mathrm{if\ }s>1;\\
&W\cup\{\nf\},&\mathrm{if\ }s=1.
\end{aligned}\right.\] Note that $D^{-1}C_0^{s+1}\cdot ((s+1)!)^{-1}((s+1)!)^{3/4}\leq 1$ (as $D=\exp(C_0^{10})$); moreover we have
\begin{equation}\label{eq.local_9.5}((n-s)!)^{1/2}((m+s+1-2|W|)!)^{1/2}\leq((n+m-2|W'|)!)^{1/2}\cdot n^2,\end{equation} which follows by considering the cases $s=1$ (and $|W'|=|W|+1$) and $s>1$ (and $|W'|=|W|$) separately, and using that $n\geq s+1$ and $n\geq 4$ (as $\nf$ has at least two leaves, each with multiplicity $\geq 2$). This extra factor $n^2\leq 2^n$ in (\ref{eq.local_9.5}), and the factor $2^n$ coming from summation in $O$, contributes another $4^n$, thus
\begin{multline}\label{eq.local_10}
\textrm{LHS of (\ref{eq.vol_3_2})}\leq C_0^{m}(16C_0^{1/2})^nD^{|\Bc|}\sum_{W'}((n+m-2|W'|)!)^{1/2} \\\times\prod_{\lf\in\Lc^S}(m_\lf!)^{1/2}\prod_{\lf\in\Lc^C}(m_\lf!)^{3/4}\prod_{\mf\in\Bc}N_\mf^{-2(\gamma_\mf^2-1)}\prod_{\mf\in\Bc\backslash\{\rf\}}\bigg(\frac{N_\mf}{N_{\mf^+}}\bigg)^2\prod_{\mf\not\in W'}\frac{N_\mf}{N_{\mf^+}}.
\end{multline} Using that $16C_0^{1/2}\leq C_0$, and that the total number of leaves of $\Tc$ with multiplicities equals $m+n$, we see that this implies (\ref{eq.vol_3_2}) for $\Tc$. The inductive proof is now complete.
\end{proof}

{\it Example.} Suppose that the copies of the leaves belonging to $\Tc_\nf$ are $ a_1 , b_1, a_2, b_2$
where $a_1,a_2$ are two copies of the leaf $a$, and $b_1,b_2$ are two copies of the leaf $b$.  Suppose outside $\Tc_\nf$  we have the copies $c_1,c_2,c_3$ of a leaf $c$.
Consider $ \pi=(c_1, a_1, b_1, c_2, a_2, b_2, c_3)$.
Restricting $\pi$  to $\Tc_\nf$  gives
  $ \pi_1=(a_1, b_1,a_2, b_2)$. Here, 
  $b_1,a_2$ are adjacent in $\pi_1$, but  
they are not adjacent in the original permutation
$\pi$, since the outside copy $c_2$ lies between them.  Hence the skipped set is  $O=\{2\}$,
and $O$ breaks $\pi_1$ into two consecutive blocks,
$B_1=(a_1,b_1)$  and $B_2=(a_2,b_2)$.

After collapsing  $\Tc_\nf$ into one compound leaf, each of
these two blocks becomes one copy of the new leaf $\nf$.  Therefore the
corresponding permutation of the collapsed tree $\Tc'$ is
 $ \pi_2=(c_1, \nf, c_2, \nf, c_3)$
where the new compound leaf $\nf$ has multiplicity $|O|+1=2$.
Conversely, the original permutation $\pi$ is recovered uniquely from
$(\pi_1,O,\pi_2)$ by replacing the two copies of $\nf$ in $\pi_2$ by the blocks
$B_1$ and $B_2$ respectively, or equivalently,
inserting $c_1,c_2,c_3$ into the spaces $*$ in $(*,B_1,* ,B_2,*)$.
  This illustrates the bijection $ \pi\to (\pi_1,O,\pi_2)$.

In this example, the inside factor
$ \Pi[(y_j^1),\pi_1,O]=
        \langle z_a-z_b\rangle^{-2}
        \langle z_a-z_b\rangle^{-2}$, noting that 
 the edge between $b_1$ and $a_2$ is skipped.  The collapsed permutation
produces the outside factor
\[
        \Pi[(y_j^2),\pi_2,\varnothing]=
        \langle z_c-z_\nf\rangle^{-2}
        \langle z_\nf-z_c\rangle^{-2}
        \langle z_c-z_\nf\rangle^{-2}
        \langle z_\nf-z_c\rangle^{-2}.
\]
The left picture below shows the collapse of the tree $\Tc_\nf$ into a new leaf $\nf$; the right picture shows the points in $\Zb^4$ (or rather the points after possible merging), where each line represents the kernel $\langle \;\; \rangle^{-2}$, red being inside and black being outside (before merging $a,b$ into $\nf$ using \eqref{eq.local_4}):
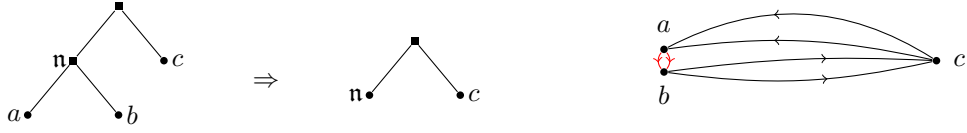
\[
\begin{tikzpicture}[scale=0.6,baseline=-10]
\node[Bdot]  (r) at (0,1.2) {};\node[Bdot] (n) at (-1,0) {};\node[dot] (c) at (1,0) {};\node[dot] (a) at (-2,-1.2) {};\node[dot]  (b) at (0,-1.2) {};
\node  at (-1.3,0) {$\nf$};\node  at (1.3,0) {$c$};
\node at (-2.3,-1.2) {$a$};\node at (0.3,-1.2) {$b$};
\draw (r) -- (n); \draw (r) -- (c); \draw (n) -- (a);\draw (n) -- (b);
\end{tikzpicture}
\qquad \Rightarrow \qquad
\begin{tikzpicture}[scale=0.6,baseline=3]
\node[Bdot]  (r) at (0,1.2) {};\node[dot] (n) at (-1,0) {};\node[dot] (c) at (1,0) {};
\node  at (-1.3,0) {$\nf$};\node  at (1.3,0) {$c$};
\draw (r) -- (n); \draw (r) -- (c); 
\end{tikzpicture}
\qquad\qquad\qquad
\begin{tikzpicture}[scale=1.5,baseline=-10]
\node[dot] (a) at (0,0.1) {}; 
\node[dot] (b) at (0,-0.1) {}; 
\node at (0,0.3) {$a$};
\node at (0,-0.3) {$b$};
\node[dot] (c) at (2.4,0) {}; 
\node at (2.6,0) {$c$};
\draw[bend right=30,midarrow] (c) to (a);
\draw[bend right=40,red,midarrow] (a) to (b);
\draw[bend right=10,midarrow] (b) to (c);
\draw[bend right=10,midarrow] (c) to (a);
\draw[bend left=40,red,midarrow] (a) to (b);
\draw[bend left=5,midarrow] (b) to (c);
\end{tikzpicture}
\]

\subsubsection{A few lemmas} For the rest of this subsection we prove Proposition \ref{prop.main_4}. We start by introducing a few lemmas which will be useful in the proof below.
\begin{definition}\label{def.lac} For positive integers $(X_1,\cdots,X_n)$, we say it is \emph{lacunary}, if any interval $[X,2X]$ contains $X_j$ for at most $O(1)$ values of $j$.
\end{definition}

\begin{lemma}\label{lem.aux1} For the binomial coefficient, we have
\begin{equation}\label{eq.aux1}
\binom{m+n}{m}\leq \alpha^m\beta^n,
\end{equation} where $\alpha>1$ is arbitrary and $\beta=\beta(\alpha)>1$ depends only on $\alpha$. Similarly, for the multinomial coefficient, we have (where $n=n_1+\cdots +n_r$)
\begin{equation}\label{eq.aux2}
\binom{n}{n_1,\cdots,n_r}=\frac{n!}{n_1!\cdots n_r!}\leq\alpha^{n_1+2n_2+\cdots +rn_r}\cdot\beta^{n_1+\cdots +n_r},
\end{equation} where again $\alpha>1$ is arbitrary and $\beta=\beta(\alpha)>1$ depends only on $\alpha$.

In particular, if the $(n_j)$ are \emph{lacunary} as in Definition \ref{def.lac}, then we have 
\begin{equation}\label{eq.aux3}
\binom{n}{n_1,\cdots,n_r}\leq C^{n_1+\cdots +n_r}.
\end{equation}
\end{lemma}
\begin{proof} In (\ref{eq.aux1}) we may assume $1<\alpha<2$. Then
\begin{equation}\frac{1}{\alpha^m}\binom{m+n}{n}=\frac{(m+1)\cdots (m+n)}{\alpha^m\cdot n!}\leq\frac{1}{n!}\cdot\frac{(m+n)^n}{\alpha^m}.\end{equation} Note that the function $m\mapsto n\cdot\log(n+m)-m\cdot\log\alpha$ has a unique point of maximum at $m=\gamma n$ where $\gamma:=(1/\log\alpha)-1$, thus
\begin{equation}\frac{1}{\alpha^m}\binom{m+n}{n}\leq\frac{1}{n!}\cdot\frac{(\gamma n+n)^n}{\alpha^{\gamma n}}\leq\big(e(\gamma+1)\alpha^{-\gamma}\big)^n\end{equation} (as $n!\geq n^ne^{-n}$), so we may choose $\beta:=e(\gamma+1)\alpha^{-\gamma}\,(=\alpha/\log\alpha)$.

Now (\ref{eq.aux2}) follows from (\ref{eq.aux1}): we have
\begin{equation}\label{eq.aux4}
\binom{n}{n_1,\cdots,n_r}=\prod_{j=1}^r\binom{n_j+\cdots +n_r}{n_j}\leq\alpha^{\sum_{j=1}^r(n_{j+1}+\cdots +n_{r})}\cdot \beta^{\sum_{j=1}^r n_j}\leq \alpha^{n_1+2n_2+\cdots +rn_r}\cdot\beta^{n_1+\cdots +n_r}.
\end{equation} Finally, for (\ref{eq.aux3}), assume $n_1\geq\cdots \geq n_r$, then by lacunary assumption we have $n_j\geq 2n_{j-C}$ for some constant $C=O(1)$, in particular $n_{j+1}+\cdots+n_r\lesssim n_j$ uniformly in $j$, therefore (\ref{eq.aux3}) follows from (\ref{eq.aux4}).
\end{proof}
\begin{lemma}\label{lem.aux2} For any $\theta>0$ and positive integers $m\leq n$, we have
\begin{equation} \label{eq.aux5}\sum_{1\leq z_1,\cdots,z_m\leq n}\prod_{j=1}^{m-1}\min(1,2^{-\theta(z_j-z_{j+1})})\leq C^n,
\end{equation} where $C=C(\theta)>1$ is an absolute constant depending on $\theta$. Similarly, if $(P_1,\cdots,P_n)$ are distinct elements of $2^\Nb$, then
\begin{equation}\label{eq.aux6}\sum_{1\leq x(1),\cdots x(n)\leq n}\prod_{j=1}^{n-1}\min\bigg(1,\bigg(\frac{P_{x(j+1)}}{P_{x(j)}}\bigg)^{\theta}\bigg)\leq C^n.
\end{equation} 
The same result is true if the product in (\ref{eq.aux6}) is taken over $j\in\{1,\cdots,n-1\}\backslash E$ for any set $E$ with at most $100$ elements.
\end{lemma}
\begin{proof} For (\ref{eq.aux5}), define $w_j:=\max(0,z_j-z_{j+1})$, then $w_j\geq 0$. We have
\begin{equation}\label{eq.aux7}
\textrm{LHS of (\ref{eq.aux5})}=\sum_{(w_1,\cdots,w_{m-1})}\prod_{j=1}^{m-1}2^{-\theta w_j}\sum_{\substack{1\leq z_1,\cdots,z_m\leq n\\\max(0,z_j-z_{j+1})=w_j}}1.
\end{equation} For fixed $(w_j)$, let us count the number of sequences $(z_j)$ in the inner sum. Define $z:=w_1+\cdots +w_{m-1}$, clearly we have $z_j-z_{j+1}\leq w_j$, thus 
\[Z_j:=z_j+(w_1+\cdots +w_{j-1})+j\] is a \emph{strictly increasing} sequence in $j$. As $1\leq Z_j\leq m+n+z$ because $z_j\leq n$ and $j\leq m$, we know the number of choices for the sequence $(Z_j)$, i.e. the number of choices for the sequence $(z_j)$, is bounded by
\[\binom{m+n+z}{m}\leq 2^{\theta(n+z)/2}C^m,\] where we have applied (\ref{eq.aux1}) with $\alpha=2^{\theta/2}$, and $C$ is an absolute constant depending on $\theta$. Then, plugging into (\ref{eq.aux7}), and note that $m\leq n$, we get
\[\textrm{LHS of (\ref{eq.aux5})}\leq (2^{\theta/2}C)^n\sum_{w_1,\cdots,w_{n-1}\geq 0}\prod_{j=1}^{n-1}2^{-\theta w_j}\prod_{j=1}^{n-1} 2^{\theta w_j/2}\leq \bigg(\frac{2^{\theta/2}C}{1-2^{-\theta/2}}\bigg)^n.\] 
Now for (\ref{eq.aux6}), we may assume $P_{j+1}\geq 2P_{j}$, then it follows from (\ref{eq.aux5}) because
\[\min\bigg(1,\frac{P_{x(j+1)}}{P_{x(j)}}\bigg)\leq \min(1,2^{-(x(j)-x(j+1))}).\] If we skip at most $100$ factors in the product (\ref{eq.aux6}), then we still have the product of at most $101$ consecutive sequences of factors, so (\ref{eq.aux6}) follows by applying (\ref{eq.aux5}) to each of these sequences.
\end{proof}
\begin{lemma}\label{lem.aux3} Let $A$ and $B$ be two finite subsets of $\mathbb{Z}^4$, and $M,N\in 2^\Nb$, such that the distance between any two points in $A$ is $\gtrsim M$, and the distance between any two points in $B$ is $\gtrsim N$. Then uniformly in $u\in\mathbb{Z}^4$ and in all parameters, we have
\begin{equation}\label{eq.aux8}\sum_{x\in A, y\in B}\mathbf{1}_{|u-x|\gtrsim M}\cdot\mathbf{1}_{|x-y|\gtrsim\max(M,N)}\cdot\frac{1}{|u-x|^2|x-y|^2}\lesssim |A|^{1/2}M^{-2}\cdot |B|^{1/2}N^{-2}\cdot\min(1,(Q/P)^{1/4}),
\end{equation} where
\begin{equation}\label{eq.aux9} P:=|A|\cdot M^4,\quad Q:=|B|\cdot N^4.
\end{equation}
\end{lemma}
\begin{proof} First, for any fixed $u\in\mathbb{Z}^4$ we have
\begin{equation}\label{eq.lem4_2}\sum_{x\in A}\mathbf{1}_{|u-x|\gtrsim M}\cdot\frac{1}{|u-x|^2}\lesssim|A|^{1/2}M^{-2}.
\end{equation} In fact, by translating $A$ we may assume $u=0$. For any dyadic number $X\in 2^{\mathbb{N}}$, using the lower bound $M$ for distances of different points in $A$ (so any $M$-box contains $O(1)$ points in $A$), we have 
\begin{equation}\#\{x\in A:|x|\gtrsim M, |x|\sim X\}\lesssim \min(|A|,(X/M)^4)\cdot \mathbf{1}_{X\gtrsim M},\end{equation} therefore, by a dyadic decomposition in $|x|$, we have
\begin{equation}\label{eq.lem4_2.5}
\sum_{x\in A}\mathbf{1}_{|x|\gtrsim M}\cdot\frac{1}{|x|^2}\lesssim \sum_{X\gtrsim M}X^{-2}\cdot \#\{x\in A:|x|\sim X\}\lesssim\sum_{X\gtrsim M} \min(|A|\cdot X^{-2},X^2/M^4)\lesssim |A|^{1/2}M^{-2}.
\end{equation} The case $Q\geq P$ of (\ref{eq.aux8}) then follows from iterating (\ref{eq.lem4_2}), first summing in $y$ and then summing in $x$.

Now we prove the case $P\geq Q$ of (\ref{eq.aux8}). Again by translation we may set $u=0$. For each $y\in B$, we will prove that
\begin{equation}\label{eq.lem4_3}
\sum_{x\in A} \mathbf{1}_{|x|\gtrsim M}\cdot \mathbf{1}_{|x-y|\gtrsim\max(M,N)}\cdot\frac{1}{|x|^2|x-y|^2}\lesssim M^{-4}\cdot\bigg(1+\frac{P^{1/4}}{\max(|y|,M,N)}\bigg).
\end{equation}

To prove (\ref{eq.lem4_3}), we consider two cases in the summation, where (a) $|x|\lesssim|y|$ or (b) $|x|\gg|y|$. In case (a) we may further assume that $|x|\leq|x-y|$ (the other case is similar), then $|x-y|\sim|y|$, therefore
\[\sum_{(a)} \mathbf{1}_{|x|\gtrsim M}\cdot \mathbf{1}_{|x-y|\gtrsim\max(M,N)}\cdot\frac{1}{|x|^2|x-y|^2}\lesssim|y|^{-2}\sum_{M\lesssim|x|\lesssim|y|,x\in A}|x|^{-2}\lesssim M^{-4}\] by the same dyadic decomposition in (\ref{eq.lem4_2.5}). Now in case (b), we have $|x|\sim|x-y|\gg|y|$. Let $|x|\sim X\in 2^{\mathbb{N}}$ and $X\gtrsim\max(|y|,M,N)$, then similar to (\ref{eq.lem4_2.5}) we have
\begin{multline}\label{eq.aux10}\sum_{(b)}\frac{1}{|x|^2|x-y|^2}\lesssim\sum_{|x|\gtrsim\max(|y|,M,N),x\in A}\frac{1}{|x|^4}\lesssim\sum_{X\gtrsim \max(|y|,M,N)}X^{-4}\cdot\#\{x\in A:|x|\sim X\}\\\lesssim\sum_{X\gtrsim \max(|y|,M,N)}X^{-4}\cdot\min\bigg(|A|,\big(\frac{X}{M}\big)^4\bigg).\end{multline} In (\ref{eq.aux10}), for $X\geq P^{1/4}=|A|^{1/4}M$ we get contribution $O(M^{-4})$; for $\max(|y|,M,N)\lesssim X\leq P^{1/4}$, we must have $P^{1/4}\gtrsim\max(|y|,M,N)$, and the summation in (\ref{eq.aux10}) gives \[M^{-4}\cdot\log\bigg(2+\frac{P^{1/4}}{\max(|y|,M,N)}\bigg)\lesssim M^{-4}\cdot\bigg(1+\frac{P^{1/4}}{\max(|y|,M,N)}\bigg),\] which is also acceptable. This proves (\ref{eq.lem4_3}).

Finally, with (\ref{eq.lem4_3}) proved, we then sum in $y$. There are three cases depending on the value of $\max(|y|,M,N)$:
\begin{itemize}
\item If $\max(|y|,M,N)=|y|$, then
\begin{multline}\label{eq.aux11}\textrm{LHS of (\ref{eq.aux8})}\lesssim M^{-4}\sum_{|y|\gtrsim N,y\in B}\bigg(1+\frac{P^{1/4}}{|y|}\bigg)\lesssim M^{-4}\bigg(|B|+P^{1/4}\sum_{Y\gtrsim N}Y^{-1}\min(|B|,(Y/N)^4)\bigg)\\\lesssim M^{-4}|B|+M^{-4}P^{1/4}\cdot |B|^{3/4}N^{-1}=|A|^{1/2}M^{-2}\cdot |B|^{1/2}N^{-2}\cdot \big[(Q/P)^{1/2}+(Q/P)^{1/4}\big]
\end{multline} by the same dyadic decomposition in (\ref{eq.lem4_2.5}), which is sufficient as $Q\leq P$.
\item If $\max(|y|,M,N)=N$, then $|y|\leq N$ implies that $y$ has at most $O(1)$ choices, thus
\begin{equation}\textrm{LHS of (\ref{eq.aux8})}\lesssim M^{-4}+M^{-4}P^{1/4}N^{-1}\lesssim (\ref{eq.aux11}).
\end{equation}
\item If $\max(|y|,M,N)=M$, then $y$ has at most $\min((M/N)^4,|B|)\leq M/N\cdot |B|^{3/4}$ choices, thus
\begin{equation}\textrm{LHS of (\ref{eq.aux8})}\lesssim M^{-4}|B|+M^{-5}P^{1/4}\cdot |B|^{3/4}\cdot M/N= (\ref{eq.aux11}).
\end{equation}
\end{itemize}  This completes the proof of (\ref{eq.aux8}).
\end{proof}
\subsubsection{Concluding the proof}\label{sec.final} We can now conclude the proof of Proposition \ref{prop.main_4}. The proof is divided into 4 steps. In the proof below, $C$ will be an absolute constant independent from $C_0$; at the end of the proof we will take $C_0\gg_{C}1$, so the $C_0^{n/2}$ factor in (\ref{eq.local_3}) can absorb any $C^n$ factors.

\textbf{Step 1.} First, recall $N_\nf'$ defined in (\ref{eq.local_2}) and $\gamma_\nf^2$ and $\gamma_\nf^\infty$ defined in (\ref{eq.mul_1}), we have
\begin{equation}\label{eq.chain_1}\sum_{\mf\in\Bc}(\gamma_\mf^\infty-1)=\sum_{\mf}(\gamma_\mf-1)+\sum_{\lf\in\Lc} (m_\lf-1)=r-1+\sum_{\lf}(m_\lf-1)=\bigg(\sum_{\lf}m_\lf\bigg)-1=m-1\end{equation} where $r$ and $m$ are the number of leaves without and with multiplicities, so by (\ref{eq.local_2}) we have
\begin{multline}\label{eq.chain_2}\prod_{\nf\in\Bc\backslash\{\rf\}}\frac{N_{\nf^+}}{N_{\nf}}\leq 8^m\prod_{\nf\in\Bc\backslash\{\rf\}}\sum_{\mf\leq\nf}\frac{\gamma_\mf^\infty N_\mf}{N_\nf}\leq 8^m\prod_{\nf\in\Bc\backslash\{\rf\}}\exp\bigg(\sum_{\mf\leq\nf}\frac{\gamma_\mf^\infty N_\mf}{N_\nf}\bigg)\\
=8^m\exp\bigg(\sum_{\nf\in\Bc\backslash\{\rf\}}\sum_{\mf\leq\nf}\frac{\gamma_\mf^\infty N_\mf}{N_\nf}\bigg)=8^m\exp\bigg(\sum_\mf\gamma_\mf^\infty\sum_{\nf\geq\mf}\frac{N_\mf}{N_\nf}\bigg)\leq 8^m\exp\bigg(\sum_\mf\gamma_\mf^\infty \bigg)\leq C^m\end{multline} using (\ref{eq.chain_1}) and that $\gamma_\mf^\infty \leq 2(\gamma_\mf^\infty -1)$. This allows to take care of the $\prod_{\nf\in\Bc\backslash\{\rf\}}(N_\nf/N_{\nf^+})^{3}$ factor in (\ref{eq.local_3}). It then suffices to show that
\begin{equation}\label{eq.chain_3}
\sum_{\pi}^{(*)}\prod_{\substack{j=1\\j\not\in O}}^{m-1}\frac{1}{\langle y_{\pi(j)}-y_{\pi(j+1)}\rangle^2}\leq C^m((m-s)!)^{1/2}\prod_{\lf\in\Lc^S}(m_\lf!)^{1/2}\prod_{\lf\in\Lc^C}(m_\lf!)^{3/4}\cdot\prod_{\nf\in\Bc} N_\nf^{-2(\gamma_\nf^2 -1)}\cdot R^{2s}\cdot (N_\rf/R)^{\min(2s,3)}.
\end{equation}

In addition, let $\lf_0$ be the leaf such that $N_{\lf_0^+}$ is minimal. We can verify
\begin{equation}\label{eq.chain_3.5}
\prod_{\nf\in\Bc} N_\nf^{\gamma_\nf^2-1}=\prod_{\lf\in\Lc^S}N_{\lf^+}^{2}\prod_{\lf\in \Lc^C}N_{\nf^+}^{m_\lf}\cdot N_{\lf_0^+}^{-1}\cdot\prod_{\nf\in\Bc,\,\lf_0\not\in\Tc_\nf}\frac{N_{\nf^+}}{N_\nf}
\end{equation} again by counting the power of each $N_\nf$ on both sides, which is
\[|\Cc_\nf\cap\Bc|+2|\Cc_\nf\cap\Lc^S|+\sum_{\lf\in\Cc_\nf\cap\Lc^C}m_\lf-1\quad\textrm{for LHS},\]
\[|\Cc_\nf\cap\Bc|+2|\Cc_\nf\cap\Lc^S|+\sum_{\lf\in\Cc_\nf\cap\Lc^C}m_\lf-1+\mathbf{1}_{\lf_0\in\Tc_\nf}-\mathbf{1}_{\nf=\lf_0^+}-|\Cc_\nf\cap\Bc\cap\mathrm{An}(\lf_0)|\quad\textrm{for RHS},\] where $\mathrm{An}(\lf_0)$ is the set of ancestors of $\lf_0$, and it is easy to verify that the above two lines are equal. Because of this, and using also (\ref{eq.chain_2}), we now only need to prove that
\begin{multline}\label{eq.chain_6}
\sum_{\pi}^{(*)}\prod_{\substack{j=1\\j\not\in O}}^{m-1}\frac{1}{\langle y_{\pi(j)}-y_{\pi(j+1)}\rangle^2}\leq C^m((m-s)!)^{1/2}\prod_{\lf\in\Lc^S}(m_\lf!)^{1/2}\prod_{\lf\in\Lc^C}(m_\lf!)^{3/4}\\\times\prod_{\lf\in\Lc^S}N_{\lf^+}^{-4}\prod_{\lf\in\Lc^C}N_{\lf^+}^{-2m_\lf}\cdot N_{\lf_0^+}^2\cdot R^{2s}\cdot (N_\rf/R)^{\min(2s,3)}.
\end{multline}

\textbf{Step 2.} From now on we aim at proving (\ref{eq.chain_6}). Consider all the leaves $\lf\in\Lc$ and the corresponding values $N_{\lf^+}$, and collect all the \emph{different} values $N$ among them. For each $N$, let $\Lc_N$ be the set of leaves $\lf$ where $N_{\lf^+}=N$, and recall $m_\lf$ is the multiplicity. Note that different $\Lc_N$ are disjoint and form a partition of $\Lc$. From Definition \ref{def.hepp_3} (a), note that $(y_j)$ is a permutation of $m_\lf$ copies of $z_\lf$ for all $\lf$, and $(z_\lf)$ satisfies
\begin{equation}\label{eq.chain_8}
|z_{\lf}-z_{\lf'}|\geq\frac{1}{2}N_{\lf\vee\lf'}\geq\frac{1}{2}\max(N,N')
\end{equation} provided $\lf\in \Lc_N$ and $\lf'\in \Lc_{N'}$ and $\lf\neq\lf'$ (allowing $N=N'$).

Note that all the multiplicities $m_\lf\geq 2$. We may assume all leaves $\lf$ with multiplicity $m_\lf=2$ are compound (as this does not affect the right hand side of (\ref{eq.chain_6}) up to another $O(C^m)$ factor). For each leaf $\lf$, there is a unique $N$ such that $\lf\in\Lc_N$, and a unique $X\in 2^\Nb$ such that $m_\lf\in[X,2X)$. For each $(N,X)$, consider all the leaves $\lf$ corresponding to this $(N,X)$ as above, and let $Y\in 2^\Nb$ be such that the number of these leaves lies in $[Y,2Y)$. This defines a mapping $(N,X)\mapsto Y$, and hence the following mapping chain
\begin{equation}\label{eq.chain_9}
\lf\xrightarrow{\sigma_1}(N,X)\xrightarrow{\sigma_2}(N,Y)\xrightarrow{\sigma_3}P,
\end{equation} where $\sigma_1$ and $\sigma_2$ are as above, and $\sigma_3(N,Y)=P:=YN^4$.

Define $\Lc_{N,X}$, $\Lc_{N,Y}$ and $\Lc_P$ to be the set of leaves $\lf$ corresponding to a given $(N,X)$, or given $(N,Y)$, or given $P$, and define $m_{N,X}$ etc. to be the total multiplicity of these leaves. For each $(N,Y)$, consider all different $X$ such that $\sigma_2(N,X)=(N,Y)$; let the maximal of them be $X^*$, which is a function of $(N,Y)$. The following relations are easily verified by the above definitions:
\begin{equation}\label{eq.chain_10}
m_{N,X}\sim XY,\ m_{N,Y}\sim X_*Y,\quad m_{N,Y}=\sum_{\sigma_2(N,X)=(N,Y)}m_{N,X},\quad m_{P}=\sum_{YN^4=P}m_{N,Y},\quad m=\sum_P m_P.
\end{equation}

\textbf{Step 3.} For each permutation $\pi$ and each $j$, let $\lf_j\in\Lc$ be the leaf of which $\pi(j)$ is a copy, and define $(N_j,X_j)$, $(N_j,Y_j)$ and $P_j$ be the results of applying successively $\sigma_1$, $\sigma_2$ and $\sigma_3$ to $\lf_j$. Denote the sequence $(N_j,X_j):=\boldsymbol{X}$, and $(N_j,Y_j):=\vY$ and $(P_j):=\vP$. Define also $\Lc_{N,X}^S=\Lc_{N,X}\cap \Lc^S$ and similarly for $\Lc_{N,X}^C$ and $\Lc_P^S$, $\Lc_P^C$ etc.

To estimate the summation in (\ref{eq.chain_6}), we will first fix $\vX=(N_j,X_j)$ and sum over all $\pi$ with fixed $\vX$, and then sum over $\vX$. This leads to the decomposition
\[\sum_\pi(\cdots)=\sum_{\vX}\sum_{\pi:\vX\,\textrm{fixed}}(\cdots).\] In this step we deal with the outer summation (in $\vX$), and prove that
\begin{equation}\label{eq.chain_11}\sum_{\vX}\prod_{j=2}^{m-1} \min\bigg(1,\bigg(\frac{P_{j+1}}{P_j}\bigg)^{\theta}\bigg)\leq C^m \prod_{\lf\in\Lc^C}(m_\lf!)^{1/4}\prod_{\lf\in\Lc^S}N_{\lf^+}^{2(m_\lf-2)}
\end{equation}
for $\theta:=1/20$; the same is true if we skip at most $99$ values of $j$ in (\ref{eq.chain_11}). To prove (\ref{eq.chain_11}), note that all the elements $P_j$ belong to some fixed set of at most $n$ distinct dyadic numbers, so by (\ref{eq.aux6}) in Lemma \ref{lem.aux2} we have that
\begin{equation}\label{eq.chain_12}\sum_{\vP}\prod_{j=2}^{m-1} \min\bigg(1,\bigg(\frac{P_{j+1}}{P_j}\bigg)^{\theta}\bigg)\leq C^n
\end{equation} (same if we skip at most $99$ values of $j$), thus it suffices to prove that
\begin{equation}\label{eq.chain_13}\sum_{\vX:\vP\,\textrm{fixed}}1\leq C^m \prod_{\lf\in\Lc^C}(m_\lf!)^{1/4}\prod_{\lf\in\Lc^S}N_{\lf^+}^{2(m_\lf-2)}.
\end{equation} Since for $\vP=(P_j)$ fixed, we only need to permute those $(N,X)$ corresponding to the same $P$, we see that
\begin{equation}\textrm{LHS of (\ref{eq.chain_13})}=\prod_{P}\binom{m_P}{(m_{N,X})_{N,X}},
\end{equation} where $(N,X)$ runs over those satisfying $\sigma_3(\sigma_2(N,X))=P$. It then suffice to show that
\begin{equation}\binom{m_P}{(m_{N,X})_{N,X}}\leq C^{m_P}\prod_{\lf\in\Lc_P^C}(m_\lf!)^{1/4}\prod_{\lf\in\Lc_P^S}N_{\lf^+}^{2(m_\lf-2)}.
\end{equation} Using the decomposition (where $YN^4=P$ and $\sigma_2(N,X)=(N,Y)$)
\[\binom{m_P}{(m_{N,X})_{N,X}}=\binom{m_P}{(m_{N,Y})_{N,Y}}\cdot\prod_{(N,Y)}\binom{m_{N,Y}}{(m_{N,X})_{N,X}}\] and using (\ref{eq.aux3}) in Lemma \ref{lem.aux1} to control $\binom{m_{N,Y}}{(m_{N,X})_{N,X}}$ (as $m_{N,X}\sim XY$ are lacunary for a fixed $Y$ and different choices of $X$, see (\ref{eq.chain_10})), it then suffices to prove
\begin{equation}\label{eq.chain_14}\binom{m_P}{(m_{N,Y})_{N,Y}}\leq C^{m_P}\prod_{\lf\in\Lc_P^C}(m_\lf!)^{1/4}\prod_{\lf\in\Lc_P^S}N_{\lf^+}^{2(m_\lf-2)}.
\end{equation}

Let $P$ be fixed, and all different choices of $(N,Y)$ with this fixed $P$ be $(N_{(j)},Y_{(j)})$ where $Y_{(j)}N_{(j)}^4=P$, and let the $X^*$ corresponding to $(N_{(j)},Y_{(j)})$ be $X_{(j)}$, then (\ref{eq.chain_10}) implies that $m_{(j)}:=m_{N_{(j)},Y_{(j)}}\sim X_{(j)}Y_{(j)}$. For each $j$, consider all the leaves in $\Lc_{N_{(j)},X_{(j)}}$; either at least half of them belong to $\Lc^S$, or at least half of them belong to $\Lc^C$. In these cases we write $j\in S$ and $j\in C$ respectively. We then have
\begin{equation}\label{eq.chain_15}\textrm{LHS of (\ref{eq.chain_14})}=\frac{(\sum_j m_{(j)})!}{\prod_j m_{(j)}!}\leq 2^{m_P}\cdot \frac{(\sum_{j\in S}m_{(j)})!}{\prod_{j\in S}m_{(j)}!}\cdot\frac{(\sum_{j\in C}m_{(j)})!}{\prod_{j\in C}m_{(j)}!}.
\end{equation} It then suffices to estimate the last two factors in (\ref{eq.chain_15}); in fact, we shall bound the factor with $j\in S$ by the product $\prod_{\lf\in\Lc_P^S}N_{\lf^+}^{2(m_\lf-2)}$ in (\ref{eq.chain_14}), and the factor $j\in C$ by the product $\prod_{\lf\in\Lc_P^C}(m_\lf!)^{1/4}$.

\textbf{The $j\in S$ part:} Since $Y_{(j)}N_{(j)}^4=P$ are all equal, we know that $N_{(j)}\neq N_{(j')}$ for $j\neq j'$. We may rearrange them in increasing order, then $N_{(j)}\geq 2^j$. For each $j\in S$, at least half (so $\geq Y_{(j)}/2$) leaves in $\Lc_{N_{(j)},X_{(j)}}\subseteq \Lc_{N_{(j)},Y_{(j)}}$ are simple; for such leaves $\lf$ we have $m_\lf>2$ and $2(m_\lf-2)\geq m_\lf/2\geq X_{(j)}/2$, so
\begin{equation}\label{eq.chain_16}\prod_{\lf\in \Lc_{N_{(j)},Y_{(j)}}^S}N_{\lf^+}^{2(m_\lf-2)}\geq N_{(j)}^{X_{(j)}Y_{(j)}/4}\geq 2^{jX_{(j)}Y_{(j)}/4}\quad\Rightarrow\quad\prod_{\lf\in\Lc_P^S}N_{\lf^+}^{2(m_\lf-2)}\geq 2^{\sum_{j\in S}(jX_{(j)}Y_{(j)}/4)}.\end{equation}  As such, by (\ref{eq.aux2}) in Lemma \ref{lem.aux1} we have
\begin{equation}\frac{(\sum_{j\in S}m_{(j)})!}{\prod_{j\in S}m_{(j)}!}\leq \beta^{m_P}\cdot\alpha^{\sum_{j\in S}jm_{(j)}}.
\end{equation} By choosing $\alpha$ sufficiently close to $1$, $\beta=\beta(\alpha)\leq C$, and using $m_{(j)}\sim X_{(j)}Y_{(j)}$ and (\ref{eq.chain_15}), we get the desired bound for the factor with $j\in S$.

\textbf{The $j\in C$ part:} For each $j\in C$, at least half (so $\geq Y_{(j)}/2$) leaves in $\Lc_{N_{(j)},X_{(j)}}$ are compound, so
\begin{equation}\label{eq.chain_17}\prod_{\lf\in \Lc_{N_{(j)},Y_{(j)}}^C}(m_\lf!)^{1/4}\geq C^{-m_{N_{(j)},Y_{(j)}}} X_{(j)}^{X_{(j)}Y_{(j)}/8}\quad\Rightarrow\quad\prod_{\lf\in\Lc_P^C}(m_\lf!)^{1/4}\geq C^{-m_P} \prod_{j\in C}X_{(j)}^{X_{(j)}Y_{(j)}/8}.\end{equation} It then suffices to prove (where we omit the $j\in C$ restriction)
\begin{equation}\label{eq.chain_18}\frac{(\sum_{j}m_{(j)})!}{\prod_{j}m_{(j)}!}\leq C^{\sum_{j}m_{(j)}}\prod_{j} (X_{(j)}^{X_{(j)}})^{Y_{(j)}/8}.
\end{equation} To prove (\ref{eq.chain_18}), we order all $X_{(j)}$ from small to large, let the different values be $Z_{(i)}\geq 2^i$; for each $i$, let \[B_{(i)}:=\{j:X_{(j)}=Z_{(i)}\},\qquad q_{(i)}:=\sum_{j\in B_{(i)}}m_{(j)}.\] Note also that $Y_{(j)}\neq Y_{(j')}$ for $j\neq j'$. Since $m_{(j)}\sim X_{(j)}Y_{(j)}$ and all the $X_{(j)}Y_{(j)}\,(j\in B_{(i)})$ are lacunary for each fixed $i$, we see that
\[\binom{\sum_j m_{(j)}}{(m_{(j)})}=\binom{\sum_i q_{(i)}}{(q_{(i)})}\prod_i\binom{q_{(i)}}{(m_{(j)})_{j\in B_{(i)}}}\leq C^{\sum_jm_{(j)}}\cdot \binom{\sum_i q_{(i)}}{(q_{(i)})}\leq C^{\sum_jm_{(j)}}\cdot\alpha^{\sum_i iq_{(i)}}\] for $\alpha$ sufficiently close to $1$ and $C=O(1)$. Now note that $q_{(i)}\sim Z_{(i)}\sum_{j\in B_{(i)}}Y_{(j)}:=Z_{(i)}W_{(i)}$, we have
\[\exp\bigg(\sum_iiq_{(i)}\bigg)\leq \exp\bigg(100\sum_iiZ_{(i)}W_{(i)}\bigg)\leq\exp\bigg(100\sum_i\log Z_{(i)}\cdot Z_{(i)}W_{(i)}\bigg)= \bigg(\prod_{i}Z_{(i)}^{Z_{(i)}W_{(i)}}\bigg)^{100}=\prod_j (X_{(j)}^{X_{(j)}})^{100Y_{(j)}},\] so (\ref{eq.chain_18}) follows by choosing $\alpha$ close enough to $1$. This finishes the proof of (\ref{eq.chain_11}).

\textbf{Step 4.} For simplicity, we will assume $N_1=N_{\lf_0^+}$ is the smallest scale. In general cases we first locate this smallest scale (say $N_j=N_{\lf_0^+}$) then slightly modify the proof by going left and right separately starting from $j$. In this step we will prove that
\begin{multline}\label{eq.chain_19}
\sum_{\pi:\vX\,\textrm{fixed}}^{(*)}\prod_{\substack{j=1\\j\not\in O}}^{m-1}\frac{1}{\langle y_{\pi(j)}-y_{\pi(j+1)}\rangle^2}\leq C^m\prod_{j=1}^{m-1} X_{j+1}Y_{j+1}^{1/2}N_{j+1}^{-2}\cdot\prod_{j\in O}(X_{j+1}Y_{j+1})^{-1/2}\\\times \prod_{j\,\textrm{even}}^{(\dagger)}\min\big(1,(P_{j+1}/P_j)^{1/8}\big)\cdot R^{2s}\cdot (N_\rf/R)^{\min(2s,3)},
\end{multline} where the condition $(\dagger)$ means we skip at most $20$ values among all even indices $j$.

Once we prove (\ref{eq.chain_19}), the same proof yields the same bound but for $j$ odd instead of $j$ even, then interpolation gives the same bound for all $j$ with at most $40$ exceptions, and with $1/8$ replaced by $1/16$; combining with (\ref{eq.chain_11}) we then get
\begin{multline}\label{eq.chain_20}
\textrm{LHS of (\ref{eq.chain_6})}\leq \sup_{\vX}\bigg(C^m\prod_{N,X}(XY^{1/2})^{m_{NX}}\prod_{N,X}(XY)^{-s_{NX}/2}\prod_{\lf\in\Lc}N_{\lf^+}^{-2m_\lf}\cdot N_{\lf_0^+}^2\\
\times\prod_{\lf\in\Lc^C}(m_\lf!)^{1/4}\prod_{\lf\in\Lc^S} N_{\lf^+}^{2(m_\lf-2)}\cdot R^{2s}\cdot\lambda\bigg),
\end{multline} where $s_{NX}$ is the number of $j\in O$ such that $(N_{j+1},X_{j+1})=(N,X)$. Note that $m_{NX}\sim XY$ and each $m_\lf\sim X$ for $\lf\in \Lc_{N,X}$, we conclude that (using also $A^B\leq e^A\cdot B!$ with $A=m_{NX}$ and $B=m_{NX}-s_{NX}$)
\[(XY^{1/2})^{m_{NX}}\cdot(XY)^{-s_{NX}/2}\leq C^{m_{NX}}((m_{NX}-s_{NX})!)^{1/2}\prod_{\lf\in\Lc_{N,X}}(m_{\lf}!)^{1/2}.\] By multiplying in $(N,X)$ and using that $\sum_{N,X}m_{NX}=m$ and $\sum_{N,X}s_{NX}=s$, we get that
\begin{multline}\label{eq.chain_20.5}
\textrm{LHS of (\ref{eq.chain_6})}\leq C^{m}((m-s)!)^{1/2}\prod_{\lf\in\Lc^S}(m_\lf!)^{1/2}\prod_{\lf\in\Lc^C}(m_\lf!)^{3/4}\\
\times\prod_{\lf\in\Lc^S}N_{\lf^+}^{-4}\prod_{\lf\in\Lc^C}N_{\lf^+}^{-2m_\lf}\cdot N_{\lf_0^+}^2 \cdot R^{2s}\cdot (N_\rf/R)^{\min(2s,3)},
\end{multline} which proves (\ref{eq.chain_6}).

Now it remains to prove (\ref{eq.chain_19}). In summing over $\pi$ we sum over each $\pi(j)$ in decreasing order in $j$ (i.e. first $\pi(n)$, then $\pi(n-1)$, then $\pi(n-2)$ etc.), and also combine the summation in $\pi(j)$ and $\pi(j+1)$ together for each even $j$. By assumption, $\pi(j)$ and $\pi(j+1)$ cannot be two copies of the same leaf unless $j\in O$. We shall prove two local inequalities, the first being rough and second more precise:
\begin{equation}\label{eq.chain_20.9}
\sum_{\pi(j),\pi(j+1)}\Lambda_j\cdot\Lambda_{j+1}\lesssim X_jY_{j}^{1/2}N_j^{-2}\cdot X_{j+1}Y_{j+1}^{1/2}N_{j+1}^{-2}\cdot Y_j^{1/2}Y_{j+1}^{1/2}\cdot (N_j/R)^{2\cdot\mathbf{1}_{j-1\in O}}\cdot (N_{j+1}/R)^{2\cdot\mathbf{1}_{j\in O}},
\end{equation}
\begin{equation}\label{eq.chain_21}\sum_{\pi(j),\pi(j+1)}\Lambda_j\cdot\Lambda_{j+1}\lesssim X_jY_{j}^{1/2}N_j^{-2}\cdot X_{j+1}Y_{j+1}^{1/2}N_{j+1}^{-2}\cdot \Xi_j\Xi_{j+1}\cdot \min\big(1,(P_{j+1}/P_j)^{1/8}\big),
\end{equation} where
\[\Lambda_j=\left\{
\begin{aligned}\mathbf{1}_{|y_{\pi(j-1)}-y_{\pi(j)}|\gtrsim N_j}\cdot |y_{\pi(j-1)}-y_{\pi(j)}|^{-2},&&j-1\not\in O;\\
R^{-2},&&j-1\in O,
\end{aligned}\qquad \textrm{with $\pi(j-1)$ fixed,}
\right.\] same for $\Lambda_{j+1}$ (but with stronger restriction $|y_{\pi(j)}-y_{\pi(j+1)}|\gtrsim \max(N_j,N_{j+1})$ when $j\not\in O$), and \[\Xi_j=(X_jY_j)^{(-1/2)\cdot\mathbf{1}_{j-1\in O}},\qquad\textrm{same for } \Xi_{j+1}.\] Note that, if a summation that cannot be combined (due to parity), we will estimate this summation by the one-variable version of (\ref{eq.chain_20.9}), namely
\begin{equation}\label{eq.chain_21.1}
\sum_{\pi(j)}\Lambda_j\lesssim X_jY_{j}^{1/2}N_j^{-2}\cdot Y_j^{1/2}\cdot (N_j/R)^{2\cdot\mathbf{1}_{j-1\in O}}.
\end{equation}

Once we prove (\ref{eq.chain_20.9})--(\ref{eq.chain_21.1}), we put together these one- and two-variable sum estimates in the order described above, to recover the right hand side of (\ref{eq.chain_19}). More precisely, for the (at most $4$) one-variable summations we apply (\ref{eq.chain_21.1}). For the two-variable summation in $(\pi(j),\pi(j+1))$ where $j-1\not\in O$ and $j\not\in  O$, we apply (\ref{eq.chain_21}); for those where $j-1\in O$ or $j\in O$, we choose 3 of them (or all of them if there are less than 3) and apply (\ref{eq.chain_20.9}) for these chosen 3, and apply (\ref{eq.chain_21}) for the rest. In this way, the $N_j/R$ factors on the right hand sides of (\ref{eq.chain_20.9}) and (\ref{eq.chain_21.1}) provide the needed $(N_\rf/R)^{\min(2s,3)}$ factor, while the applications of (\ref{eq.chain_21}) provide all the other factors on the right hand side of (\ref{eq.chain_19}) (with $R^{2s}$ coming from the $\Lambda_j=R^{-2}$ factor for $j-1\in O$). Note that the (at most 7) applications of (\ref{eq.chain_20.9}) and (\ref{eq.chain_21.1}) may cause us to skip at most $20$ factors $(\min(1,P_{j+1}/P_j))^{1/8}$, but this is allowed by (\ref{eq.chain_19}); for the same reason we may also lose various $Y_j$ and $\Xi_j^{-1}$ factors on the right hand side of (\ref{eq.chain_19}), but again there are at most $20$ of them and each is bounded by $m$, which leads to $m^{20}\leq C^m$ loss. This proves (\ref{eq.chain_19}) assuming (\ref{eq.chain_20.9})--(\ref{eq.chain_21.1}).

Now we prove (\ref{eq.chain_20.9})--(\ref{eq.chain_21.1}). Note that (\ref{eq.chain_20.9}) is trivial because the number of choices for $\pi(j)$ and $\pi(j+1)$ are bounded by $X_jY_j$ and $X_{j+1}Y_{j+1}$ respectively, so (\ref{eq.chain_20.9}) easily follows from the lower bounds $|y_{\pi(j-1)}-y_{\pi(j)}|\gtrsim N_j$ and $|y_{\pi(j)}-y_{\pi(j+1)}|\gtrsim N_{j+1}$. The same holds for (\ref{eq.chain_21.1}).

It remains to prove (\ref{eq.chain_21}). Note that each $\pi(j)$ is a copy of a leaf $\lf$ in $\Lc_{N_j,X_j}$ so that $y_{\pi(j)}=z_\lf$; each leaf has $m_\lf\sim X_j$ copies, and the $z_\lf$ form a set of cardinality $\sim Y_j$ with any two points separated by distance $\gtrsim N_j$ thanks to (\ref{eq.chain_8}). Therefore, the case $j-1\not\in O$ and $j\not\in O$ follows directly from (\ref{eq.aux8}) in Lemma \ref{lem.aux3}.

Next, suppose $j-1\in O$ and $j\in O$. Note that $\pi(j)$ and $\pi(j+1)$ each has $\sim X_jY_j$ and $\sim X_{j+1}Y_{j+1}$ choices, it suffices to show that
\[X_jY_j\cdot X_{j+1}Y_{j+1}\cdot R^{-4}\lesssim X_jY_{j}^{1/2}N_j^{-2}\cdot X_{j+1}Y_{j+1}^{1/2}N_{j+1}^{-2}\cdot (X_jY_jX_{j+1}Y_{j+1})^{-1/2}\cdot \min\big(1,(P_{j+1}/P_j)^{1/8}\big),\] which is just
\begin{equation}\label{eq.chain_22}\bigg(\frac{R}{N_j}\bigg)^2\bigg(\frac{R}{N_{j+1}}\bigg)^2\geq (X_j^{1/2}Y_j)(X_{j+1}^{1/2}Y_{j+1})\max\big(1,(P_j/P_{j+1})^{1/8}\big).
\end{equation} Now note that 
\[R\geq N_\rf'=\sum_{\nf}\gamma_\nf^\infty \cdot N_\nf\geq\sum_{\lf\in\Lc}m_\lf\cdot N_{\lf^+}\geq \sum_{\lf\in \Lc_{N_j,X_j}}m_\lf\cdot N_{\lf^+}\gtrsim N_jX_jY_j,\] so $R/N_j\gtrsim X_jY_j$ (same for $j+1$),
\begin{equation}\label{eq.chain_22.5}\frac{P_j}{P_{j+1}}=\frac{Y_jN_j^4}{Y_{j+1}N_{j+1}^4}\leq Y_j\bigg(\frac{R}{N_{j+1}}\bigg)^4\leq \bigg(\frac{R}{N_{j}}\bigg)\bigg(\frac{R}{N_{j+1}}\bigg)^4,\end{equation} from which (\ref{eq.chain_22}) easily follows.

Next assume of $j-1\not\in O$ and $j\in O$. In this case we first sum over $\pi(j+1)$ trivially and then bound the sum in $\pi(j)$ using (\ref{eq.lem4_2}), to get
\begin{equation}\label{eq.chain_23}\textrm{LHS of (\ref{eq.chain_21})}\lesssim X_jY_j^{1/2}N_j^{-2}\cdot X_{j+1}Y_{j+1}R^{-2},\end{equation} so after some algebraic reductions it suffices to show that
\begin{equation}\label{eq.chain_24}\bigg(\frac{R}{N_{j+1}}\bigg)^2\gtrsim X_{j+1}^{1/2}Y_{j+1}\cdot\max\big(1,(P_j/P_{j+1})^{1/8}\big).
\end{equation} If $N_{j+1}\leq N_j$, then (\ref{eq.chain_24}) easily follows from $R/N_{j+1}\gtrsim X_{j+1}Y_{j+1}$ and (\ref{eq.chain_22.5}); if $N_{j+1}\geq N_j$, then note that
\begin{equation}\label{eq.chain_24.5}\frac{P_j}{P_{j+1}}=\frac{Y_j}{Y_{j+1}}\cdot\frac{N_j^4}{N_{j+1}^4}\leq Y_j\cdot\frac{N_j}{N_{j+1}}\leq\frac{R}{N_{j+1}},\end{equation} so (\ref{eq.chain_24}) still follows easily.

Finally, assume $j-1\in O$ and $j\not\in O$. In this case, we can first sum over either $\pi(j)$ or $\pi(j+1)$ and then bound this sum using (\ref{eq.lem4_2}), and then sum over the other variable trivially. Here if $N_j\geq N_{j+1}$ we choose to sum over $\pi(j)$ first; otherwise we choose to sum over $\pi(j+1)$ first. In this way, it suffices to show that
\begin{equation}\label{eq.chain_25}\bigg(\frac{R}{\min(N_j,N_{j+1})}\bigg)^2\gtrsim (X_jY_j)^{1/2}\cdot\max(Y_j^{1/2},Y_{j+1}^{1/2})\cdot\max\big(1,(P_j/P_{j+1})^{1/8}\big).
\end{equation} But this easily follows from $R/N_{j+1}\gtrsim X_{j+1}Y_{j+1}$, and $R/N_j\gtrsim X_jY_j$, and (\ref{eq.chain_22.5}).

This proves (\ref{eq.chain_21}) and thus (\ref{eq.chain_19}). Using also \textbf{Step 3}, this then proves (\ref{eq.chain_6}), which finishes the proof of Proposition \ref{prop.main_4} in view of \textbf{Steps 1--2}.

\bibliographystyle{alpha}
\bibliography{refs}

\begin{thebibliography}{OSSW24}

\bibitem[BCCH20]{BCCH}
Yvain Bruned, Ajay Chandra, Ilya Chevyrev, and Martin Hairer.
\newblock Renormalising {SPDEs} in regularity structures.
\newblock {\em Journal of the European Mathematical Society}, 23(3):869--947,
  2020.

\bibitem[BGHZ22]{HairerGeometric}
Y.~Bruned, F.~Gabriel, M.~Hairer, and L.~Zambotti.
\newblock Geometric stochastic heat equations.
\newblock {\em J. Amer. Math. Soc.}, 35(1):1--80, 2022.

\bibitem[CC18]{CatellierChouk}
R\'emi Catellier and Khalil Chouk.
\newblock Paracontrolled distributions and the 3-dimensional stochastic
  quantization equation.
\newblock {\em Ann. Probab.}, 46(5):2621--2679, 2018.

\bibitem[CCHS22]{CCHS_2D}
Ajay Chandra, Ilya Chevyrev, Martin Hairer, and Hao Shen.
\newblock Langevin dynamic for the 2{D} {Y}ang-{M}ills measure.
\newblock {\em Publ. Math. Inst. Hautes \'Etudes Sci.}, 136:1--147, 2022.

\bibitem[CCHS24]{CCHS_3D}
Ajay Chandra, Ilya Chevyrev, Martin Hairer, and Hao Shen.
\newblock Stochastic quantisation of {Y}ang-{M}ills-{H}iggs in 3{D}.
\newblock {\em Invent. Math.}, 237(2):541--696, 2024.

\bibitem[CD20]{CD20}
S.~Chatterjee and A.~Dunlap.
\newblock Constructing a solution of the {$(2+1)$}-dimensional {KPZ} equation.
\newblock {\em Ann. Probab.}, 48(2):1014--1055, 2020.

\bibitem[CES21]{CES21}
G.~Cannizzaro, D.~Erhard, and P.~Sch\"onbauer.
\newblock 2{D} anisotropic {KPZ} at stationarity: scaling, tightness and
  nontriviality.
\newblock {\em Ann. Probab.}, 49(1):122--156, 2021.

\bibitem[CET23]{CannizzaroDuke}
Giuseppe Cannizzaro, Dirk Erhard, and Fabio Toninelli.
\newblock Weak coupling limit of the anisotropic {KPZ} equation.
\newblock {\em Duke Mathematical Journal}, 172(16):3013--3104, 2023.

\bibitem[CGT24]{CannizzaroBurgers}
Giuseppe Cannizzaro, Massimiliano Gubinelli, and Fabio Toninelli.
\newblock Gaussian fluctuations for the stochastic {B}urgers equation in
  dimension {$d \geq 2$}.
\newblock {\em Comm. Math. Phys.}, 405(4):Paper No. 89, 2024.

\bibitem[CH16]{ChandraHairer}
Ajay Chandra and Martin Hairer.
\newblock An analytic {BPHZ} theorem for regularity structures.
\newblock {\em arXiv preprint arXiv:1612.08138}, 2016.

\bibitem[CMT24]{CMTsuperdiffusive}
Giuseppe Cannizzaro, Quentin Moulard, and Fabio Toninelli.
\newblock Superdiffusive central limit theorem for the stochastic burgers
  equation at the critical dimension.
\newblock {\em arXiv preprint arXiv:2501.00344}, 2024.

\bibitem[CSZ17]{CSZ17b}
F.~Caravenna, R.~Sun, and N.~Zygouras.
\newblock Universality in marginally relevant disordered systems.
\newblock {\em Ann. Appl. Probab.}, 27(5):3050--3112, 2017.

\bibitem[CSZ19]{CSZ19b}
F.~Caravenna, R.~Sun, and N.~Zygouras.
\newblock On the moments of the {$(2+1)$}-dimensional directed polymer and
  stochastic heat equation in the critical window.
\newblock {\em Comm. Math. Phys.}, 372(2):385--440, 2019.

\bibitem[CSZ20]{CSZ20}
F.~Caravenna, R.~Sun, and N.~Zygouras.
\newblock The two-dimensional {KPZ} equation in the entire subcritical regime.
\newblock {\em Ann. Probab.}, 48(3):1086--1127, 2020.

\bibitem[CSZ23a]{CSZ23a}
F.~Caravenna, R.~Sun, and N.~Zygouras.
\newblock The critical 2d stochastic heat flow.
\newblock {\em Invent. Math.}, 233(1):325--460, 2023.

\bibitem[CSZ23b]{CSZ23b}
F.~Caravenna, R.~Sun, and N.~Zygouras.
\newblock The critical {$2d$} stochastic heat flow is not a {G}aussian
  multiplicative chaos.
\newblock {\em Ann. Probab.}, 51(6):2265--2300, 2023.

\bibitem[CSZ24]{CSZ2024review}
Francesco Caravenna, Rongfeng Sun, and Nikos Zygouras.
\newblock The critical 2d stochastic heat flow and related models.
\newblock {\em arXiv preprint arXiv:2412.10311}, 2024.

\bibitem[CSZ25]{CSZ2025ICM}
Francesco Caravenna, Rongfeng Sun, and Nikos Zygouras.
\newblock From disordered systems to the critical {2D} stochastic heat flow.
\newblock {\em arXiv preprint arXiv:2511.08479}, 2025.

\bibitem[CT24]{CT24}
G.~Cannizzaro and F.~Toninelli.
\newblock Lecture notes on stationary critical and super-critical {SPDE}s,
  2024.
\newblock arXiv:2403.15006.

\bibitem[DG24]{DG24}
A.~Dunlap and C.~Graham.
\newblock Edwards-{W}ilkinson fluctuations in subcritical 2d stochastic heat
  equations, 2024.
\newblock arXiv:2405.09520.

\bibitem[DH23a]{DH21}
Yu~Deng and Zaher Hani.
\newblock Full derivation of the wave kinetic equation.
\newblock {\em Inventiones mathematicae}, 233(2):543--724, 2023.

\bibitem[DH23b]{DH23_2}
Yu~Deng and Zaher Hani.
\newblock Long time justification of wave turbulence theory.
\newblock arXiv:2311.10082, 2023.

\bibitem[DH26a]{DH23}
Yu~Deng and Zaher Hani.
\newblock Derivation of the wave kinetic equation: full range of scaling laws.
\newblock {\em Memoirs of the American Mathematical Society}, 320(1631), 2026.

\bibitem[DH26b]{DH21_2}
Yu~Deng and Zaher Hani.
\newblock Propagation of chaos and the higher order statistics in the wave
  kinetic theory.
\newblock {\em Journal of the European Mathematical Society}, 28(2):673--733,
  2026.

\bibitem[DHM24]{DHM24}
Yu~Deng, Zaher Hani, and Xiaoxuan Ma.
\newblock Long time derivation of the {B}oltzmann equation from hard sphere
  dyamics.
\newblock {\em Annals of Mathematics}, 2024.
\newblock to appear.

\bibitem[DS22]{DubedatShen}
Julien Dub{\'e}dat and Hao Shen.
\newblock Stochastic {R}icci flow on compact surfaces.
\newblock {\em International Mathematics Research Notices},
  2022(16):12253--12301, 2022.

\bibitem[Duc22]{D22}
P.~Duch.
\newblock Flow equation approach to singular stochastic {PDEs}, 2022.
\newblock arXiv:2109.11380.

\bibitem[FG19]{GubiQuasi}
Marco Furlan and Massimiliano Gubinelli.
\newblock Paracontrolled quasilinear {SPDE}s.
\newblock {\em The Annals of Probability}, 47(2):1096--1135, 2019.

\bibitem[GH19]{HairerQuasi}
M{\'a}t{\'e} Gerencs{\'e}r and Martin Hairer.
\newblock A solution theory for quasilinear singular {SPDEs}.
\newblock {\em Communications on Pure and Applied Mathematics},
  72(9):1983--2005, 2019.

\bibitem[GIP15]{GIP15}
M.~Gubinelli, P.~Imkeller, and N.~Perkowski.
\newblock Paracontrolled distributions and singular {PDE}s.
\newblock {\em Forum Math. Pi}, 3:e6, 75, 2015.

\bibitem[GN25]{GN25}
S.~Ganguly and K.~Nam.
\newblock Sharp moment and upper tail asymptotics for the critical $2d$
  {S}tochastic {H}eat {F}low, 2025.
\newblock arXiv:2507.22029.

\bibitem[GQT21]{GQT21}
Y.~Gu, J.~Quastel, and {L}.-{C}. Tsai.
\newblock Moments of the 2{D} {SHE} at criticality.
\newblock {\em Probab. Math. Phys.}, 2(1):179--219, 2021.

\bibitem[GR26]{GR26}
Simon Gabriel and Tommaso Rosati.
\newblock Fluctuations in the weakly coupled {4D Anderson Hamiltonian}.
\newblock {\em arXiv preprint arXiv:2602.22509}, 2026.

\bibitem[GRZ25]{GabrielAC}
Simon Gabriel, Tommaso Rosati, and Nikos Zygouras.
\newblock The {Allen--Cahn} equation with weakly critical random initial datum.
\newblock {\em Probability Theory and Related Fields}, 192(3):1373--1446, 2025.

\bibitem[Gu20]{G20}
Y.~Gu.
\newblock Gaussian fluctuations from the 2{D} {KPZ} equation.
\newblock {\em Stoch. Partial Differ. Equ. Anal. Comput.}, 8(1):150--185, 2020.

\bibitem[Hai13]{HairerKPZ}
Martin Hairer.
\newblock Solving the {KPZ} equation.
\newblock {\em Ann. of Math. (2)}, 178(2):559--664, 2013.

\bibitem[Hai14]{H14}
M.~Hairer.
\newblock A theory of regularity structures.
\newblock {\em Invent. Math.}, 198(2):269--504, 2014.

\bibitem[Hai16]{HairerBPHZ}
Martin Hairer.
\newblock An analyst’s take on the {BPHZ} theorem.
\newblock In {\em The Abel Symposium}, pages 429--476. Springer, 2016.

\bibitem[HQ18]{HairerQuastel}
Martin Hairer and Jeremy Quastel.
\newblock A class of growth models rescaling to {KPZ}.
\newblock {\em Forum of Mathematics, Pi}, 6:e3, 2018.

\bibitem[Kup16]{Kup16}
A.~Kupiainen.
\newblock Renormalization group and stochastic {PDE}s.
\newblock {\em Ann. Henri Poincar\'e}, 17(3):497--535, 2016.

\bibitem[OSSW24]{Otto2024priori}
Felix Otto, Jonas Sauer, Scott~A Smith, and Hendrik Weber.
\newblock A priori bounds for quasi-linear {SPDEs} in the full subcritical
  regime.
\newblock {\em Journal of the European Mathematical Society}, 27(1):71--118,
  2024.

\bibitem[Pia25]{Piazza2025}
Jack Piazza.
\newblock Stochastic conformal flows in even dimensions.
\newblock {\em arXiv preprint arXiv:2506.01217}, 2025.

\bibitem[Tao24]{Tao24}
R.~Tao.
\newblock Gaussian fluctuations of a nonlinear stochastic heat equation in
  dimension two.
\newblock {\em Stoch. Partial Differ. Equ. Anal. Comput.}, 12(1):220--246,
  2024.

\bibitem[Tsa24]{T24}
L.-{C}. Tsai.
\newblock Stochastic heat flow by moments, 2024.
\newblock arXiv:2410.14657.

\end{thebibliography}


\begin{thebibliography}{99}
\bibitem{DH21} Y. Deng and Z. Hani. Full derivation of the wave kinetic equation. \emph{Invent. Math.} 233 (2023), no. 2, 543--724.
\bibitem{DH21_2} Y. Deng and Z. Hani. Propagation of chaos and the higher order statistics in the wave kinetic theory. \emph{J. Eur. Math. Soc.} 28 (2026), no. 2, 673--733.
\bibitem{DH23} Y. Deng and Z. Hani. Derivation of the wave kinetic equation: full range of scaling laws. \emph{Mem. Amer. Math. Soc.} 320 (2026), no. 1631.
\bibitem{DH23_2} Y. Deng and Z. Hani. Long time justification of wave turbulence theory. \emph{arXiv} 2311.10082.
\bibitem{DHM24} Y. Deng, Z. Hani and X. Ma. Long time derivation of the Boltzmann equation from hard sphere dyamics. \emph{Ann. of Math.}, to appear.
\bibitem{CSZ17} Francesco Caravenna, Rongfeng Sun and Nikos Zygouras. Universality in marginally relevant disordered systems. \emph{Ann. Appl. Probab.} 27 (2017), no. 5, 3050--3112.
\bibitem{CSZreview2024} Francesco Caravenna, Rongfeng Sun and Nikos Zygouras. The Critical 2d Stochastic Heat Flow and Related Models. \emph{arXiv} 2412.10311.
\bibitem{GR26} Simon Gabriel and Tommaso Rosati. Fluctuations in the weakly coupled 4D Anderson Hamiltonian. \emph{arXiv} 2602.22509.
\end{thebibliography}

\end{document}